\DeclareRobustCommand{\SkipTocEntry}[4]{}
\makeindex

\documentclass[letterpaper,12pt,reqno,latexsym,amsbsy,xypic,mathrsfs,verbatim]{amsart}
\usepackage{amsfonts,amssymb,amsthm}
\usepackage{amsmath,amscd}
\usepackage{pstricks}
\usepackage{pstricks,pst-node}
\usepackage{mathrsfs}
\usepackage[all]{xy}
\usepackage[enableskew,vcentermath]{youngtab}

\renewcommand{\subsection}[1]{\vspace{.18in}\par\noindent\addtocounter{subsection}{1}\setcounter{equation}{0}{\bf\thesubsection.\hspace{5pt}#1}}

\newtheorem{theorem}{Theorem}[section]
\numberwithin{theorem}{section}
\numberwithin{equation}{theorem}

\theoremstyle{definition}
\newtheorem{defn}[theorem]{Definition}

\newtheorem{rem}[theorem]{Remark}

\theoremstyle{plain}
\newtheorem{prop}[theorem]{Proposition}
\newtheorem{thm}[theorem]{Theorem}
\newtheorem{lem}[theorem]{Lemma}
\newtheorem{cor}[theorem]{Corollary}

\newcommand{\ds}{\displaystyle}

\newcommand{\C}{\mathbb C}

\newcommand{\Z}{\mathbb Z}
\newcommand{\N}{\mathbb N}
\newcommand{\Q}{\mathbb Q}
\newcommand{\Qv}{\mathbb{Q}(v)}

\newcommand{\n}{\mathfrak{n}}

\newcommand{\ep}{\epsilon}

\newcommand{\ol}{\overline}

\newcommand{\mc}{\mathcal}

\newcommand{\refdot}{\boldsymbol{\cdot}}

\def\genE{{\mathsf{E}}}
\def\genF{{\mathsf{F}}}
\def\genK{{\mathsf{K}}}
\def\genXE{{\mathsf{E}}}
\def\genXF{{\mathsf{E}}}

\def\Sn{{\mathfrak S}_n}
\def\qn{{\mathfrak q_n}}

\def\Qvns{{\mathcal{Q}^{s}_v{({n})}}}

\def\Qvnr{{\mathcal{Q}^{s}_v{({n,r})}}}
\def\QqnR{{\mathcal{Q}^{s}_{q}{({n,r;R})}}}
\def\Uqn{{\boldsymbol U}(\qn)}
\def\Uv{{\boldsymbol U}_{\!{v}}}
\def\Uvqn{{\boldsymbol U}_{\!{v}}(\qn)}
\def\Avn{\mathfrak{A}_v(n)}

\newcommand{\qUZ}{{{\boldsymbol U}_{v,\mc Z}}}
\newcommand{\eUZ}{{{\boldsymbol U}_{\epsilon,\Z}}}
\newcommand{\eU}{{{\boldsymbol U}_{\epsilon}}}
\newcommand{\bi}{{\bar i}}
\newcommand{\bj}{{\bar j}}

\newcommand{\mcZ}{{\mathcal{Z}}}

{\vskip-\lastskip\medskip
  \noindent
  {\em #1.}\enspace
  }%
{\qed\par\medskip
  }
\allowdisplaybreaks[1]

\begin{document}

\title[Quantum queer superalgebra and its integral form]{Quantum queer superalgebra and its integral form}
\author{Jianmin Chen, Zhenhua Li, Hongying Zhu}

\address{Jianmin Chen, School of Mathematical Sciences, Xiamen University, Xiamen 361005, China}
\email{chenjianmin@xmu.edu.cn}
\address{Zhenhua Li,  Jiyang,  Jinan 251400, China}
\email{zhen-hua.li@qq.com}
\address{Hongying Zhu, School of Mathematical Sciences, Xiamen University, Xiamen 361005, China}
\email{hongyingz@stu.xmu.edu.cn}


\date{\today}
\begin{abstract} 
In this paper, we introduce quantum root vectors for the quantum queer superalgebra $\Uvqn$ via a braid-group action, compute their complete commutation relations, and construct a PBW-type basis for the Lusztig integral form $\qUZ$. This yields an explicit presentation of $\qUZ$ and provides
a way to understand the structure of the quantum queer superalgebra at roots of unity.
\end{abstract}


\maketitle

\setcounter{tocdepth}{1}  \tableofcontents

\section{Introduction}

Quantum group $\boldsymbol{U}_v(\mathfrak{g})$ is a certain (Hopf algebra) deformation of the universal enveloping algebra $\boldsymbol{U}(\mathfrak{g})$ of a complex semi-simple finite dimensional Lie algebra $\mathfrak{g}$, which was introduced by Drinfeld \cite{Uvsln} and Jimbo \cite{Jimbo} in their work on the quantum Yang–Baxter equation. In order to study representations of the quantum group $\boldsymbol{U}_v(\mathfrak{g})$ especially when $v$ is specialised to a root of unity over arbitrary commutative rings, and relate them to the representation theory of finite groups and modular theory, Lusztig \cite{Lus2} introduced the integral form of $\boldsymbol{U}_v(\mathfrak{g})$. 

Lusztig  firstly replaced the usual generators $E_i$, $F_i$ by their divided powers $E_i^{(r)}=\frac{E_i^r}{[r]^!_v}$, $F_i^{(r)}=\frac{F_i^r}{[r]^!_v}$, which are well-defined elements in $\Z[v,v^{-1}]$, and then introduce the action of braid group associated with $\mathfrak{g}$ to obtain enough integral generators for the entire positive (resp. negative) part. Associated with the action of braid group on $\boldsymbol{U}_v(\mathfrak{g})$, Lusztig defined quantum root vectors $E_{\alpha}$ (and $F_{\alpha}$) and constructed PBW-type basis of the integral form $\boldsymbol{U}_{\Z[v,v^{-1}]}$ by divided powers of these root vectors, which is a quantum generalisation of the classical Poincar{\'e}–Birkhoff–Witt theorem, see \cite{Lus2,Lus4} for more details. 

The Lusztig integral form over $\Z[v, v^{-1}]$ is generated by divided-power root vectors, which provides a concrete integral structure for the quantum group. It allows researchers to investigate its modular representation theory when $v$ is specialised to a root of unity, which is closely related to the representation theory of finite groups. For example, Lusztig \cite{Lus2,Lus5} showed that at a root of unity the quantum algebra differs from the ordinary enveloping algebra only by a finite-dimensional Hopf algebra; Schnitzer \cite{schnizer94} proved that any representation of $\boldsymbol{U}_{\ep}(\mathfrak{g}_{n-1})$ can be lifted to  $\boldsymbol{U}_{\ep}(\mathfrak{g}_{n})$ for $\ep$ a root of unity; and  Chari-Pressley \cite[Chapter 11]{CP} examined the representation theory of $\boldsymbol{U}_v(\mathfrak{g})$ under the same specialisation.

As the theory matured, attention shifted to the expression of Lusztig’s integral form across various quantum groups and related algebras.
Beyond its foundational role in bridging quantum groups and modular representation theory, this integral structure simultaneously controls specialisation to roots of unity and adapts to canonical bases.
A canonical example is integral quantum Schur-Weyl duality: the identification of the double centraliser depends on the surjectivity of the canonical homomorphism from the integral form of $\boldsymbol{U}_v(\mathfrak{gl}_n)$ onto the $v$-Schur algebra $\mathcal{S}(n,r)$. This surjectivity, first outlined in \cite[Section 5]{BLM} and detailed in \cite[(3.4)]{Du95} and \cite[Section 14.6]{DDPW}, is further unified by the quasi-hereditary technology supplied in \cite[3.5]{Du94}, thereby providing a robust integral refinement of the classical Schur-Weyl reciprocity.

Based on the definition of quantum queer superalgebra $\Uvqn$ in \cite{Ol, DW} and the braid group action on $\boldsymbol{U}_v(\mathfrak{sl}_n)$ in \cite{Lus2}, we extended the braid group action on $\Uvqn$ and constructed its quantum root vectors for $\Uvqn$ in \cite{CLZ} . In the present paper, we turn to the Lusztig integral form $\qUZ$ of $\Uvqn$, whose existence over $\Z[v,v^{-1}]$ was established in \cite{DW}. We compute the complete commutation relations of the quantum root vectors for $\Uvqn$ and give an explicit presentation of its Lusztig integral form $\qUZ$. Exploiting the new realization of $\Uvqn$ introduced in \cite{DGLW2}, we construct a PBW-type basis for $\qUZ$ and prove that the natural Superalgebra homomorphism from $\qUZ$ to the standardized queer $v$-Schur superalgebra ${\mathcal{Q}}^{s}_{v,\mcZ}(n,r)$ over $\Z[v,v^{-1}]$, which is crucial for the integral Schur-Weyl duality. We also study the structure of quantum queer superalgebra at roots of unity, laying the groundwork for its subsequent representation theory.

We organize the paper as follows. In Section 2, by recalling the Drinfeld-Jimbo-type presentation of the quantum queer superalgebra and the associated braid group action, we derive the super-commutation formulas for all quantum root vectors; these relations are compatible with the formulas given in \cite{Lus2}. In Section 3, we construct an explicit PBW-type basis for the Lusztig integral form $\qUZ$. Section 4 shifts focus to the canonical surjective homomorphism from $\qUZ$ onto the standardized queer $v$-Schur superalgebra ${\mathcal{Q}}^{s}_{v,\mcZ}(n,r)$. A presentation of $\qUZ$ by generators and relations is given in Section 5. As an application,  we specialize to roots of unity, examine some basic properties for the quantum superalgebra in Section 6.

Throughout this paper, let $\Q$ be the field of rational numbers, and $\Qv$ be the field of
rational functions in one variable $v$ over $\Q$.

\section{Commutation relations for quantum root vectors}\label{root vector}

As we know that, quantum root vectors  are  used in the computation of the PBW-basis (see \cite{Lus2}), the universal $R$-matrix of $\boldsymbol{U}_v(\mathfrak{g})$ (see \cite[Section 8.3]{CP}), and the canonical bases (see \cite{Lus4,Lus5,JCJ}).
In this section, we recall the definition of {$\Uvqn$} given in \cite[Proposition 5.2]{DW}  (also see \cite[Theorem 2.1]{DJ}) and the action of braid group associated to it. We then investigate the quantum root vectors and their commutation formulas.

\begin{defn}\label{defqn}\cite[Proposition 5.2]{DW}
The quantum queer superalgebra $\Uvqn$ is the (Hopf) superalgebra
over $\Qv$  generated by
even generators  {${\genK}_{i}$}, {${\genK}_{i}^{-1}$},  {${\genE}_{j}$},  {${\genF}_{j}$},
and odd generators  {${\genK}_{\ol{i}}$},  {${\genE}_{\ol{j}}$}, {${\genF}_{\ol{j}}$},
with  {$ 1 \le i \le n$}, {$ 1 \le j \le n-1$}, subjecting to the following relations:
\begin{align*}
({\rm QQ1})\quad
&	{\genK}_{i} {\genK}_{i}^{-1} = {\genK}_{i}^{-1} {\genK}_{i} = 1,  \qquad
	{\genK}_{i} {\genK}_{j} = {\genK}_{j} {\genK}_{i} , \qquad
	{\genK}_{i} {\genK}_{\ol{j}} = {\genK}_{\ol{j}} {\genK}_{i}, \\
&	{\genK}_{\ol{i}} {\genK}_{\ol{j}} + {\genK}_{\ol{j}} {\genK}_{\ol{i}}
	= 2 {\delta}_{i,j} \frac{{\genK}_{i}^2 - {\genK}_{i}^{-2}}{{v}^2 - {v}^{-2}}; \\
({\rm QQ2})\quad
& 	{\genK}_{i} {\genE}_{j} = {v}^{(\boldsymbol{\ep}_i, \alpha_j)} {\genE}_{j} {\genK}_{i}, \qquad
	{\genK}_{i} {\genE}_{\ol{j}} = {v}^{(\boldsymbol{\ep}_i, \alpha_j)} {\genE}_{\ol{j}} {\genK}_{i}, \\
& 	{\genK}_{i} {\genF}_{j} = {v}^{-(\boldsymbol{\ep}_i, \alpha_j)} {\genF}_{j} {\genK}_{i}, \qquad
	{\genK}_{i} {\genF}_{\ol{j}} = {v}^{-(\boldsymbol{\ep}_i, \alpha_j)} {\genF}_{\ol{j}} {\genK}_{i}; \\
({\rm QQ3})\quad
& {\genK}_{\ol{i}} {\genE}_{i} - {v} {\genE}_{i} {\genK}_{\ol{i}} = {\genE}_{\ol{i}} {\genK}_{i}^{-1}, \qquad
	{v} {\genK}_{\ol{i}} {\genE}_{i-1} -  {\genE}_{i-1} {\genK}_{\ol{i}} = - {\genK}_{i}^{-1} {\genE}_{\ol{i-1}}, \\
& {\genK}_{\ol{i}} {\genF}_{i} - {v} {\genF}_{i} {\genK}_{\ol{i}} = - {\genF}_{\ol{i}} {\genK}_{i}, \qquad
	{v} {\genK}_{\ol{i}} {\genF}_{i-1} -  {\genF}_{i-1} {\genK}_{\ol{i}} = {\genK}_{i} {\genF}_{\ol{i-1}},\\
& {\genK}_{\ol{i}} {\genE}_{\ol{i}} + {v} {\genE}_{\ol{i}} {\genK}_{\ol{i}} = {\genE}_{i} {\genK}_{i}^{-1}, \qquad
	{v} {\genK}_{\ol{i}} {\genE}_{\ol{i-1}} +  {\genE}_{\ol{i-1}} {\genK}_{\ol{i}} =   {\genK}_{i}^{-1} {\genE}_{i-1}, \\
& {\genK}_{\ol{i}} {\genF}_{\ol{i}} + {v} {\genF}_{\ol{i}} {\genK}_{\ol{i}} =   {\genF}_{i} {\genK}_{i}, \qquad
	{v} {\genK}_{\ol{i}} {\genF}_{\ol{i-1}} +  {\genF}_{\ol{i-1}} {\genK}_{\ol{i}} = {\genK}_{i} {\genF}_{i-1}, \\
&
 {\genK}_{\ol{i}} {\genE}_{j} - {\genE}_{j} {\genK}_{\ol{i}} =  {\genK}_{\ol{i}} {\genF}_{j} - {\genF}_{j} {\genK}_{\ol{i}}
 = {\genK}_{\ol{i}} {\genE}_{\ol{j}} + {\genE}_{\ol{j}} {\genK}_{\ol{i}} =  {\genK}_{\ol{i}} {\genF}_{\ol{j}} + {\genF}_{\ol{j}} {\genK}_{\ol{i}}
	= 0 \mbox{ for } j \ne i, i-1; \\
({\rm QQ4}) \quad
& {\genE}_{i} {\genF}_{j} - {\genF}_{j} {\genE}_{i}
	= \delta_{i,j}  \frac{{\genK}_{i} {\genK}_{i+1}^{-1} - {\genK}_{i}^{-1}{\genK}_{i+1}}{{v} - {v}^{-1}}, \\
&
{\genE}_{\ol{i}} {\genF}_{\ol{j}} + {\genF}_{\ol{j}} {\genE}_{\ol{i}}
	= \delta_{i,j}  ( \frac{{\genK}_{i} {\genK}_{i+1} - {\genK}_{i}^{-1} {\genK}_{i+1}^{-1}}{{v} - {v}^{-1}}
	 + ({v} - {v}^{-1}) {\genK}_{\ol{i}} {\genK}_{\ol{i+1}} )  ,\\
& {\genE}_{i} {\genF}_{\ol{j}} - {\genF}_{\ol{j}} {\genE}_{i}
	= \delta_{i,j}  ( {\genK}_{i+1}^{-1} {\genK}_{\ol{i}}  - {\genK}_{\ol{i+1}} {\genK}_{i}^{-1} ) , \qquad
 {\genE}_{\ol{i}} {\genF}_{j} - {\genF}_{j} {\genE}_{\ol{i}}
	= \delta_{i,j}  ( {\genK}_{i+1} {\genK}_{\ol{i}}  - {\genK}_{\ol{i+1}} {\genK}_{i} ) ;\\
({\rm QQ5}) \quad
&{\genE}_{\ol{i}}^2 = -\frac{ {v} - {v}^{-1} }{{v} + {v}^{-1}} {\genE}_{i}^2, \quad
	{\genF}_{\ol{i}}^2 = \frac{ {v} - {v}^{-1} }{{v} + {v}^{-1}} {\genF}_{i}^2, \\
&
{\genE}_{i} {\genE}_{\ol{j}} - {\genE}_{\ol{j}} {\genE}_{i}
	=  {\genF}_{i} {\genF}_{\ol{j}} - {\genF}_{\ol{j}} {\genF}_{i}
	= 0   \quad \mbox{ for } |i - j| \ne 1,
\\
&
{\genE}_{i} {\genE}_{j} - {\genE}_{j} {\genE}_{i} = {\genF}_{i} {\genF}_{j} - {\genF}_{j} {\genF}_{i}
= {\genE}_{\ol{i}}{\genE}_{\ol{j}}  + {\genE}_{\ol{j}}  {\genE}_{\ol{i}}= {\genF}_{\ol{i}} {\genF}_{\ol{j}}  + {\genF}_{\ol{j}} {\genF}_{\ol{i}}
	= 0 \quad \mbox{ for }|i-j|  > 1, 	
\\
& {\genE}_{i} {\genE}_{i+1} - {v} {\genE}_{i+1} {\genE}_{i}
	= {\genE}_{\ol{i}} {\genE}_{\ol{i+1}} + {v} {\genE}_{\ol{i+1}} {\genE}_{\ol{i}}, \qquad
  {\genE}_{i} {\genE}_{\ol{i+1}} - {v} {\genE}_{\ol{i+1}} {\genE}_{i}
	= {\genE}_{\ol{i}} {\genE}_{i+1} - {v} {\genE}_{i+1} {\genE}_{\ol{i}}, \\
& {\genF}_{i} {\genF}_{i+1} - {v} {\genF}_{i+1} {\genF}_{i}
	= - ({\genF}_{\ol{i}} {\genF}_{\ol{i+1}} + {v} {\genF}_{\ol{i+1}} {\genF}_{\ol{i}}), \qquad
  {\genF}_{i} {\genF}_{\ol{i+1}}  - {v} {\genF}_{\ol{i+1}}  {\genF}_{i}
	=  {\genF}_{\ol{i}} {\genF}_{i+1} - {v} {\genF}_{i+1} {\genF}_{\ol{i}} ;\\
({\rm QQ6}) \quad
& {\genE}_{i}^2 {\genE}_{j} - ( {v} + {v}^{-1} ) {\genE}_{i} {\genE}_{j} {\genE}_{i} + {\genE}_{j}  {\genE}_{i}^2 = 0, \qquad
	{\genF}_{i}^2 {\genF}_{j} - ( {v} + {v}^{-1} ) {\genF}_{i} {\genF}_{j} {\genF}_{i} + {\genF}_{j}  {\genF}_{i}^2 = 0, \\
& {\genE}_{i}^2 {\genE}_{\ol{j}} - ( {v} + {v}^{-1} ) {\genE}_{i} {\genE}_{\ol{j}} {\genE}_{i} + {\genE}_{\ol{j}}  {\genE}_{i}^2 = 0, \qquad
	{\genF}_{i}^2 {\genF}_{\ol{j}} - ( {v} + {v}^{-1} ) {\genF}_{i} {\genF}_{\ol{j}} {\genF}_{i} + {\genF}_{\ol{j}}  {\genF}_{i}^2 = 0,  \\
& \qquad \mbox{ where } \quad |i-j| = 1.
\end{align*}
\end{defn}

\begin{rem}
   Definition \ref{defqn} introduces the $\Z_2$ gradation for $\Uvqn$, which is in agreement with the $\Z_2$ gradation of the universal enveloping algebra $\Uqn$ of $\qn$ in the limit $v\to 1$, where $\Uqn$ is obtained by setting 
   \begin{align*}
       &\lim_{v\to 1}\genE_i=e_i,\quad \lim_{v\to 1}\genF_i=f_i,\quad \lim_{v\to 1}\frac{\genK_i-\genK_i^{-1}}{v-v^{-1}}=h_i,\quad \lim_{v\to 1}\genK_i^{\pm 1}=1,\\
       &\lim_{v\to 1}\genE_{\ol{i}}=e_{\ol{i}},\quad
       \lim_{v\to 1}\genF_{\ol{i}}=f_{\ol{i}},\quad \lim_{v\to 1}\genK_{\ol{i}}=h_{\ol{i}}.
   \end{align*}
\end{rem}

Following \cite[(5.6)]{DW}, there is an anti-involution 
$\Omega$ over {$\Uvqn$} with action given by 
\begin{equation}\label{omega}
\begin{aligned}
	&\Omega (v)  = v^{-1}, \quad
	\Omega(\genE_{{j}} ) =\genF_{{j}}, \quad \Omega ( \genF_{{j}} )=\genE_{{j}}, \quad \Omega ( \genK_{{i}} )=\genK_{{i}}^{-1}, \\
	&\Omega(  \genE_{\ol{j}}) =\genF_{\ol{j}}, \quad \Omega (\genF_{\ol{j}} ) =\genE_{\ol{j}}, 
		\quad \Omega (\genK_{\ol{i}}) =\genK_{\ol{i}},	
\end{aligned}
\end{equation}
where $1\le i \le n$, $1\le j \le n-1 $. 

 Let {$\Uv^{0}$} be the $\Qv$-subalgebra of {$\Uvqn$} generated by $\genK_i^{\pm 1}$ and $\genK_{\ol{i}} $ for $1 \le i \le n $, 
   and let {$\Uv^{+}$} (respectively. {$\Uv^{-}$}) be the $\Qv$-subalgebra of {$\Uvqn$} generated by $\genE_j$ and $ \genE_{\ol{j}} $ (respectively. $\genF_j, \genF_{\ol{j}} $) for $1 \le j \le n-1 $. Referring to \cite[Theorem 6.3]{CLZ}, the braid group action on $\Uvqn$ is defined as follows.

\begin{prop}\cite[Theorem 6.3]{CLZ}\label{action_uvqn}
The  braid group acts by {$\Qv$}-superalgebra automorphisms on {$\Uvqn$}.
More precisely, 
\begin{align*}
&
T_{j} ({\genK}^{\pm 1} _{i}  )  =   {\genK}^{\pm 1} _{s_{j}(i)}\quad
\mbox{ for}\enspace 1 \le j \le n-1,   \quad 1 \le i \le n,\\
&
	 T_{i} ({\genK}_{\ol{i-1}}  )  =   {\genK}_{\ol{i-1}},\quad
	T_{i} ({\genK}_{\ol{i}}  )  =   {\genK}_{\ol{i+1}} ,\quad
	T_{i} ({\genK}_{\ol{i+1}}  )  =  (v-v^{-1}) {\genK}_{\ol{i+1}}\genF_i\genE_i -(v-v^{-1})\genF_i\genE_i {\genK}_{\ol{i+1}} +\genK_{\ol{i}}; \\	
&
T_{i} ({\genE}_{i}) = - {\genF}_{i} {\genK}_{i} {\genK}_{i+1}^{-1} , \quad
T_{i} ({\genF}_{i}) = -  {\genK}_{i}^{-1} {\genK}_{i+1} {\genE}_{i} ,\\
&T_{i} ({\genE}_{\ol{i}}) = -{\genK}_{\ol{i+1}}{\genF}_{i}  {\genK}_{i} +v{\genF}_{i}{\genK}_{\ol{i+1}}  {\genK}_{i}, \quad
	T_{i} ({\genF}_{\ol{i}}) = - {\genK}_{i}^{-1} {\genE}_{i}{\genK}_{\ol{i+1}} +v^{-1} {\genK}_{\ol{i+1}}{\genK}_{i}^{-1} {\genE}_{i};\\
& 
T_{i} ({\genE}_{j}) 
=   -  {\genE}_{i} {\genE}_{j}  +   {v}^{-1} {\genE}_{j} {\genE}_{i}, \quad
T_{i} ({\genF}_{j}) 
=  -   {\genF}_{j} {\genF}_{i} +   {v} {\genF}_{i}  {\genF}_{j}  \quad \mbox{ for} \enspace |i - j| = 1, \\
	& 
	 T_{i} ({\genE}_{\ol{j}}) 
	=   -  {\genE}_{i} {\genE}_{\ol{j}}  +   {v}^{-1} {\genE}_{\ol{j}} {\genE}_{i}, \quad
	T_{i} ({\genF}_{\ol{j}}) 
	= -   {\genF}_{\ol{j}} {\genF}_{i} +   {v} {\genF}_{i}  {\genF}_{\ol{j}} 
	 \quad \mbox{ for} \enspace |i - j| = 1, \\
	&
    T_{i} ({\genE}_{j}) = {\genE}_{j} , \quad
T_{i} ({\genF}_{j}) = {\genF}_{j}, \quad
	T_{i} ({\genK}_{\ol{j}}  ) = {\genK}_{\ol{j}} , \quad
	T_{i} ({\genE}_{\ol{j}}) = {\genE}_{\ol{j}}, \quad
	T_{i} ({\genF}_{\ol{j}}) = {\genF}_{\ol{j}} 
	\quad \mbox{ for}\enspace |i - j| > 1.
\end{align*}
The inverse of $T_{i}$ is given by
\begin{align*}
&
T_{i}^{-1} ({\genK}^{\pm 1} _{i-1}  )  =   {\genK}^{\pm 1} _{i-1},\quad
T_{i}^{-1} ({\genK}^{\pm 1} _{i}  )  =   {\genK}^{\pm 1} _{i+1},\quad
T_{i}^{-1}({\genK}^{\pm 1} _{i+1}  )  =   {\genK}^{\pm 1} _{i}  , \quad
	T_{i}^{-1} ({\genK}_{\ol{i-1}}  )  =   {\genK}_{\ol{i-1}},\\
	&T_{i}^{-1} ({\genK}_{\ol{i+1}}  )  =   {\genK}_{\ol{i}} ,\quad
	T_{i}^{-1}  ({\genK}_{\ol{i}}  )=(v-v^{-1}) \genE_i \genF_i \genK_{\ol{i}} -(v-v^{-1}) \genE_i \genF_i \genK_{\ol{i}} +{\genK}_{\ol{i+1}}; \\	
&
T_{i}^{-1}({\genE}_{i}) =  - {\genK}_{i+1}{\genK}_{i}^{-1}{\genF}_{i}  , \quad
T_{i}^{-1}({\genF}_{i}) = -   {\genE}_{i} {\genK}_{i} {\genK}_{i+1}^{-1},\\
&T_{i}^{-1}  ({\genE}_{\ol{i}}) = -{\genK}_{{i+1}}{\genF}_{i}  {\genK}_{\ol{i}} +v{\genK}_{{i+1}}  {\genK}_{\ol{i}}{\genF}_{i}, \quad
	T_{i}^{-1}  ({\genF}_{\ol{i}}) = -{\genK}_{\ol{i}}{\genE}_{i} {\genK}_{i+1}^{-1}  +v^{-1} {\genE}_{i}{\genK}_{\ol{i}}{\genK}_{i+1}^{-1} ;\\	
& 
T_{i}^{-1}({\genE}_{j}) 
=   -  {\genE}_{j}{\genE}_{i}   +   {v}^{-1}  {\genE}_{i}{\genE}_{j}, \quad
T_{i}^{-1} ({\genF}_{j}) 
= -    {\genF}_{i} {\genF}_{j}+   {v}   {\genF}_{j} {\genF}_{i} \quad \mbox{ for} \enspace |i - j| = 1, \\
	& T_{i}^{-1} ({\genE}_{\ol{j}}) 
	=   -  {\genE}_{\ol{j}} {\genE}_{{i}}  +   {v}^{-1} {\genE}_{{i}} {\genE}_{\ol{j}}, \quad
	T_{i}^{-1} ({\genF}_{\ol{j}}) 
	= -   {\genF}_{{i}} {\genF}_{\ol{j}} +   {v} {\genF}_{\ol{j}}  {\genF}_{{i}}  \quad \mbox{ for} \enspace |i - j| = 1;\\
	&
    T_{i}^{-1} ({\genK}^{\pm 1} _{j}  ) = {\genK}^{\pm 1} _{j} , \quad
T_{i}^{-1} ({\genE}_{j}) = {\genE}_{j}, \quad
T_{i}^{-1} ({\genF}_{j}) = {\genF}_{j},\quad \mbox{ for}  \enspace |i - j| > 1,\\
	&T_{i}^{-1} ({\genK}_{\ol{j}}  ) = {\genK}_{\ol{j}} , \quad
	T_{i}^{-1} ({\genE}_{\ol{j}}) = {\genE}_{\ol{j}}, \quad
	T_{i}^{-1} ({\genF}_{\ol{j}}) = {\genF}_{\ol{j}} 
	\quad \mbox{ for}  \enspace |i - j| > 1.
\end{align*}
\end{prop}

Recall the quantum root vector defined by the action of the braid group in \cite[Section 7]{CLZ}, associated with the reduced expression of the longest element $w_0$ in the symmetric group $\Sn$. With the fixed reduced expression $w_0=(s_1)\refdot(s_2s_1)\refdot(s_3s_2s_1)\refdot\cdots\refdot (s_{n-1}\cdots s_1)$, the following proposition is taken from \cite{CLZ}.

 \begin{prop}\cite[Corollary 7.4]{CLZ}\label{rv}
 Assume that $1\leq i< n$, $i+1<  j< n$. We have
\begin{align*}
     &{\genXE}_{i,i+1}=\genE_i,\quad 
      {\overline {\genXE}_{i,i+1}}=\genE_{\overline{i}},\quad
      {\genXF}_{i+1,i}=\genF_i,\quad 
      {\overline {\genXF}_{i+1,i }}=\genF_{\overline{i}},\\
     	&{\genXE}_{i,j}
     	=-{\genXE}_{i,j-1}{\genXE}_{j-1,j}+v^{-1}{\genXE}_{j-1,j}{\genXE}_{i,j-1} 
     	=[ 
     	   \cdots
     	   [ 
     	     [\genE_i, -\genE_{i+1}]_{v^{-1}}, -\genE_{i+2} 
     	              ]_{v^{-1}}, 
     	               \cdots ,-\genE_{j-1} ]_{v^{-1}},\\
     	&{\ol{\genE}}_{i,j}
     	=-{\genXE}_{i,j-1}{\ol {\genXE}_{j-1,j}}+v^{-1}{\ol {\genXE}_{j-1,j}}{\genXE}_{i,j-1}
     	 =[ 
     	 \cdots
     	    [ 
     	          [\genE_i, -\genE_{i+1}]_{v^{-1}}, -\genE_{i+2} 
     	                    ]_{v^{-1}}, 
     	 \cdots ,-\genE_{\ol{j-1}} ]_{v^{-1}},\\
     	 &{\genXF}_{j,i}
     	 =-{\genXF}_{j,j-1}{\genXF}_{j-1,i}+v{\genXF}_{j-1,i}{\genXF}_{j,j-1}
     	 =[-\genF_{j-1}, 
     	       [-\genF_{j-2}, 
     	          \cdots 
     	           [-\genF_{i+1}, \genF_i]_v 
     	                 ]_v 
     	                   \cdots
     	                                      ]_v,\\
     	  &{\ol{\genXF}}_{j,i}
     	  =-{\ol {\genXF}_{j,j-1}}{\genXF}_{j-1,i}+v{\genXF}_{j-1,i}{\ol {\genXF}_{j,j-1}}
     	 =[-\genF_{\ol{j-1}}, 
     	      [-\genF_{j-2}, 
     	        \cdots 
     	          [-\genF_{i+1}, \genF_i]_v 
     	                    ]_v 
     	                       \cdots
     	                             ]_v,
\end{align*}  
where $[a,b]_{x}=ab -xba$ with $a,b \in \Uvqn , x \in \Qv$.
\end{prop}

\begin{cor}\label{rv_k}
    The following hold for  $1\leq i<k<j<n$:
    \begin{align*}
&{\genXE}_{i,j}=-{\genXE}_{i,k}{\genXE}_{k,j}+v^{-1}{\genXE}_{k,j}{\genXE}_{i,k},\quad {\ol{\genE}}_{i,j}=-{\genXE}_{i,k}{\ol {\genXE}_{k,j}}+v^{-1}{\ol {\genXE}_{k,j}}{\genXE}_{i,k},\\
&{\genXF}_{j,i}=-{\genXF}_{j,k}{\genXF}_{k,i}+v{\genXF}_{j-1,i}{\genXF}_{j,k},\quad {\ol{\genXF}}_{j,i}=-{\ol {\genXF}_{j,k}}{\genXF}_{k,i}+v{\genXF}_{k,i}{\ol {\genXF}_{j,k}}.
    \end{align*}
\end{cor}
\begin{proof}
    Suppose $1\leq i<k<j<n$. We prove the second formula by induction on $k$. Indeed, if $k=j-2$, then
{\small\begin{align*}
    {\ol{\genE}}_{i,j}&=-{\genXE}_{i,j-1}{\ol {\genXE}_{j-1,j}}+v^{-1}{\ol {\genXE}_{j-1,j}}{\genXE}_{i,j-1}\\
    &=-(-{\genXE}_{i,j-2}{\genXE}_{j-2,j-1}+v^{-1}{\genXE}_{j-2,j-1}{\genXE}_{i,j-2}){\ol {\genXE}_{j-1,j}}+v^{-1}{\ol {\genXE}_{j-1,j}}(-{\genXE}_{i,j-2}{\genXE}_{j-2,j-1}+v^{-1}{\genXE}_{j-2,j-1}{\genXE}_{i,j-2})\\
    &={\genXE}_{i,j-2}{\genXE}_{j-2,j-1}{\ol {\genXE}_{j-1,j}}-v^{-1}{\genXE}_{j-2,j-1}{\genXE}_{i,j-2}{\ol {\genXE}_{j-1,j}}
    -v^{-1}{\ol {\genXE}_{j-1,j}}{\genXE}_{i,j-2}{\genXE}_{j-2,j-1}+v^{-2}{\ol {\genXE}_{j-1,j}}{\genXE}_{j-2,j-1}{\genXE}_{i,j-2}\\
    &={\genXE}_{i,j-2}{\genXE}_{j-2,j-1}{\ol {\genXE}_{j-1,j}}-v^{-1}{\genXE}_{j-2,j-1}{\ol {\genXE}_{j-1,j}}{\genXE}_{i,j-2}
    -v^{-1}{\genXE}_{i,j-2}{\ol {\genXE}_{j-1,j}}{\genXE}_{j-2,j-1}+v^{-2}{\ol {\genXE}_{j-1,j}}{\genXE}_{j-2,j-1}{\genXE}_{i,j-2}\\
    &=-{\genXE}_{i,j-2}{\ol {\genXE}_{j-2,j}}+v^{-1}{\ol {\genXE}_{j-2,j}}{\genXE}_{i,j-2},
\end{align*}}
by Proposition \ref{rv}. Now assume $i<k<j-2$, then by induction, we have 
{\small\begin{align*}
    {\ol{\genE}}_{i,j}&=-{\genXE}_{i,k+1}{\ol {\genXE}_{k+1,j}}+v^{-1}{\ol {\genXE}_{k+1,j}}{\genXE}_{i,k+1}\\
    &=-(-{\genXE}_{i,k}{\genXE}_{k,k+1}+v^{-1}{\genXE}_{k,k+1}{\genXE}_{i,k}){\ol {\genXE}_{k+1,j}}+v^{-1}{\ol {\genXE}_{k+1,j}}(-{\genXE}_{i,k}{\genXE}_{k,k+1}+v^{-1}{\genXE}_{k,k+1}{\genXE}_{i,k})\\
    &={\genXE}_{i,k}{\genXE}_{k,k+1}{\ol {\genXE}_{k+1,j}}-v^{-1}{\genXE}_{k,k+1}{\genXE}_{i,k}{\ol {\genXE}_{k+1,j}}-v^{-1}{\ol {\genXE}_{k+1,j}}{\genXE}_{i,k}{\genXE}_{k,k+1}+v^{-2}{\ol {\genXE}_{k+1,j}}{\genXE}_{k,k+1}{\genXE}_{i,k}\\
    &={\genXE}_{i,k}{\genXE}_{k,k+1}{\ol {\genXE}_{k+1,j}}-v^{-1}{\genXE}_{k,k+1}{\ol {\genXE}_{k+1,j}}{\genXE}_{i,k}-v^{-1}{\genXE}_{i,k}{\ol {\genXE}_{k+1,j}}{\genXE}_{k,k+1}+v^{-2}{\ol {\genXE}_{k+1,j}}{\genXE}_{k,k+1}{\genXE}_{i,k}\\
    &=-{\genXE}_{i,k}{\ol {\genXE}_{k,j}}+v^{-1}{\ol {\genXE}_{k,j}}{\genXE}_{i,k}
\end{align*}}
as desired. By a parallel argument, we can prove the remaining three formulas.
\end{proof}

\begin{rem}\label{EX}
A direct calculation shows that 
\begin{align*}
  &{\genXE}_{i,j}=T_i T_{i+1}\cdots T_{j-2}(\genE_{j-1}),\quad   {\genXE}_{j,i}=T_i T_{i+1}\cdots T_{j-2}(\genF_{j-1}),\\
  &{\ol{\genXE}}_{i,j}=T_i T_{i+1}\cdots T_{j-2}(\genE_{\ol{j-1}}),\quad   {\ol{\genXE}}_{j,i}=T_i T_{i+1}\cdots T_{j-2}(\genF_{\ol{j-1}}).
\end{align*}
Recall the root vectors defined in \cite[Section 5]{DW} (by replacing quantum parameter $q$ with $v$):
for $1\leq i\leq n-1$,  define
$$
X_{i,i+1}=\genE_i,\quad X_{i+1,i}=\genF_i, \quad  \ol{X}_{i,i+1}=\genE_{\ol{i}},\quad  \ol{X}_{i+1,i}=\genF_{\ol{i}};
$$
for $|j-i|>1$,  define
\begin{equation}\label{q-root}
\aligned
	X_{i,j}&:=\left\{
\begin{array}{ll}
X_{i,k}X_{k,j} - {v} X_{k,j}X_{i,k}&\text{ if }i<j,\\
X_{i,k} X_{k,j} - {v}^{-1} X_{k,j}X_{i,k}&\text{ if }i>j,
\end{array}
\right.\\
	\ol{X}_{i,j} &:=\left\{
\begin{array}{ll}
X_{i,k}  \ol{X}_{k,j}-{v}  \ol{X}_{k,j}X_{i,k}&\text{ if }i<j,\\
 \ol{X}_{i,k} X_{k,j}-{v}^{-1}X_{k,j}  \ol{X}_{i,k}&\text{ if }i>j,
\end{array}
\right.
\endaligned
\end{equation}
{where $k$ is strictly between $i$ and $j$, and the expressions are independent of the choice of $k$.}
Referring to Proposition \ref{rv}, 
{it is seen that $X_{i,j}$, $\ol{X}_{i,j}$ differs ${\genXE}_{i, j}$, $ \ol{\genXE}_{i, j}$ by replacing  $v$ with {${v}^{-1}$} and multiplying $-1$.}
\end{rem}

 As a natural consequence, we provide commutation formulas for all quantum root vectors of $\Uvqn$.

 \begin{lem}\label{KE}
 Assume that $1\leq i,j,a\leq n$ satisfies $i<j$. Then
     \begin{align*}
\aligned
	\genK_{\ol{a}}{\genXE}_{i,j}&=\left\{
\begin{array}{ll}
{\genXE}_{i,j}\genK_{\ol{a}} &(a<i \text{ or } i<a<j \text{ or } a>j ),\\
v{\genXE}_{i,j}\genK_{\ol{i}}-\genE_{\ol{i}}{\genXE}_{i+1,j}\genK_i^{-1}+v^{-1} {\genXE}_{i+1,j}\genE_{\ol{i}}\genK_i^{-1}&(a=i),\\
v^{-1}{\genXE}_{i,j}\genK_{\ol{j}}-{\ol{\genE}}_{i,j}\genK_j^{-1} &(a=j);
\end{array}
\right.\\
\genK_{\ol{a}}{\ol{\genXE}}_{i,j}&=\left\{
\begin{array}{ll}
-{\ol{\genXE}}_{i,j}\genK_{\ol{a}} &(a<i \text{ or } i<a<j \text{ or } a>j ),\\
-v{\ol{\genXE}}_{i,j}\genK_{\ol{i}}-\genE_{\ol{i}}{\ol{\genXE}}_{i+1,j}\genK_i^{-1}-v^{-1} {\ol{\genXE}}_{i+1,j}\genE_{\ol{i}}\genK_i^{-1} &(a=i),\\
-v^{-1}{\ol{\genXE}}_{i,j}\genK_{\ol{j}}+{{\genXE}}_{i,j}\genK_j^{-1} &(a=j).
\end{array}
\right.
\endaligned
\end{align*}
 \end{lem}
 \begin{proof}
    Suppose $1\leq i,j,a\leq n$ and $i<j$. 
    By the last formula in (QQ3), if $a<i$ or $a>j$, it is easy to check that $\genK_{\ol{a}}{\genXE}_{i,j}={\genXE}_{i,j}\genK_{\ol{a}}$.
    Using the relations in (QQ3) and the formulas in Corollary \ref{rv_k}, a straightforward calculation shows that
    \begin{align*}
\genK_{\ol{i}}{\genXE}_{i,j}
&=\genK_{\ol{i}}(-\genE_i {\genXE}_{i+1,j}+v^{-1}{\genXE}_{i+1,j}\genE_i)\\
&=-\genK_{\ol{i}}\genE_i {\genXE}_{i+1,j}+v^{-1}{\genXE}_{i+1,j}\genK_{\ol{i}}\genE_i\\
&=-(v\genE_i\genK_{\ol{i}}+\genE_{\ol{i}}\genK_i^{-1}) {\genXE}_{i+1,j}+v^{-1}{\genXE}_{i+1,j}(v\genE_i\genK_{\ol{i}}+\genE_{\ol{i}}\genK_i^{-1})\\
&=-v\genE_i\genK_{\ol{i}}{\genXE}_{i+1,j}-\genE_{\ol{i}}\genK_i^{-1}{\genXE}_{i+1,j}+{\genXE}_{i+1,j}\genE_i\genK_{\ol{i}}+v^{-1}{\genXE}_{i+1,j}\genE_{\ol{i}}\genK_i^{-1}\\
&=-v\genE_i{\genXE}_{i+1,j}\genK_{\ol{i}}+{\genXE}_{i+1,j}\genE_i\genK_{\ol{i}}-\genE_{\ol{i}}\genK_i^{-1}{\genXE}_{i+1,j}+v^{-1}{\genXE}_{i+1,j}\genE_{\ol{i}}\genK_i^{-1}\\
&=v{\genXE}_{i,j}\genK_{\ol{i}}-\genE_{\ol{i}}{\genXE}_{i+1,j}\genK_i^{-1}+v^{-1}{\genXE}_{i+1,j}\genE_{\ol{i}}\genK_i^{-1},\\
\genK_{\ol{j}}{\genXE}_{i,j}
&=\genK_{\ol{j}}(-{\genXE}_{i,j-1}\genE_{j-1}+v^{-1}\genE_{j-1}{\genXE}_{i,j-1})\\
&=-{\genXE}_{i,j-1}\genK_{\ol{j}}\genE_{j-1}+v^{-1}\genK_{\ol{j}}\genE_{j-1}{\genXE}_{i,j-1}\\
&=-{\genXE}_{i,j-1}(v^{-1}\genE_{j-1}\genK_{\ol{j}}-v^{-1}\genK_j^{-1}\genE_{\ol{j-1}})+v^{-1}(v^{-1}\genE_{j-1}\genK_{\ol{j}}-v^{-1}\genK_j^{-1}\genE_{\ol{j-1}}){\genXE}_{i,j-1}\\
&=-v^{-1}{\genXE}_{i,j-1}\genE_{j-1}\genK_{\ol{j}}+v^{-1}{\genXE}_{i,j-1}\genK_j^{-1}\genE_{\ol{j-1}}+v^{-2}\genE_{j-1}{\genXE}_{i,j-1}\genK_{\ol{j}}-v^{-2}\genK_j^{-1}\genE_{\ol{j-1}}{\genXE}_{i,j-1}\\
&=v^{-1}{\genXE}_{i,j}\genK_{\ol{j}}+{\genXE}_{i,j-1}\genK_{\ol{j-1}}\genK_j^{-1}-v^{-1}\genE_{\ol{j-1}}{\genXE}_{i,j-1}\genK_j^{-1}\\
&=v^{-1}{\genXE}_{i,j}\genK_{\ol{j}}-{\ol{\genXE}}_{i,j}\genK_j^{-1},\\
\genK_{\ol{a}}{\genXE}_{a-1,a+1}
&=\genK_{\ol{a}}(-\genE_{a-1}\genE_{a}+v^{-1}\genE_{a}\genE_{a-1})\\
&=-\genK_{\ol{a}}\genE_{a-1}\genE_{a}+v^{-1}\genK_{\ol{a}}\genE_{a}\genE_{a-1}\\
&=-(v^{-1}\genE_{a-1}\genK_{\ol{a}}-v^{-1}\genK_a^{-1}\genE_{\ol{a-1}})\genE_{a}+v^{-1}(v\genE_{a}\genK_{\ol{a}}+\genE_{\ol{a}}\genK_a^{-1})\genE_{a-1}\\
&=-v^{-1}\genE_{a-1}(v\genE_{a}\genK_{\ol{a}}+\genE_{\ol{a}}\genK_a^{-1})+v^{-1}\genK_a^{-1}\genE_{\ol{a-1}}\genE_{a}\\
&\qquad +\genE_{a}(v^{-1}\genE_{a-1}\genK_{\ol{a}}-v^{-1}\genK_a^{-1}\genE_{\ol{a-1}})+v^{-1}\genE_{\ol{a}}\genK_a^{-1}\genE_{a-1}\\
&=-\genE_{a-1}\genE_{a}\genK_{\ol{a}}-v^{-1}\genE_{a-1}\genE_{\ol{a}}\genK_a^{-1}+v^{-1}\genE_{\ol{a-1}}\genE_{a}\genK_a^{-1}\\
&\qquad+v^{-1}\genE_{a}\genE_{a-1}\genK_{\ol{a}}-\genE_{a}\genE_{\ol{a-1}}\genK_a^{-1}+\genE_{\ol{a}}\genE_{a-1}\genK_a^{-1}\\
&={\genXE}_{a-1,a+1}\genK_{\ol{a}}.
 \end{align*}
For the case $i<a<j$, we have 
\begin{align*}
    \genK_{\ol{a}}{\genXE}_{i,j}
&=\genK_{\ol{a}}(-{\genXE}_{i,a-1}{\genXE}_{a-1,j}+v^{-1}{\genXE}_{a-1,j}{\genXE}_{i,a-1})\\
&=-{\genXE}_{i,a-1}\genK_{\ol{a}}(-{\genXE}_{a-1,a+1}{\genXE}_{a+1,j}+v^{-1}{\genXE}_{a+1,j}{\genXE}_{a-1,a+1})\\
&\qquad+v^{-1}\genK_{\ol{a}}(-{\genXE}_{a-1,a+1}{\genXE}_{a+1,j}+v^{-1}{\genXE}_{a+1,j}{\genXE}_{a-1,a+1}){\genXE}_{i,a-1}\\
&={\genXE}_{i,a-1}\genK_{\ol{a}}{\genXE}_{a-1,a+1}{\genXE}_{a+1,j}-v^{-1}{\genXE}_{i,a-1}{\genXE}_{a+1,j}\genK_{\ol{a}}{\genXE}_{a-1,a+1}\\
&\qquad-v^{-1}\genK_{\ol{a}}{\genXE}_{a-1,a+1}{\genXE}_{a+1,j}{\genXE}_{i,a-1}+v^{-2}{\genXE}_{a+1,j}\genK_{\ol{a}}{\genXE}_{a-1,a+1}{\genXE}_{i,a-1}\\
&={\genXE}_{i,j}\genK_{\ol{a}}.
\end{align*}
Similarly, the formulas of $ \genK_{\ol{a}}{\ol{\genXE}}_{i,j}$ can be verified  and we skip the details. 

 \end{proof}
 
Applying the anti-automorphism $\Omega$ to $ \genK_{\ol{a}}{\ol{\genXE}}_{i,j}$ and $\genK_{\ol{a}}{\genXE}_{i,j}$, one can obtain the commutation formulas for $\genK_{\ol{a}}{{\genXF}}_{j,i}$ and $\genK_{\ol{a}}{\ol{\genXF}}_{j,i}$. We now discuss the commutation formulas for even-even positive quantum root vectors.
 
 \begin{lem}\label{pe-pe1}
 The following holds for $1\leq i,j,k,l\leq n$ satisfying $i<j,k<l,i\leq k$:
 \begin{align*}
	{\genXE}_{i,j}{\genXE}_{k,l}=\left\{
\begin{array}{ll}
{\genXE}_{k,l}{\genXE}_{i,j} &(i<j<k<l \text{ or }i<k<l<j),\\
v^{-1}{\genXE}_{k,l}{\genXE}_{i,j}-{\genXE}_{i,l} &(i<j=k<l),\\
v{\genXE}_{k,l}{\genXE}_{i,j} &(i=k<j<l \text{ or } i<k<j=l ),\\
{\genXE}_{k,l}{\genXE}_{i,j}+(v-v^{-1}){\genXE}_{i,l}{\genXE}_{k,j} &(i<k<j<l).
\end{array}
\right.
\end{align*}
 \end{lem}
  \begin{proof}
     Suppose $1\leq i,j,k,l\leq n$ and $i<j,k<l,i\leq k$. By the relation $\genE_i \genE_j =\genE_j\genE_i$ ($|i-j|>1$) in (QQ5), if $i<j<k<l$, it is easy to check that ${\genXE}_{i,j}{\genXE}_{k,l}={\genXE}_{k,l}{\genXE}_{i,j}$. Applying \cite[Remark 6.5]{CLZ} and Remark \ref{EX}, we observe that
     $$
     \genE_a \genE_{a-1,a+2}=T_{a-1}T_a (\genE_{a-1})T_{a-1}T_a (\genE_{a+1})=T_{a-1}T_a (\genE_{a+1}\genE_{a-1})= \genE_{a-1,a+2}\genE_a.$$
     Also, using  the formulas in Corollary \ref{rv_k} and the relations in (QQ5-QQ6), a straightforward calculation shows that
     \begin{equation}\label{Epe}
     \begin{aligned}
   {\genE}_{a}{\genXE}_{i,j}=\left\{
\begin{array}{ll}
{\genXE}_{i,j}{\genE}_{a} &\hspace{-3em}(a<i-1 \text{ or } i<a<j-1 \text{ or } a>j),\\
v^{-1}{\genXE}_{i,j}{\genE}_{i-1}-{\genXE}_{i-1,j} &(a=i-1),\\
v{\genXE}_{i,j}{\genE}_{i} &(a=i ),\\
v^{-1}{\genXE}_{i,j}{\genE}_{j-1} &(a=j-1),\\
v{\genXE}_{i,j}{\genE}_{j}+v {\genXE}_{i,j+1} &(a=j).
\end{array}
\right.      
 \end{aligned}
\end{equation}
 Then the lemma is proved case-by-case using the first equation in Corollary \ref{rv_k} and the above formulas. Obviously, ${\genXE}_{i,j}{\genXE}_{k,l}={\genXE}_{k,l}{\genXE}_{i,j}$ for $i<k<l<j$.
 
 If $i<j=k<l$, we have
\begin{align*}
  {\genXE}_{i,j}{\genXE}_{j,l}
  &={\genXE}_{i,j}(-{\genE}_{j}{\genXE}_{j+1,l}+v^{-1}{\genXE}_{j+1,l}{\genE}_{j})\\
  &=-{\genXE}_{i,j}{\genE}_{j}{\genXE}_{j+1,l}+v^{-1}{\genXE}_{i,j}{\genXE}_{j+1,l}{\genE}_{j}\\
  &=-(v^{-1}{\genE}_{j}{\genXE}_{i,j}-{\genXE}_{i,j+1}){\genXE}_{j+1,l}+v^{-1}{\genXE}_{j+1,l}(v^{-1}{\genE}_{j}{\genXE}_{i,j}-{\genXE}_{i,j+1})\\
  &=-v^{-1}{\genE}_{j}{\genXE}_{j+1,l}{\genXE}_{i,j}+{\genXE}_{i,j+1}{\genXE}_{j+1,l}+v^{-2}{\genXE}_{j+1,l}{\genE}_{j}{\genXE}_{i,j}-v^{-1}{\genXE}_{j+1,l}{\genXE}_{i,j+1}\\
  &=v^{-1}{\genXE}_{j,l}{\genXE}_{i,j}-{\genXE}_{i,l}.
 \end{align*}
Similarly, we have ${\genXE}_{i,j}{\genXE}_{k,j}=v{\genXE}_{k,j}{\genXE}_{i,j}$. 

Observe that  the case $i=k<j<l$ can be obtained by the case $i<k<l<j$. And also the case $i<k<j<l$ is deduced from the commutation formula of ${\genXE}_{i,j}{\genXE}_{j,l}$ and the case $i=k<j<l$. We now verify the case $i<k<j<l$ in detail, and the proof for the other cases is similar.
\begin{align*}
{\genXE}_{i,j}{\genXE}_{k,l}
&=(-{\genXE}_{i,k}{\genXE}_{k,j}+v^{-1}{\genXE}_{k,j}{\genXE}_{i,k}){\genXE}_{k,l}\\
&=-{\genXE}_{i,k}{\genXE}_{k,j}{\genXE}_{k,l}+v^{-1}{\genXE}_{k,j}{\genXE}_{i,k}{\genXE}_{k,l}\\
&=-v{\genXE}_{i,k}{\genXE}_{k,l}{\genXE}_{k,j}+v^{-1}{\genXE}_{k,j}(v^{-1}{\genXE}_{k,l}{\genXE}_{i,k}-{\genXE}_{i,l})\\
&=-v(v^{-1}{\genXE}_{k,l}{\genXE}_{i,k}-{\genXE}_{i,l}){\genXE}_{k,j}+v^{-1}{\genXE}_{k,l}{\genXE}_{k,j}{\genXE}_{i,k}-v^{-1}{\genXE}_{k,j}{\genXE}_{i,l}\\
&=-{\genXE}_{k,l}{\genXE}_{i,k}{\genXE}_{k,j}+v{\genXE}_{i,l}{\genXE}_{k,j}+v^{-1}{\genXE}_{k,l}{\genXE}_{k,j}{\genXE}_{i,k}-v^{-1}{\genXE}_{k,j}{\genXE}_{i,l}\\
&={\genXE}_{k,l}{\genXE}_{i,j}+v{\genXE}_{i,l}{\genXE}_{k,j}-v^{-1}{\genXE}_{k,j}{\genXE}_{i,l}\\
&={\genXE}_{k,l}{\genXE}_{i,j}+(v-v^{-1}){\genXE}_{i,l}{\genXE}_{k,j}.
 \end{align*}
 \end{proof}

 Observe that if we interchange $(i,j)$ and $(k,l)$, we can obtain another set of commutation formulas for  ${\genXE}_{i,j}{\genXE}_{k,l}$. More precisely, the following holds for $1\leq i,j,k,l\leq n$ satisfying $i<j,k<l,k\leq i$:
 \begin{equation}\label{pe-pe2}
\begin{aligned}
	{\genXE}_{i,j}{\genXE}_{k,l}=\left\{
\begin{array}{ll}
{\genXE}_{k,l}{\genXE}_{i,j} &(k<l<i<j \text{ or } k<i<j<l),\\
v{\genXE}_{k,l}{\genXE}_{i,j}+v{\genXE}_{k,j} &(k<i=l<j),\\
v^{-1}{\genXE}_{k,l}{\genXE}_{i,j} &(k=i<l<j \text{ or } k<i<j=l ),\\
{\genXE}_{k,l}{\genXE}_{i,j}-(v-v^{-1}){\genXE}_{k,j}{\genXE}_{i,l} &(k<i<l<j).
\end{array}
\right.
\end{aligned}
\end{equation}
  Then we give complete commutation formulas for even positive root vectors ${\genXE}_{i,j}$ and ${\genXE}_{k,l}$ with $i<j, k<l$. And also, applying $\Omega$ to Lemma \ref{pe-pe1} and \eqref{pe-pe2}, we can obtain commutation formulas for even negative root vectors  ${\genXE}_{j,i}$ and ${\genXE}_{l,k}$ with $i<j, k<l$. 

\begin{lem}\label{ne-pe1}
     The following holds for $1\leq i,j,k,l\leq n$ satisfying $i<j,k<l,i\leq k$:
\begin{align*}
	{\genXE}_{j,i}{\genXE}_{k,l}=\left\{
\begin{array}{ll}
{\genXE}_{k,l}{\genXE}_{j,i} &\hspace{-3em}(i<j\leq k<l \text{ or }i<k<l<j),\\
{\genXE}_{k,l}{\genXE}_{j,i}-{\genXE}_{j,l}\genK_j \genK_i^{-1} &(i=k<j<l),\\
{\genXE}_{k,l}{\genXE}_{j,i}+{\genXE}_{k,i}\genK_j \genK_k^{-1}
&(i<k<j=l ),\\
{\genXE}_{k,l}{\genXE}_{j,i}-(v-v^{-1}){\genXE}_{k,i}{\genXE}_{j,l}\genK_j \genK_k^{-1} &(i<k<j<l),\\
{\genXE}_{k,l}{\genXE}_{j,i}+\frac{\genK_j \genK_i^{-1}-\genK_i \genK_j^{-1}  }{v-v^{-1}} &(i=k\enspace \& \enspace j=l).
\end{array}
\right.
\end{align*}
 \end{lem}
\begin{proof}
   Assume that $1\leq i,j,k,l\leq n$ and $i<j,k<l,i\leq k$. By the first equation in (QQ4),  it is easy to check that $\genF_i{\genXE}_{k,l}={\genXE}_{k,l}\genF_i$ for $i<k$ or $i>l-1$ and ${\genXE}_{j,i}{\genXE}_{k,l}={\genXE}_{k,l}{\genXE}_{j,i}$ for $i<j\leq k<l$. Using the first equation in (QQ4) and the equations in Corollary \ref{rv_k}, a direct calculation shows that 
   \begin{align*}
    {\genXE}_{j,i} \genE_{j-1}  
&=(-\genF_{j-1}{\genXE}_{j-1,i}+v{\genXE}_{j-1,i}\genF_{j-1})\genE_{j-1}\\
&=-\genF_{j-1}\genE_{j-1}{\genXE}_{j-1,i}+v{\genXE}_{j-1,i}\genF_{j-1}\genE_{j-1}\\
&=-(\genE_{j-1}\genF_{j-1}-\frac{\genK_{j-1}\genK_j^{-1}-\genK_{j-1}^{-1}\genK_j}{v-v^{-1}}){\genXE}_{j-1,i}\\
&\quad+v{\genXE}_{j-1,i}(\genE_{j-1}\genF_{j-1}-\frac{\genK_{j-1}\genK_j^{-1}-\genK_{j-1}^{-1}\genK_j}{v-v^{-1}})\\
&=\genE_{j-1}{\genXE}_{j,i}+{\genXE}_{j-1,i}\genK_{j-1}^{-1}\genK_j.
   \end{align*}
   Similarly, one can check case-by-case that
   \begin{equation}\label{FEij}
   \begin{aligned}
  \genF_{a}{\genXE}_{i,j}=\left\{
  \begin{array}{ll}
{\genXE}_{i,j}\genF_{a} &\hspace{-3em}(a<i \text{ or } i<a<j-1 \text{ or } a>j-1),\\
{\genXE}_{i,j}\genF_{a}-{\genXE}_{i+1,j}\genK_i^{-1}\genK_{i+1} &(a=i),\\
{\genXE}_{i,j}\genF_{a}+v^{-1}{\genXE}_{i,j-1}\genK_{j-1}\genK_j^{-1} &(a=j-1).
  \end{array}
\right.
   \end{aligned}
      \end{equation}
   Then the lemma is proved by Corollary \ref{rv_k} and the above formulas. It is easy to check that ${\genXE}_{j,i}{\genXE}_{k,l}={\genXE}_{k,l}{\genXE}_{j,i}$ for $i<k<l<j$. For the case $i=k<j<l$, we have 
   \begin{align*}
    {\genXE}_{j,i}{\genXE}_{i,l}
&=(-{\genXE}_{j,i+1}{\genXE}_{i+1,i}+v{\genXE}_{i+1,i}{\genXE}_{j,i+1}){\genXE}_{i,l}\\
&=-{\genXE}_{j,i+1}({\genXE}_{i,l}\genF_i-{\genXE}_{i+1,l}\genK_i^{-1}\genK_{i+1})+v({\genXE}_{i,l}\genF_i-{\genXE}_{i+1,l}\genK_i^{-1}\genK_{i+1}){\genXE}_{j,i+1}\\
&={\genXE}_{i,l}{\genXE}_{j,i}+({\genXE}_{j,i+1}{\genXE}_{i+1,l}-{\genXE}_{i+1,l}{\genXE}_{j,i+1})\genK_{i+1}\genK_i^{-1},
   \end{align*}
then by recursion, we obtain
\begin{align*}
 {\genXE}_{j,i} {\genXE}_{i,l}-{\genXE}_{i,l}{\genXE}_{j,i} 
 &= ({\genXE}_{j,i+1}{\genXE}_{i+1,l}-{\genXE}_{i+1,l}{\genXE}_{j,i+1})\genK_{i+1}\genK_i^{-1}\\
 &=({\genXE}_{j,i+2}{\genXE}_{i+2,l}-{\genXE}_{i+2,l}{\genXE}_{j,i+2})\genK_{i+2}\genK_i^{-1}\\
 &\quad\vdots\\
 &=({\genXE}_{j,j-1}{\genXE}_{j-1,l}-{\genXE}_{j-1,l}{\genXE}_{j,j-1})\genK_{j-1}\genK_i^{-1}\\
 &=-{\genXE}_{j,l}\genK_{j-1}^{-1}\genK_j\genK_{j-1}\genK_i^{-1}\\
 &=-{\genXE}_{j,l}\genK_j\genK_i^{-1}.
\end{align*}
Similarly, we have ${\genXE}_{j,i} {\genXE}_{k,j}= {\genXE}_{k,j}{\genXE}_{j,i}+ {\genXE}_{k,i}\genK_j\genK_{k}^{-1}$ by \eqref{FEij} with $i<k<j=l$. And also, a straightforward computation shows that the case $i<k<j<l$ can be deduced from the cases $i<j=k<l$ and $i=k<j<l$, and the case $i=k \enspace\&\enspace j=l$ can be obtained by the case $i<k<j<l$, and we omit the details.
\end{proof}

Observe that if we interchange $(i,j)$ and $(k,l)$, we can obtain another set of commutation formulas for  ${\genXE}_{j,i}{\genXE}_{k,l}$ by a similar calculation. More precisely, the following holds for $1\leq i,j,k,l\leq n$ satisfying $i<j,k<l,k\leq i$:
\begin{align*}
	{\genXE}_{j,i}{\genXE}_{k,l}=\left\{
\begin{array}{ll}
{\genXE}_{k,l}{\genXE}_{j,i} &\hspace{-3em}(k<l\leq i<j \text{ or }k<i<j<l),\\
{\genXE}_{k,l}{\genXE}_{j,i}-v{\genXE}_{j,l}\genK_i \genK_l^{-1} &(k=i<l<j),\\
{\genXE}_{k,l}{\genXE}_{j,i}+v^{-1}{\genXE}_{k,i}\genK_i\genK_j^{-1} 
&(k<i<j=l ),\\
{\genXE}_{k,l}{\genXE}_{j,i}-(v-v^{-1}){\genXE}_{k,i}{\genXE}_{j,l}\genK_i \genK_l^{-1} &(k<i<l<j).
\end{array}
\right.
\end{align*} 

Hence, this together with Lemma \ref{ne-pe1} gives complete commutation formulas for the even negative root vectors ${\genXE}_{j,i}$ and the even positive root vectors ${\genXE}_{k,l}$ with $i<j, k<l$.
 
 \begin{lem}\label{pe-po1}
 The following holds for $1\leq i,j,k,l\leq n$ satisfying $i<j,k<l$:
 
 If $i<j<k<l$ or $i<k<l<j$ or  $i=k\enspace \& \enspace j=l$ or $k<l<i<j$ or  $k<i<j<l$, then ${\genXE}_{i,j}{\ol{\genXE}}_{k,l}={\ol{\genXE}}_{k,l}{\genXE}_{i,j}$ and 
 \begin{align*}
{\genXE}_{i,j}{\ol{\genXE}}_{k,l}=\left\{
\begin{array}{ll}
v^{-1}{\ol{\genXE}}_{k,l}{\genXE}_{i,j}-{\ol{\genXE}}_{i,l} &(i<j=k<l),\\
v{\ol{\genXE}}_{k,l}{\genXE}_{i,j} &(i=k<j<l ),\\
{\ol{\genXE}}_{k,l}{\genXE}_{i,j}+(v-v^{-1}){\genXE}_{k,j}{\ol{\genXE}}_{i,l} &(i<k<j<l),\\
v^{-1}{\ol{\genXE}}_{k,l}{\genXE}_{i,j}+{\ol{\genXE}}_{i,j}{\genXE}_{k,j}-v^{-1}{\genXE}_{k,j}{\ol{\genXE}}_{i,j}&(i<k<j=l),\\
{v^{-1}{\ol{\genXE}}_{k,l}{\genXE}_{i,j}-({\ol{\genE}}_{i}{\genXE}_{i+1,j}-v^{-1}{\genXE}_{i+1,j}{{\genE}}_{\ol{i}}){\genXE}_{k,i}}\\\quad+v^{-1}{\genXE}_{k,i}({{\genE}}_{\ol{i}}{\genXE}_{i+1,j}-v^{-1}{\genXE}_{i+1,j}{{\genE}}_{\ol{i}}) &(k<l=i<j),\\
v^{-1}{\ol{\genXE}}_{k,l}{\genXE}_{i,j} &(k=i<l<j ),\\
{\ol{\genXE}}_{k,l}{\genXE}_{i,j}+v^{-1}{\ol{\genXE}}_{i,l}{\genXE}_{k,j}-v{\genXE}_{k,j}{\ol{\genXE}}_{i,l} &(k<i<l<j),\\
v{\ol{\genXE}}_{k,l}{\genXE}_{i,j}-v{\genXE}_{k,j}{\ol{\genXE}}_{i,j}+{\ol{\genXE}}_{i,j}{\genXE}_{k,j} &(k<i<j=l).
\end{array}
\right.
\end{align*}
 \end{lem}
\begin{proof}
Suppose $1\leq i,j,k,l\leq n$ and $i<j,k<l$.  We only prove the cases where $i\leq k$ and the remaining cases can be verified in a similar way. 
By \cite[Remark 6.5]{CLZ} and Remark \ref{EX}, we observe that
\begin{align*}
    &\genE_a {\ol{\genXE}}_{{a-1,a+2}}=T_{a-1}T_a (\genE_{a-1})T_{a-1}T_a (\genE_{\ol{a+1}})=T_{a-1}T_a (\genE_{\ol{a+1}}\genE_{a-1})= \genXE_{a-1,a+2}\genE_a,\\
    &\genE_{\ol{a}} {{\genXE}}_{{a-1,a+2}}=T_{a-1}T_a (\genE_{\ol{a-1}})T_{a-1}T_a (\genE_{{a+1}})=T_{a-1}T_a (\genE_{{a+1}}\genE_{\ol{a-1}})= \genXE_{a-1,a+2}\genE_{\ol{a}}.
\end{align*} 
Also, by the relations in (QQ5-QQ6) and the equations in Corollary \ref{rv_k}, a direct calculation shows that
\begin{equation}\label{Epo}
\begin{aligned}
   & \genE_a {\ol{\genXE}}_{i,j}=\left\{
\begin{array}{ll}
 {\ol{\genXE}}_{i,j}\genE_a &\hspace{-2em}(a<i-1 \text{ or } i<a<j-1 \text{ or } a>j),\\
 v^{-1} {\ol{\genXE}}_{i,j}\genE_{i-1}- {\ol{\genXE}}_{i-1,j} &(a=i-1),\\
 v {\ol{\genXE}}_{i,j}\genE_i &(a=i),\\
 v {\ol{\genXE}}_{i,j}\genE_{j-1}-v{{\genXE}}_{i,j}{{\genE}}_{\ol{j-1}}+{{\genE}}_{\ol{j-1}} {{\genXE}}_{i,j}&(a=j-1),\\
 v^{-1} {\ol{\genXE}}_{i,j}\genE_j +{{\genE}}_{\ol{j}}{{\genXE}}_{i,j}-v^{-1}{{\genXE}}_{i,j}{{\genE}}_{\ol{j}}&(a=j); 
\end{array}
\right.\\
\end{aligned}
\end{equation}
\begin{equation}
\begin{aligned}
{{\genE}}_{\ol{a}}{{\genXE}}_{i,j}=\left\{
\begin{array}{ll}
{{\genXE}}_{i,j}{{\genE}}_{\ol{a}} &\hspace{-2em}(a<i-1 \text{ or } i<a<j-1 \text{ or }  a>j),\\
{v{{\genXE}}_{i,j}{{\genE}}_{\ol{i-1}}-\genE_{i-1}({{\genE}}_{\ol{i}}{{\genXE}}_{i+1,j}-v^{-1}{{\genXE}}_{i+1,j}{{\genE}}_{\ol{i}})}\\\quad+v({{\genE}}_{\ol{i}}{{\genXE}}_{i+1,j}-v^{-1}{{\genXE}}_{i+1,j}{{\genE}}_{\ol{i}})\genE_{i-1}
&(a=i-1),\\
v{{\genXE}}_{i,j}{{\genE}}_{\ol{i}}
&(a=i),\\
v{{\genXE}}_{i,j}{{\genE}}_{\ol{j-1}}-v{\ol{\genXE}}_{i,j}\genE_{j-1}+\genE_{j-1}{\ol{\genXE}}_{i,j} &(a=j-1),\\
v{{\genXE}}_{i,j}{{\genE}}_{\ol{j}}
+v{\ol{\genXE}}_{i,j+1} &(a=j).
\end{array}
\right.
\end{aligned}
\end{equation}
Then the lemma is proved by Corollary \ref{rv_k} and the above formulas. It is easy to check ${\genXE}_{i,j}{\ol{\genXE}}_{k,l}={\ol{\genXE}}_{k,l}{\genXE}_{i,j}$ for $i<j<k<l$ or $i<k<l<j$. For the case $i<j=k<l$, we have 
\begin{align*}
    {{\genXE}}_{i,j}{\ol{\genXE}}_{j,l}
&=(-{{\genXE}}_{i,j-1}\genE_{j-1}+v^{-1}\genE_{j-1}{{\genXE}}_{i,j-1}){\ol{\genXE}}_{j,l}\\
&=-{{\genXE}}_{i,j-1}\genE_{j-1}{\ol{\genXE}}_{j,l}+v^{-1}\genE_{j-1}{{\genXE}}_{i,j-1}{\ol{\genXE}}_{j,l}\\
&=-{{\genXE}}_{i,j-1}(v^{-1}{\ol{\genXE}}_{j,l}\genE_{j-1}-{\ol{\genXE}}_{j-1,l})+v^{-1}(v^{-1}{\ol{\genXE}}_{j,l}\genE_{j-1}-{\ol{\genXE}}_{j-1,l}){{\genXE}}_{i,j-1}\\
&=-v^{-1}{\ol{\genXE}}_{j,l}{{\genXE}}_{i,j-1}\genE_{j-1}+{{\genXE}}_{i,j-1}{\ol{\genXE}}_{j-1,l}+v^{-2}{\ol{\genXE}}_{j,l}\genE_{j-1}{{\genXE}}_{i,j-1}-v^{-1}{\ol{\genXE}}_{j-1,l}{{\genXE}}_{i,j-1}\\
&=v^{-1}{\ol{\genXE}}_{j,l}{{\genXE}}_{i,j}-{\ol{\genXE}}_{i,l}.
\end{align*}
And a straightforward computation shows that the case $i<k<j=l$ can be deduced from the case $i=k<j<l$, and the case $i=k<j<l$ can be obtained from the case $i<k<l<j$. The case $i<k<j<l$ can be deduced from the cases $i<j=k<l$ and $i=k<j<l$. In addition, the case $i=k \enspace \& \enspace j=l$ can be deduced from the case $i=k<j<l$. Let us illustrate by checking in detail the case where $i=k \enspace \& \enspace j=l$, other cases can be verified similarly.
\begin{align*}
   {{\genXE}}_{i,j}{\ol{\genXE}}_{i,j} 
&=(-{{\genXE}}_{i,j-1}\genE_{j-1}+v^{-1}\genE_{j-1}{{\genXE}}_{i,j-1}){\ol{\genXE}}_{i,j}\\
&=-{{\genXE}}_{i,j-1}\genE_{j-1}{\ol{\genXE}}_{i,j}+v^{-1}\genE_{j-1}{{\genXE}}_{i,j-1}{\ol{\genXE}}_{i,j}\\
&=-{{\genXE}}_{i,j-1}(v{\ol{\genXE}}_{i,j}\genE_{j-1}-v{{\genXE}}_{i,j}{\genE}_{\ol{j-1}}+{\genE}_{\ol{j-1}}{{\genXE}}_{i,j})+v^{-1}\genE_{j-1}(v{\ol{\genXE}}_{i,j}{{\genXE}}_{i,j-1}) \quad (\text{by \eqref{Epo}})\\
&=-v{{\genXE}}_{i,j-1}{\ol{\genXE}}_{i,j}\genE_{j-1}+v{{\genXE}}_{i,j-1}{{\genXE}}_{i,j}{\genE}_{\ol{j-1}}-{{\genXE}}_{i,j-1}{\genE}_{\ol{j-1}}{{\genXE}}_{i,j}\\
&\qquad+(v{\ol{\genXE}}_{i,j}\genE_{j-1}-v{{\genXE}}_{i,j}{\genE}_{\ol{j-1}}+{\genE}_{\ol{j-1}}{{\genXE}}_{i,j}){{\genXE}}_{i,j-1}\\
&=-v^2{\ol{\genXE}}_{i,j}{{\genXE}}_{i,j-1}\genE_{j-1}+v^2{{\genXE}}_{i,j}{{\genXE}}_{i,j-1}{\genE}_{\ol{j-1}}-{{\genXE}}_{i,j-1}{\genE}_{\ol{j-1}}{{\genXE}}_{i,j}\\
&\qquad +v{\ol{\genXE}}_{i,j}\genE_{j-1}{{\genXE}}_{i,j-1}-v{{\genXE}}_{i,j}{\genE}_{\ol{j-1}}{{\genXE}}_{i,j-1}+v^{-1}{\genE}_{\ol{j-1}}{{\genXE}}_{i,j-1}{{\genXE}}_{i,j}\\
&=v^2{\ol{\genXE}}_{i,j}{{\genXE}}_{i,j}-v^2{{\genXE}}_{i,j}{\ol{\genXE}}_{i,j}+{\ol{\genXE}}_{i,j}{{\genXE}}_{i,j} \quad (\text{by Lemma \ref{pe-pe1}}).
\end{align*}
Then we have ${{\genXE}}_{i,j}{\ol{\genXE}}_{i,j}={\ol{\genXE}}_{i,j}{{\genXE}}_{i,j}$ as desired. 
\end{proof}

Hence, this gives complete commutation formulas for the even positive root vectors ${\genXE}_{i,j}$ and the odd positive root vectors ${\ol{\genXE}}_{k,l}$ with $i<j, k<l$. And also, applying $\Omega$ to Lemma \ref{pe-po1}, we can obtain the commutation formulas for even negative root vectors  ${\genXE}_{j,i}$ and odd negative root vectors ${\ol{\genXE}}_{l,k}$ with $i<j, k<l$.
 
 \begin{lem}\label{po-po1}
 The following holds for $1\leq i,j,k,l\leq n$ satisfying $i<j,k<l,i\leq k$:
 \begin{align*}
{\ol{\genXE}}_{i,j}{\ol{\genXE}}_{k,l}=\left\{
\begin{array}{ll}
-{\ol{\genXE}}_{k,l}{\ol{\genXE}}_{i,j} &(i<j<k<l\text{ or } i<k<l<j),\\
{-v{\ol{\genXE}}_{k,l}{\ol{\genXE}}_{i,j}+(v\genE_{\ol{j}}{\ol{\genXE}}_{j+1,l}+{\ol{\genXE}}_{j+1,l}\genE_{\ol{j}}){{\genXE}}_{i,j}}\\ \quad-v^{-1}{{\genXE}}_{i,j}(v\genE_{\ol{j}}{\ol{\genXE}}_{j+1,l}+{\ol{\genXE}}_{j+1,l}\genE_{\ol{j}})
 &(i<j=k<l),\\
-v{\ol{\genXE}}_{k,l}{\ol{\genXE}}_{i,j} &(i=k<j<l ),\\
-{\ol{\genXE}}_{k,l}{\ol{\genXE}}_{i,j}-(v+v^{-1})
{\ol{\genXE}}_{i,l}{\ol{\genXE}}_{k,j} &(i<k<j<l),\\
-v^{-1}{\ol{\genXE}}_{k,l}{\ol{\genXE}}_{i,j}-v^{-1}(v-v^{-1}){{\genXE}}_{i,j}{\genXE}_{k,j} &(i<k<j=l),\\
-\frac{v-v^{-1}}{v+v^{-1}}{\genXE}^2_{i,j} &(i=k\enspace\&\enspace j=l).
\end{array}
\right.
\end{align*}
 \end{lem}
\begin{proof}
    Assume that $1\leq i,j,k,l\leq n$ and  $i<j,k<l,i\leq k$. 
Applying \cite[Remark 6.5]{CLZ} and Remark \ref{EX}, we observe that
\begin{align*}
    \genE_{\ol{a}} {\ol{\genXE}}_{{a-1,a+2}}=T_{a-1}T_a (\genE_{\ol{a-1}})T_{a-1}T_a (\genE_{\ol{a+1}})=T_{a-1}T_a (-\genE_{\ol{a+1}}\genE_{a-1})= -\genXE_{a-1,a+2}\genE_{\ol{a}}.
\end{align*} 
    Also, by the relations in (QQ5-QQ6) and the equations in Corollary \ref{rv_k}, a direct calculation shows that 
    \begin{align*}
{{\genE}}_{\ol{a}}{\ol{\genXE}}_{i,j}=\left\{
\begin{array}{ll}
-{\ol{\genXE}}_{i,j}{{\genE}}_{\ol{a}} &(a<i-1 \text{ or } i<a<j-1
\text{ or } k>j),\\
{-v{\ol{\genXE}}_{i,j}{{\genE}}_{\ol{i-1}}-\genE_{i-1}({{\genE}}_{\ol{i}}{\ol{\genXE}}_{i+1,j}+v^{-1}{\ol{\genXE}}_{i+1,j}{{\genE}}_{\ol{i}})}\\\quad+v({{\genE}}_{\ol{i}}{\ol{\genXE}}_{i+1,j}+v^{-1}{\ol{\genXE}}_{i+1,j}{{\genE}}_{\ol{i}})\genE_{i-1} &(a=i-1),\\
-v{\ol{\genXE}}_{i,j}{{\genE}}_{\ol{i}} &(a=i),\\
-v{\ol{\genXE}}_{i,j}{{\genE}}_{\ol{j-1}}-(v-v^{-1}){{\genXE}}_{i,j}\genE_{j-1} &(a=j-1),\\
-v^{-1}{\ol{\genXE}}_{i,j}{{\genE}}_{\ol{j}}+ v^{-1}{{\genXE}}_{i,j}\genE_{j}-\genE_{j}{{\genXE}}_{i,j}&(a=j).
\end{array}
\right.
\end{align*}
Then the lemma is proved case-by-case using \eqref{Epo}, Lemma \ref{pe-po1}, Corollary \ref{rv_k} and the above formulas. It is easy to check ${\ol{\genXE}}_{i,j}{\ol{\genXE}}_{k,l}=-{\ol{\genXE}}_{k,l}{\ol{\genXE}}_{i,j}$ for $i<j<k<l$ and $i<k<l<j$. For the case $i<k<j=l$, a direct calculation shows that
\begin{align*}
{\ol{\genXE}}_{i,j}{\ol{\genXE}}_{k,j} 
&={\ol{\genXE}}_{i,j}(-{{\genXE}}_{k,j-1}\genE_{\ol{j-1}}+v^{-1}\genE_{\ol{j-1}}{{\genXE}}_{k,j-1})\\
&=-{\ol{\genXE}}_{i,j}{{\genXE}}_{k,j-1}\genE_{\ol{j-1}}+v^{-1}{\ol{\genXE}}_{i,j}\genE_{\ol{j-1}}{{\genXE}}_{k,j-1}\\
&=-{{\genXE}}_{k,j-1}{\ol{\genXE}}_{i,j}\genE_{\ol{j-1}}+v^{-1}{\ol{\genXE}}_{i,j}\genE_{\ol{j-1}}{{\genXE}}_{k,j-1} \qquad (\text{ by Lemma \ref{pe-po1}} )\\
&=-{{\genXE}}_{k,j-1}[-v^{-1}\genE_{\ol{j-1}}{\ol{\genXE}}_{i,j}-v^{-1}(v-v^{-1}){{\genXE}}_{i,j}\genE_{j-1}]\\
&\qquad+v^{-1}[-v^{-1}\genE_{\ol{j-1}}{\ol{\genXE}}_{i,j}-v^{-1}(v-v^{-1}){{\genXE}}_{i,j}\genE_{j-1}]{{\genXE}}_{k,j-1}\\
&=-v^{-1}{\ol{\genXE}}_{k,j}{\ol{\genXE}}_{i,j}+v^{-1}(v-v^{-1}){{\genXE}}_{k,j-1}{{\genXE}}_{i,j}\genE_{j-1}\\
&\qquad
-v^{-2}(v-v^{-1}){{\genXE}}_{i,j}\genE_{j-1}{{\genXE}}_{k,j-1}\\
&=-v^{-1}{\ol{\genXE}}_{k,j}{\ol{\genXE}}_{i,j}+v^{-1}(v-v^{-1}){{\genXE}}_{k,j-1}{{\genXE}}_{i,j}\genE_{j-1}
\\
&\qquad-v^{-2}(v-v^{-1}){{\genXE}}_{i,j}(v{{\genXE}}_{k,j-1}\genE_{j-1}+v{{\genXE}}_{k,j}) \qquad (\text{ by } \eqref{Epe})\\
&=-v^{-1}{\ol{\genXE}}_{k,j}{\ol{\genXE}}_{i,j}-v^{-1}(v-v^{-1}){{\genXE}}_{i,j}{{\genXE}}_{k,j}\qquad (\text{ by Lemma \ref{pe-pe1}} ).
\end{align*}

Similarly, the cases $i<j=k<l$, $i<k<l<j$, $i<k<j<l$ can be verified, and the case $i<k<j<l$ can be deduced from the case $i<k<l<j$, and we omit the details. For the case $i=k\enspace\&\enspace j=l$, we have
\begin{align*}
    {\ol{\genXE}}_{i,j}^2
&={\ol{\genXE}}_{i,j}(-\genE_i{\ol{\genXE}}_{i+1,j}+v^{-1}{\ol{\genXE}}_{i+1,j}\genE_i)\\
&=-{\ol{\genXE}}_{i,j}\genE_i{\ol{\genXE}}_{i+1,j} +v^{-1}{\ol{\genXE}}_{i,j}{\ol{\genXE}}_{i+1,j}\genE_i\\
&=-v^{-1}\genE_i[-v^{-1}{\ol{\genXE}}_{i+1,j} {\ol{\genXE}}_{i,j}-(1-v^{-2}){{\genXE}}_{i,j}{{\genXE}}_{i+1,j}]\\&\qquad
+v^{-1}[-v^{-1}{\ol{\genXE}}_{i+1,j} {\ol{\genXE}}_{i,j}-(1-v^{-2}){{\genXE}}_{i,j}{{\genXE}}_{i+1,j}]\genE_i \quad (\text{by \eqref{Epo}}  )\\
&=-v^{-2}{\ol{\genXE}}_{i,j}^2-(1-v^{-2}){{\genXE}}_{i,j}^2,
\end{align*}
then we obtain ${\ol{\genXE}}_{i,j}^2=-\frac{v-v^{-1}}{v+v^{-1}}{{\genXE}}_{i,j}^2$.
\end{proof}

Observe that if we interchange $(i,j)$ and $(k,l)$, we can obtain another set of commutation formulas for  ${\ol{\genXE}}_{i,j}{\ol{\genXE}}_{k,l}$ by a similar calculation. More precisely, the following holds for $1\leq i,j,k,l\leq n$ satisfying $i<j,k<l,k\leq i$:
\begin{align*}
{\ol{\genXE}}_{i,j}{\ol{\genXE}}_{k,l}=\left\{
\begin{array}{ll}
-{\ol{\genXE}}_{k,l}{\ol{\genXE}}_{i,j} &(k<l<i<j\text{ or } k<i<j<l),\\
-v^{-1}{\ol{\genXE}}_{k,i}{\ol{\genXE}}_{i,j}+(\genE_{\ol{i}}{\ol{\genXE}}_{i+1,j}+v^{-1}{\ol{\genXE}}_{i+1,j}\genE_{\ol{i}}){{\genXE}}_{k,i}\\ \quad-v^{-1}{{\genXE}}_{k,i}(\genE_{\ol{i}}{\ol{\genXE}}_{i+1,j}+v^{-1}{\ol{\genXE}}_{i+1,j}\genE_{\ol{i}}) &(k<l=i<j),\\
-v^{-1}{\ol{\genXE}}_{k,l}{\ol{\genXE}}_{i,j} &(k=i<l<j ),\\
-{\ol{\genXE}}_{k,l}{\ol{\genXE}}_{i,j}-(v+v^{-1}){\ol{\genXE}}_{k,j}{\ol{\genXE}}_{i,l} &(k<i<l<j),\\
-v{\ol{\genXE}}_{k,l}{\ol{\genXE}}_{i,j}-(v-v^{-1}){{\genXE}}_{k,j}{\genXE}_{i,j} &(k<i<j=l).
\end{array}
\right.
\end{align*}

 Hence, this together with Lemma \ref{po-po1} gives complete commutation formulas for the odd positive root vectors ${\ol{\genXE}}_{i,j}$ and ${\ol{\genXE}}_{k,l}$ with $i<j, k<l$. And also, applying $\Omega$ to Lemma \ref{po-po1} and the above set of formulas, we can obtain the commutation formulas for odd negative root vectors  ${\ol{\genXE}}_{j,i}$ and ${\ol{\genXE}}_{l,k}$ with $i<j, k<l$.
 
 \begin{lem}\label{ne-po1}
 The following holds for $1\leq i,j,k,l\leq n$ satisfying $i<j,k<l$:

 If $i<j\leq k<l$  or $i<k<l<j$ or $k<l\leq i<j$ or $k<i<j<l$, then ${\genXE}_{j,i}{\ol{\genXE}}_{k,l}={\ol{\genXE}}_{k,l}{\genXF}_{j,i} $ and 
 \begin{align*}
{\genXE}_{j,i}{\ol{\genXE}}_{k,l}=\left\{
\begin{array}{ll}
{\ol{\genXE}}_{k,l}{\genXF}_{j,i}-{\ol{\genXE}}_{j,l}\genK_j \genK_i^{-1} &(i=k<j<l),\\
{\ol{\genXE}}_{k,l}{\genXE}_{j,i}+{\ol{\genXE}}_{j,i}{\genXE}_{k,j}\genK_j^2-{\genXE}_{k,j}{\ol{\genXE}}_{j,i}\genK_j^2-(v-v^{-1}){\genXE}_{k,i}\genK_{\ol{j}}\genK_k^{-1}\genK_j^2
&(i<k<j=l ),\\
{\ol{\genXE}}_{k,l}{\genXE}_{j,i}-(v-v^{-1}){\genXE}_{k,i}{\ol{\genXE}}_{j,l}\genK_j \genK_k^{-1} &(i<k<j<l),\\
{{\ol{\genXE}}_{k,l}{\genXE}_{j,i}+({\ol{\genXE}}_{j,i}{\genXE}_{i,j}-{\genXE}_{i,j}{\ol{\genXE}}_{j,i})\genK_j^2+\genK_{\ol{j}}\genK_i^{-1}\genK_j^2+\genK_{\ol{j}}\genK_i} &(i=k\enspace \& \enspace j=l),\\
{\ol{\genXE}}_{k,l}{\genXE}_{j,i}+v\genK_{\ol{l}}{\genXE}_{j,l}\genK_i-{\genXE}_{j,l}\genK_{\ol{l}}\genK_i &(k=i<l<j),\\
 {\ol{\genXE}}_{k,l}{\genXE}_{j,i}-(v-v^{-1})\genK_{\ol{l}}{\genXE}_{j,l}{\genXE}_{k,i}\genK_i+(1-v^{-2}){\genXE}_{j,l}\genK_{\ol{l}}{\genXE}_{k,i}\genK_i
 &(k<i<l<j),\\
 {\ol{\genXE}}_{k,j}{\genXE}_{j,i}+ {\ol{\genXE}}_{j,i}{\genXE}_{k,j}\genK_{j}^2-{\genXE}_{k,j}{\ol{\genXE}}_{j,i}\genK_{j}^2-(1-v^{-2}){\genXE}_{k,i}\genK_i\genK_{\ol{j}} &(k<i<j=l).
\end{array}
\right.
\end{align*}
 \end{lem}
 \begin{proof}
     Suppose $1\leq i,j,k,l\leq n$ and $i<j,k<l$. We only prove the cases where $i\leq k$ and the remaining cases can be verified in a similar way. By the relations in (QQ4) and Corollary \ref{rv_k}, a direct calculation shows that
\begin{align*}
    \genF_{j-1}{\ol{\genXE}}_{i,j}
&=\genF_{j-1}(-{{\genXE}}_{i,j-1}\genE_{\ol{j-1}}+v^{-1}\genE_{\ol{j-1}}{{\genXE}}_{i,j-1})\\
&=-{{\genXE}}_{i,j-1}\genF_{j-1}\genE_{\ol{j-1}}+v^{-1}\genF_{j-1}\genE_{\ol{j-1}}{{\genXE}}_{i,j-1}\\
&=-{{\genXE}}_{i,j-1}(\genE_{\ol{j-1}}\genF_{j-1}-\genK_j\genK_{\ol{j-1}}+\genK_{\ol{j}}\genK_{j-1})\\
&\qquad+v^{-1}(\genE_{\ol{j-1}}\genF_{j-1}-\genK_j\genK_{\ol{j-1}}+\genK_{\ol{j}}\genK_{j-1}){{\genXE}}_{i,j-1}\\
&=-{{\genXE}}_{i,j-1}\genE_{\ol{j-1}}\genF_{j-1}+{{\genXE}}_{i,j-1}\genK_j\genK_{\ol{j-1}}-{{\genXE}}_{i,j-1}\genK_{\ol{j}}\genK_{j-1}\\
&\qquad +v^{-1}\genE_{\ol{j-1}}{{\genXE}}_{i,j-1}\genF_{j-1}-v^{-1}\genK_j\genK_{\ol{j-1}}{{\genXE}}_{i,j-1}+v^{-1}\genK_{\ol{j}}\genK_{j-1}{{\genXE}}_{i,j-1}\\
&={\ol{\genXE}}_{i,j}\genF_{j-1}-(v-v^{-1})\genK_{\ol{j}}\genK_{j-1}{{\genXE}}_{i,j-1}\\
&\qquad+(v\genK_{\ol{j-1}}{{\genXE}}_{i,j-1}+v{\ol{\genXE}}_{i,j-1}\genK_{j-1}^{-1}-v^{-1}\genK_{\ol{j-1}}{{\genXE}}_{i,j-1})\genK_j \qquad (\text{by Lemma \ref{KE}})\\
&={\ol{\genXE}}_{i,j}\genF_{j-1}-(v-v^{-1})\genK_{\ol{j}}\genK_{j-1}{{\genXE}}_{i,j-1}\\
&\qquad +(v-v^{-1})\genK_{\ol{j-1}}{{\genXE}}_{i,j-1}\genK_j+v{\ol{\genXE}}_{i,j-1}\genK_j\genK_{j-1}^{-1}.
\end{align*}
Similarly, one can check case-by-case that
\begin{equation}\label{Fpo}
\begin{aligned}
    \genF_a{\ol{\genXE}}_{i,j}&=\left\{
 \begin{array}{ll}
 {\ol{\genXE}}_{i,j}\genF_a &\hspace{-4em}(a<i \text{ or } i<a<j-1 \text{ or } a>j-1),\\
 {\ol{\genXE}}_{i,j}\genF_i - {\ol{\genXE}}_{i+1,j}\genK_{i+1}\genK_i^{-1} &(a=i),\\
 {\ol{\genXE}}_{i,j}\genF_{j-1}-(v-v^{-1})\genK_{\ol{j}}\genK_{j-1}{{\genXE}}_{i,j-1}\\
\quad +(v-v^{-1})\genK_{\ol{j-1}}{{\genXE}}_{i,j-1}\genK_j+v{\ol{\genXE}}_{i,j-1}\genK_j\genK_{j-1}^{-1}&(a=j-1).
 \end{array}
\right.\\
\genE_{\ol{a}}\genXE_{j,i}&=\left\{
 \begin{array}{ll}
\genXE_{j,i}\genE_{\ol{a}} &\hspace{-4em}(a<i \text{ or } i<a<j-1 \text{ or } a>j-1),\\
\genXE_{j,i}\genE_{\ol{i}}-v\genK_{\ol{i+1}} \genXE_{j,i+1}\genK_i+\genXE_{j,i+1}\genK_{\ol{i+1}}\genK_i &(a=i),\\
\genXE_{j,i}\genE_{\ol{j-1}}-\genK_{\ol{j-1}}\genXE_{j-1,i}\genK_j+v\genXE_{j-1,i}\genK_{\ol{j-1}}\genK_j &(a=j-1).
  \end{array}
\right.
\end{aligned}
\end{equation}
Then the lemma is proved by Corollary \ref{rv_k}, Lemma \ref{ne-pe1} and the above formulas. It is easy to check ${\genXE}_{j,i}{\ol{\genXE}}_{k,l}={\ol{\genXE}}_{k,l}{\genXE}_{j,i}$ for $i<j\leq k<l$ and $i<k<l<j$. For the case $i=k<j<l$, we obtain
\begin{align*}
    {\genXE}_{j,i}{\ol{\genXE}}_{i,l}
&=(-{\genXE}_{j,i+1}{\genXE}_{i+1,i}+v{\genXE}_{i+1,i}{\genXE}_{j,i+1}){\ol{\genXE}}_{i,l}\\
&=-{\genXE}_{j,i+1}{\genF}_{i}{\ol{\genXE}}_{i,l}+v{\genF}_{i}{\ol{\genXE}}_{i,l}{\genXE}_{j,i+1}\\
&=-{\genXE}_{j,i+1}({\ol{\genXE}}_{i,l}{\genF}_{i}-{\ol{\genXE}}_{i+1,l}\genK_{i+1}\genK_i^{-1})+v({\ol{\genXE}}_{i,l}{\genF}_{i}-{\ol{\genXE}}_{i+1,l}\genK_{i+1}\genK_i^{-1}){\genXE}_{j,i+1}\\
&={\ol{\genXE}}_{i,l}{\genXE}_{j,i}+({\genXE}_{j,i+1}{\ol{\genXE}}_{i+1,l}-{\ol{\genXE}}_{i+1,l}{\genXE}_{j,i+1})\genK_{i+1}\genK_i^{-1},
\end{align*}
then by recursion, we have 
\begin{align*}
  {\genXE}_{j,i}{\ol{\genXE}}_{i,l}-{\ol{\genXE}}_{i,l}{\genXE}_{j,i} 
  &=({\genXE}_{j,i+1}{\ol{\genXE}}_{i+1,l}-{\ol{\genXE}}_{i+1,l}{\genXE}_{j,i+1})\genK_{i+1}\genK_i^{-1}\\
  &=({\genXE}_{j,j-1}{\ol{\genXE}}_{j-1,l}-{\ol{\genXE}}_{j-1,l}{\genXE}_{j,j-1})\genK_{j-1}\genK_i^{-1}\\
  &=-{\ol{\genXE}}_{j,l}\genK_{j}\genK_i^{-1}.
\end{align*}
Hence, we obtain ${\genXE}_{j,i}{\ol{\genXE}}_{i,l}={\ol{\genXE}}_{i,l}{\genXE}_{j,i}-{\ol{\genXE}}_{j,l}\genK_{j}\genK_i^{-1}$. 
For the case $i<k<j=l$, we have 
\begin{align*}
  {\genXE}_{j,i}{\ol{\genXE}}_{k,j}  
&={\genXE}_{j,i}(-{\genXE}_{k,j-1}\genE_{\ol{j-1}}+v^{-1}\genE_{\ol{j-1}}{\genXE}_{k,j-1})\\
&=-{\genXE}_{k,j-1}{\genXE}_{j,i}\genE_{\ol{j-1}}+v^{-1}{\genXE}_{j,i}\genE_{\ol{j-1}}{\genXE}_{k,j-1}\\
&=-{\genXE}_{k,j-1}(\genE_{\ol{j-1}}{\genXE}_{j,i}+\genK_j\genK_{\ol{j-1}}{\genXE}_{j-1,i}-v{\genXE}_{j-1,i}\genK_j\genK_{\ol{j-1}})\\
&\qquad+v^{-1}(\genE_{\ol{j-1}}{\genXE}_{j,i}+\genK_j\genK_{\ol{j-1}}{\genXE}_{j-1,i}-v{\genXE}_{j-1,i}\genK_j\genK_{\ol{j-1}}){\genXE}_{k,j-1}\\
&={\ol{\genXE}}_{k,j}{\genXE}_{j,i}-{\genXE}_{k,j-1}\genK_j\genK_{\ol{j-1}}{\genXE}_{j-1,i}+v{\genXE}_{k,j-1}{\genXE}_{j-1,i}\genK_j\genK_{\ol{j-1}}\\
&\qquad +v^{-1}\genK_j\genK_{\ol{j-1}}{\genXE}_{j-1,i}{\genXE}_{k,j-1}-{\genXE}_{j-1,i}\genK_j\genK_{\ol{j-1}}{\genXE}_{k,j-1}\\
&={\ol{\genXE}}_{k,j}{\genXE}_{j,i}-{\genXE}_{k,j-1}\genK_j\genK_{\ol{j-1}}{\genXE}_{j-1,i}+v{\genXE}_{j-1,i}{\genXE}_{k,j-1}\genK_j\genK_{\ol{j-1}}-v{\genXE}_{k,i}\genK_{j-1}\genK_k^{-1}\genK_j\genK_{\ol{j-1}}\\
&\qquad +v^{-1}\genK_j\genK_{\ol{j-1}}{\genXE}_{k,j-1}{\genXE}_{j-1,i}+v^{-1}\genK_j\genK_{\ol{j-1}}{\genXE}_{k,i}\genK_{j-1}\genK_k^{-1}-{\genXE}_{j-1,i}\genK_j\genK_{\ol{j-1}}{\genXE}_{k,j-1}\\
&={\ol{\genXE}}_{k,j}{\genXE}_{j,i}-\genF_{\ol{j-1}}{\genXE}_{k,j}{\genXE}_{j-1,i}\genK_j^2+v{\genXE}_{j-1,i}\genF_{\ol{j-1}}{\genXE}_{k,j}\genK_j^2\\
&\qquad -{\genXE}_{k,j}{\ol{\genXE}}_{j,i}\genK_j^2-(v-v^{-1}){\genXE}_{k,i}\genK_{\ol{j}}\genK_k^{-1}\genK_j^2\\
&={\ol{\genXE}}_{k,j}{\genXE}_{j,i}+{\ol{\genXE}}_{j,i}{\genXE}_{k,j}\genK_j^2-{\genXE}_{k,j}{\ol{\genXE}}_{j,i}\genK_j^2-(v-v^{-1}){\genXE}_{k,i}\genK_{\ol{j}}\genK_k^{-1}\genK_j^2.
\end{align*}
The second-to-last equality follows by applying the anti-involution $\Omega$ to \eqref{Fpo}, which yields $\genF_{\ol{j-1}}{\genXE}_{k,j}={\genXE}_{k,j}\genF_{\ol{j-1}}+\genK_j^{-1}{\genXE}_{k,j-1}\genK_{\ol{j-1}}-v^{-1}\genK_j^{-1}\genK_{\ol{j-1}}{\genXE}_{k,j-1}$.
Similarly, one can verify the case $i=k<j<l$ directly. And a straightforward computation shows that the cases $i<k<j<l$ and $i=k \enspace\& j=l$ can be deduced from the case $i=k<j<l$, and we omit the details.
 \end{proof}

 Applying $\Omega$ to Lemma \ref{ne-po1}, we can obtain the commutation formulas for odd negative root vectors ${\ol{\genXE}}_{l,k}$ and even positive root vectors ${{\genXE}}_{i,j}$ with $i<j,k<l$.
 
 \begin{lem}\label{no-po1}
 The following holds for $1\leq i,j,k,l\leq n$ satisfying $i<j,k<l,i\leq k$:
 \begin{align*}
{\ol{\genXE}}_{j,i}{\ol{\genXE}}_{k,l}&=\left\{
\begin{array}{ll}
-{\ol{\genXE}}_{k,l}{\ol{\genXE}}_{j,i} &\hspace{-2cm}(i<j\leq k<l \text{ or }i<k<l<j),\\
-{\ol{\genXE}}_{k,l}{\ol{\genXE}}_{j,i}-\genK_{\ol{j}}{\ol{\genXE}}_{j,l}\genK_i^{-1}-v^{-1}{\ol{\genXE}}_{j,l} \genK_{\ol{j}}\genK_i^{-1} &(i=k<j<l),\\
-{\ol{\genXE}}_{k,l}{\ol{\genXE}}_{j,i}-{\genXE}_{k,i}\genK_k^{-1}\genK_j^{-1}-(v-v^{-1})({\genXE}_{j,i}{\ol{\genXE}}_{k,j}-{\ol{\genXE}}_{k,j}{\genXE}_{j,i})\genK_j^{-1}\genK_{\ol{j}}
&(i<k<j=l ),\\
-{\ol{\genXE}}_{k,l}{\ol{\genXE}}_{j,i}-(v-v^{-1})\genK_{\ol{j}}{\ol{\genXE}}_{j,l}{\genXE}_{k,i}\genK_k^{-1}-v^{-1}(v-v^{-1}){\genXE}_{k,i}{\ol{\genXE}}_{j,l}\genK_{\ol{j}}\genK_k^{-1} &(i<k<j<l),\\
-{\ol{\genXE}}_{k,l}{\ol{\genXE}}_{j,i}+\frac{\genK_j \genK_i-\genK_i^{-1} \genK_j^{-1}  }{v-v^{-1}}+(v-v^{-1})\genK_{\ol{j}}^2\genK_i^{-1}\genK_j\\ \quad{-(v-v^{-1})({\ol{\genXE}}_{j,i}{\genXE}_{i,j}-{\genXE}_{i,j}{\ol{\genXE}}_{j,i})\genK_j\genK_{\ol{j}}} &(i=k\enspace \& \enspace j=l).
\end{array}
\right.
\end{align*}
\end{lem}
\begin{proof}
    Assume that $1\leq i,j,k,l\leq n$ and $i<j,k<l,i\leq k$. By the relations in (QQ4) and the equations in Corollary \ref{rv_k}, a direct calculation shows that
\begin{align*}
    \genF_{\ol{a}}{\ol{\genXE}}_{i,j}=\left\{
\begin{array}{ll}
- {\ol{\genXE}}_{i,j}\genF_{\ol{a}} &\hspace{-5em}(a<i \text{ or } i<a<j-1 \text{ or } a>j-1),\\
- {\ol{\genXE}}_{i,j}\genF_{\ol{i}}-\genK_{\ol{i+1}}{\ol{\genXE}}_{i+1,j}\genK_i^{-1}-v^{-1}{\ol{\genXE}}_{i+1,j}\genK_{\ol{i+1}}\genK_i^{-1} &(a=i),\\
- {\ol{\genXE}}_{i,j}\genF_{\ol{j-1}}-\genK_{j-1}{\genXE}_{i,j-1}\genK_j-(v-v^{-1}){\genXE}_{i,j-1}\genK_{\ol{j-1}}\genK_{\ol{j}}\\
\quad+v^{-1}(v-v^{-1})\genK_{\ol{j-1}}{\genXE}_{i,j-1}\genK_{\ol{j}} &(a=j-1).
\end{array}
\right.
\end{align*}
Applying $\Omega$, we can obtain the commutation formulas for $\genE_{\ol{a}}{\ol{\genXE}}_{j,i}$:
\begin{align*}
\genE_{\ol{a}}{\ol{\genXE}}_{j,i}
=\left\{
\begin{array}{ll}
- {\ol{\genXE}}_{j,i}\genE_{\ol{a}} &\hspace{-5em}(a<i \text{ or } i<a<j-1 \text{ or } a>j-1),\\
 - {\ol{\genXE}}_{j,i}\genE_{\ol{i}}-{\ol{\genXE}}_{j,i+1}\genK_{\ol{i+1}}\genK_i-v \genK_{\ol{i+1}}{\ol{\genXE}}_{j,i+1}\genK_i &(a=i),\\
- {\ol{\genXE}}_{j,i}\genE_{\ol{j-1}}-{\genXE}_{j-1,i}\genK_{j-1}^{-1}\genK_j^{-1}+(v-v^{-1})\genK_{\ol{j-1}}{\genXE}_{j-1,i}\genK_{\ol{j}}\\
\quad-v(v-v^{-1}){\genXE}_{j-1,i}\genK_{\ol{j-1}}\genK_{\ol{j}}&(a=j-1).
\end{array}
\right.
\end{align*}
Then the lemma is proved case-by-case using Corollary \ref{rv_k} and the above formulas. It is easy to check ${\ol{\genXE}}_{j,i}{\ol{\genXE}}_{k,l}=-{\ol{\genXE}}_{k,l}{\ol{\genXE}}_{j,i}$ for $i<j\leq k<l $ and $i<k<l<j$. For the case $i=k<j<l$, we have 
\begin{align*}
{\ol{\genXE}}_{j,i}{\ol{\genXE}}_{i,l}
&=(-{\ol{\genXE}}_{j,i+1}\genF_i+v\genF_i{\ol{\genXE}}_{j,i+1}){\ol{\genXE}}_{i,l}\\
&=-{\ol{\genXE}}_{j,i+1}\genF_i{\ol{\genXE}}_{i,l}-v\genF_i{\ol{\genXE}}_{i,l}{\ol{\genXE}}_{j,i+1}\\
&=-{\ol{\genXE}}_{j,i+1}({\ol{\genXE}}_{i,l}\genF_i-{\ol{\genXE}}_{i+1,l}\genK_i^{-1}\genK_{i+1})-v({\ol{\genXE}}_{i,l}\genF_i-{\ol{\genXE}}_{i+1,l}\genK_i^{-1}\genK_{i+1}){\ol{\genXE}}_{j,i+1}\\
&=-{\ol{\genXE}}_{i,l}{\ol{\genXE}}_{j,i}+({\ol{\genXE}}_{j,i+1}{\ol{\genXE}}_{i+1,l}+{\ol{\genXE}}_{i+1,l}{\ol{\genXE}}_{j,i+1})\genK_{i+1}\genK_i^{-1},
\end{align*}
then by recursion, we obtain
\begin{align*}
{\ol{\genXE}}_{j,i}{\ol{\genXE}}_{i,l}+{\ol{\genXE}}_{i,l}{\ol{\genXE}}_{j,i}
&=({\ol{\genXE}}_{j,i+1}{\ol{\genXE}}_{i+1,l}+{\ol{\genXE}}_{i+1,l}{\ol{\genXE}}_{j,i+1})\genK_{i+1}\genK_i^{-1}\\
&=(\genF_{\ol{j-1}}{\ol{\genXE}}_{j-1,l}+{\ol{\genXE}}_{j-1,l}\genF_{\ol{j-1}})\genK_{j-1}\genK_i^{-1}\\
&=-\genK_{\ol{j}}{\ol{\genXE}}_{j,l}\genK_i^{-1}-v^{-1}{\ol{\genXE}}_{j,l}\genK_{\ol{j}}\genK_i^{-1}.
\end{align*}
Hence, we have ${\ol{\genXE}}_{j,i}{\ol{\genXE}}_{i,l}=-{\ol{\genXE}}_{i,l}{\ol{\genXE}}_{j,i}-\genK_{\ol{j}}{\ol{\genXE}}_{j,l}\genK_i^{-1}-v^{-1}{\ol{\genXE}}_{j,l}\genK_{\ol{j}}\genK_i^{-1}$.
Similarly, one can verify the cases $i<k<j=l$ and $i=k\enspace\& \enspace j=l $ by Lemma \ref{ne-po1}. And a straightforward computation shows that the case $i<k<j<l$ can be deduced from the case $i=k<j<l$, and we omit the details.
\end{proof}

Observe that if we interchange $(i,j)$ and $(k,l)$, we can obtain another set of commutation formulas for  ${\ol{\genXE}}_{j,i}{\ol{\genXE}}_{k,l}$ by a similar calculation. More precisely, the following holds for $1\leq i,j,k,l\leq n$ satisfying $i<j,k<l,k\leq i$:
 \begin{align*}
{\ol{\genXE}}_{j,i}{\ol{\genXE}}_{k,l}&=\left\{
\begin{array}{ll}
-{\ol{\genXE}}_{k,l}{\ol{\genXE}}_{j,i} &\hspace{-2cm}(k<l\leq i<j  \text{ or }k<i<j<l),\\
-{\ol{\genXE}}_{k,l}{\ol{\genXE}}_{j,i}-{\ol{\genXE}}_{j,l}\genK_{\ol{l}}\genK_i-v \genK_{\ol{l}}{\ol{\genXE}}_{j,l} \genK_i &(k=i<l<j),\\
-{\ol{\genXE}}_{k,l}{\ol{\genXE}}_{j,i}-v^{-1}{\genXE}_{k,i}\genK_i\genK_j+(v-v^{-1})\genK_{\ol{j}}({\ol{\genXE}}_{j,i}{\genXE}_{k,j}-{\genXE}_{k,j}{\ol{\genXE}}_{j,i})\genK_j 
&(k<i<j=l ),\\
-{\ol{\genXE}}_{k,l}{\ol{\genXE}}_{j,i}+(v-v^{-1})\genK_{\ol{l}}{\ol{\genXE}}_{j,l}{\genXE}_{k,i}\genK_i+v^{-1}(v-v^{-1}){\ol{\genXE}}_{j,l}\genK_{\ol{l}}{\genXE}_{k,i}\genK_i &(k<i<l<j).
\end{array}
\right.
\end{align*}

Hence, this together with Lemma \ref{no-po1} gives complete commutation formulas for the odd negative root vectors ${\ol{\genXE}}_{j,i}$ and the odd positive root vectors ${\ol{\genXE}}_{k,l}$ with $i<j, k<l$.

\section{PBW-type basis for the integral form of $\Uvqn$}\label{PBW}
In this section, using the realization of $\Uvqn$, we construct a PBW-type basis for the Lusztig form of the quantum queer superalgebra.

 For $m\geq 1$, let $[m]_v=\ds\frac{v^m-v^{-m}}{v-v^{-1}}$ and $[m]^!_v=[m]_v [m-1]_v\cdots[1]_v.$
We also use the convention $[0]_v=[0]^!_v=1.$ For $c\in \Z,t\geq 1$,  set
$$
\begin{bmatrix}
c\\m
\end{bmatrix}_v
=\frac{[c]_v[c-1]_v\cdots[c-m+1]_v}{[m]^!_v},\quad
\begin{bmatrix}
c\\0
\end{bmatrix}_v
=1.
$$
Generally, for an element $Y$ in an associative $\Q(v)$-algebra and $m\in\N$, let
$$
Y^{(m)}=\frac{Y^m}{[m]^!_v}.
$$
If $Y$ is invertible, define, for $t\geq 1$ and $c\in\Z$,

\begin{equation}\label{Kt}
\aligned
\begin{bmatrix}
Y;c\\ t
\end{bmatrix}_v
&=\prod^t_{s=1}\frac{Yv^{c-s+1}-Y^{-1}v^{-c+s-1}}{v^s-v^{-s}},
\quad \text{and }\;\begin{bmatrix}
Y;c\\ 0
\end{bmatrix}_v=1.
\endaligned
\end{equation}
And denote $\begin{bmatrix}
\genK_i\\t
\end{bmatrix}_v=\begin{bmatrix}
\genK_i;0\\t
\end{bmatrix}_v$ for convenience. 

   For clarity, we adopt the following notations for the product ordering specified in \cite[(25)]{GLL}. For any $m, n \in \mathbb{N}$, $m\le n$, set
   \begin{align*}
   \prod_{m\le h\le n}T_h :=T_mT_{m+1}\cdots T_n ,\qquad \prod_{n\ge  h\ge m}T_h :=T_nT_{n-1}\cdots T_m.
   \end{align*}

   {Let $\N_2 = \{{{0}}, {{1}}\}$
and denote
\begin{align*}
&M_n(\N|\N_2)
   =\{A=(A^{\ol{0}}|A^{\ol{1}}) 
   \enspace 
   | 
   \enspace
   A^{\ol{0}}=(a^{\ol{0}}_{i,j})\in M_n(\N), 
   A^{\ol{1}}=(a^{\ol{1}}_{i,j})\in M_n(\N_2) \},\\
 &M_n(\N|\N_2)^{\pm}
   =\{
   A
   \in M_n(\N|\N_2)
   \enspace
   |
   \enspace
   a_{i,i}^{\ol{0}}=0 \enspace \mbox{for all} \enspace i
   \},\\
   &M_n(\N|\N_2)_r
   =\{A \in M_n(\N|\N_2)
   \enspace 
   | 
   \enspace
   \sum_{1\le i,j \le n} (a^{\ol{0}}_{i,j} + 
   a^{\ol{1}}_{i,j}) =r \}.
\end{align*}}

Let $\mcZ=\Z[v,v^{-1}]$. Referring to \cite[Section 8]{DW}, the Lusztig form for $\Uvqn$, denoted by $\qUZ$, is a $\mcZ$-subsuperalgebra of $\Uvqn$ generated by
\begin{align*}
\genK_i^{\pm1},~
\begin{bmatrix}
\genK_i\\t
\end{bmatrix},~ \genE_j^{(m)},~ \genF^{(m)}_j,~ \genK_{\bi},~ \genE_{\bj},~ \genF_{\bj}\quad (1\leq i\leq n, 1\leq j\leq n-1,t,m\in\N).
\end{align*}
Let ${{\boldsymbol U}^{0}_{v,\mc Z}}$ be the $\mcZ$-subalgebra of $\Uvqn$ generated by $\genK_i^{\pm1},~
\begin{bmatrix}
\genK_i\\t
\end{bmatrix}_v$ and $\genK_{\bi}$ with $1\leq i\leq n$, $t\in \N$, and let ${{\boldsymbol U}^{+}_{v,\mc Z}}$ (resp. ${{\boldsymbol U}^{-}_{v,\mc Z}}$) be the $\mcZ$-subalgebra of $\Uvqn$ generated by $\genE_j^{(m)}$ and $ \genE_{\bj}$ (resp. $\genF^{(m)}_j, \genF_{\bj}$) with $1\leq j<n$. 
Similar to \cite[Proposition 1,7]{Lus2} and \cite[Lemma 9.3.2]{CP}, associated with Proposition \ref{action_uvqn}, we have the following.
\begin{lem}
    For any $1\leq i <n$, $T_i(\qUZ)=\qUZ$.
\end{lem}
\begin{proof}
    It suffices to verify that both $T_i$ and $T_i^{-1}$ map each generator of $\qUZ$ into $\qUZ$. The even part is shown specifically in \cite[Proposition 1,7]{Lus2}. By Proposition \ref{action_uvqn}, we have
    \begin{align*}
&T_{i} ({\genE}_{\ol{i}}) = -{\genK}_{\ol{i+1}}{\genF}_{i}^{(1)}  {\genK}_{i} +v{\genF}_{i}^{(1)}{\genK}_{\ol{i+1}}  {\genK}_{i},\quad
T_{i} ({\genF}_{\ol{i}}) = - {\genK}_{i}^{-1} {\genE}_{i}^{(1)}{\genK}_{\ol{i+1}} +v^{-1} {\genK}_{\ol{i+1}}{\genK}_{i}^{-1} {\genE}_{i}^{(1)};\\
&T_{i} ({\genE}_{\ol{j}}) =   -  {\genE}_{i}^{(1)} {\genE}_{\ol{j}}  +   {v}^{-1} {\genE}_{\ol{j}} {\genE}_{i}^{(1)}, \quad
T_{i} ({\genF}_{\ol{j}}) = -   {\genF}_{\ol{j}} {\genF}_{i}^{(1)} +   {v} {\genF}_{i}  ^{(1)}{\genF}_{\ol{j}} 
	 \quad \mbox{ for} \enspace |i - j| = 1;\\
&T_{i} ({\genE}_{\ol{j}}) = {\genE}_{\ol{j}}, \quad
	T_{i} ({\genF}_{\ol{j}}) = {\genF}_{\ol{j}} 
	\quad \mbox{ for}\enspace |i - j| > 1.
    \end{align*}
Since $T_{i} ({\genK}_{\ol{j}}  )$ is $  {\genK} _{\ol{s_{j}(i)}}$ or $(v-v^{-1}) {\genK}_{\ol{i+1}}\genF_i^{(1)}\genE_i^{(1)} -(v-v^{-1})\genF_i^{(1)}\genE_i^{(1)} {\genK}_{\ol{i+1}} +\genK_{\ol{i}}$, it follows that $T_i(\qUZ)\subset \qUZ$. Also, $T_i^{-1}(\qUZ)\subset \qUZ$ can be verified in a similar way.
\end{proof}
Hence, fix a reduced expression $w_0=s_{i_1}s_{i_2} \cdots s_{i_N}$ of the longest element in the symmetric group $\Sn$, for $m\in \N$, the divided power of root vectors can be given by
\begin{align*}    &\genE_{\beta_k}^{(m)}=T_{i_1}T_{i_2} \cdots T_{i_{k-1}}(\genE_{i_{k}}^{(m)})\in\qUZ,\quad {\ol{\genXE}}_{\beta_k}^{(m)}=T_{i_1}T_{i_2} \cdots T_{i_{k-1}}(\genE_{\ol{i_{k}}}^{(m)})\in\qUZ,\\    &\genF_{\beta_k}^{(m)}=T_{i_1}T_{i_2} \cdots T_{i_{k-1}}(\genF_{i_{k}}^{(m)})\in\qUZ,\quad {\ol{\genF}}_{\beta_k}^{(m)}=T_{i_1}T_{i_2} \cdots T_{i_{k-1}}(\genF_{\ol{i_{k}}}^{(m)})\in\qUZ.
\end{align*}


{We now recall the realization of $\Uvqn$ specified in \cite{DGLW2} to investigate the structure of the PBW-type basis for the integral form $\qUZ$.}

Let $\Lambda(n,r)=\{\lambda=(\lambda_1, \cdots, \lambda_n) \in \N^n~|~|\lambda|:=\sum_i\lambda_i=r \}$ and $\mathfrak{S}_r$ be the symmetric group on $r$ letters with generators $s_i=(i,i+1)$ for $1\leq i <r$. For any $\lambda=(\lambda_1, \cdots, \lambda_n) \in \Lambda(n,r)$, the Yong subgroup $\mathfrak{S}_{\lambda}$ is the subgroup 
\begin{align*}
    \mathfrak{S}_{\lambda}:=\langle s_i~|~1\leq i <r \backslash \{\tilde{\lambda}_j~|~1\leq j <n\}\rangle ,
\end{align*}
where $\tilde{\lambda}_j=\sum\limits_{k=1}^{j}\lambda_k$ and $\tilde{\lambda}_0 =0$. And then denote $x_\lambda=\sum_{w\in \mathfrak{S}_\lambda}T_w$.

Let $R$ be a commutative ring and $R^{\times}$ be its group of units such that $2\in R^{\times}$. Let $q\in R^{\times}$. The Hecke-Clifford superalgebra ${\mathcal{H}^c_{r,R}}$ is a superalgebra over $R$ generated by even generators $T_i:=T_{s_i}$ and odd generators $c_j$ for $1\leq i <r, 1\leq j \leq r$, and the following relations:
\begin{align*}
    &c_i^2=-1,\quad c_i c_j=-c_j c_i \quad (1\leq i, j \leq r-1);\\
    &(T_i-q)(T_i +1)=0,\quad T_i T_j= T_j T_i\quad (q\in R,~|i-j|>1),\\
    &T_i T_{i+1} T_i=T_{i+1}T_i T_{i+1}\quad (i\neq r-1);\\
    &T_ic_j=c_jT_i\quad (j\neq i,i+1);\\
    &T_ic_i=c_{i+1}T_i,\quad T_ic_{i+1}=c_iT_i-(q-1)(c_i-c_{i+1}).
\end{align*}

 Denote the parity of $z$ by $p(z)$. Following \cite[Definition 1.1]{CW}, for right $ {\mathcal{H}^c_{r,R}}$-modules $M$ and $N$, Hom$^s_{{\mathcal{H}^c_{r,R}}}(M,N)$ is defined to be the $R$-supermodule generated by all ${\mathcal{H}^c_{r,R}}$-superhomomorphisms $\Phi:M\to N$ satisfying $\Phi(mh)=(-1)^{p(\Phi)p(h)}\Phi(m)h$ with $m\in M,~h\in {\mathcal{H}^c_{r,R}}$.
Then, referring to \cite[Section 2]{DGLW2}, the standardised queer $q$-Schur superalgebra $\QqnR$ is defined as 
\begin{align*}
   \QqnR:=\text{End}^s_{\mathcal{H}^c_{r,R}} (\bigoplus_{\lambda\in \Lambda(n,r)} x_{\lambda}{\mathcal{H}^c_{r,R}})=\bigoplus_{\lambda,\mu\in\Lambda(n,r)} \text{Hom}^s_{\mathcal{H}^c_{r,R}}(x_{\lambda}\mathcal{H}^c_{r,R},x_{\mu}\mathcal{H}^c_{r,R}).
\end{align*}
And if the ground ring $R=\mcZ$ or $R=\Qv$ with $q=v^2$, by \cite[Section 3]{DGLW2}, the standardise queer $v$-Schur superalgebras are defined as
\begin{align*}
    {\mathcal{Q}}^{s}_{v,\mcZ}(n,r):={\mathcal{Q}}^{s}_{q}{({n,r;\mcZ})},\qquad 
    \Qvnr:={\mathcal{Q}}^{s}_{q}{({n,r;\Qv})}.
\end{align*}

Set $A^{\star}=(A^{\ol{0}}|A^{\ol{1}}):=\begin{pmatrix}
   A^{\ol{0}} & A^{\ol{1}}\\
   A^{\ol{1}} & A^{\ol{0}}
\end{pmatrix}\in M_n(\N|\N_2)$.
For any $A^{\star}=(A^{\ol{0}}|A^{\ol{1}})\in M_n(\N|\N_2)^{\pm}, ~\boldsymbol{j}\in\Z^n$, by \cite[Section 5]{DGLW2}, the element $A^{\star}(\boldsymbol{j},r)$ in $\Qvnr$ is defined as 
\begin{equation}\label{Ajr}
    A^{\star}(\boldsymbol{j},r)=\left\{
\begin{aligned}
&\sum_{\substack{\lambda\in\Lambda(n,r-|A^{\star}|)}}v^{\lambda\ast{\boldsymbol{j}}}[A^{\ol{0}}+\lambda~|~A^{\ol{1}}], &\text{ if } |A|\geq r;\\
&0, &\text{ otherwise },
\end{aligned}
\right.
\end{equation}
where $|A^{\star}|=\sum_{i,j}a_{i,j}$, $a_{i,j}=a^{\ol{0}}_{i,j}+a^{\ol{1}}_{i,j}$, $\lambda\ast{\boldsymbol{j}}=\sum_{i=1}^n \lambda_i j_i$, $A^{\ol{0}}+\lambda:=A^{\ol{0}}+diag(\lambda)$ .

Recall some definitions in \cite[Section 8]{DGLW2}:
\begin{align*}
  &\Qvns:=\prod_{r\geq 1} \Qvnr,\\
  & A^{\star}(\boldsymbol{j}):=\sum_{r\geq 1} A^{\star}(\boldsymbol{j},r)\in\Qvns \text{ for } A^{\star}\in M_n(\N|\N_2)^{\pm }, \boldsymbol{j}\in\Z^n,\\
  &\Avn:=\mathrm{span}_{\Qv}\{A^{\star}(\boldsymbol{j})~|~A^{\star}\in M_n(\N|\N_2)^{\pm }, \boldsymbol{j}\in\Z^n\}\subset \Qvns.
  \end{align*}

Denote by $E_{i,j}\in M_n(\N)$ the matrix with $1$ at the $(i,j)$ position and $0$ elsewhere. The following two lemmas are taken from \cite{DGLW2}.

\begin{lem}\cite[Lemma 8.1, Theorem 8.3(2)]{DGLW2}\label{DGLW2}
    The $\Qv$-space $\Avn$ is a sub-superalgebra of $\Qvns$ generated by 
    \begin{align*}
    &\{ (O|O)(\pm\boldsymbol{ \ep_i}),~ (E_{j,j+1}|O)(\boldsymbol{0}),~ (E_{j+1,j}|O)(\boldsymbol{0}),~ (O|E_{i,i})(\boldsymbol{0}),\\
   &\quad  (O|E_{j,j+1})(\boldsymbol{0}),~ (O|E_{j+1,j})(\boldsymbol{0})~|~1\le i \le n, 1\le j \le n-1\}.    
    \end{align*}
    Moreover, the set $\{ A^{\star}(\boldsymbol{j})~|~A^{\star}\in M_n(\N|\N_2)^{\pm}, \boldsymbol{j}\in{\Z}^n\}$ form a $\Qv$-basis of $\Avn$.
\end{lem}

\begin{lem}\cite[Theorem 8.3(3), Theorem 9.3]{DGLW2}\label{qiso}
    For all $\leq i\leq n,1\leq j<n$, there is a superalgebra isomorphism $\xi_n$ defined by 
    \begin{align*}
       & \xi_n : \Uvqn\to \Avn,\\
       &\genE_j\mapsto (E_{j,j+1}|O)(\boldsymbol{0}),\quad \genF_j\mapsto (E_{j+1,j}|O)(\boldsymbol{0}), \quad \genK_i^{\pm 1}\mapsto (O|O)(\pm\boldsymbol{ \ep_i}),\\
       &\genE_{\bj}\mapsto (O|E_{j,j+1})(\boldsymbol{0}),\quad \genF_{\bj}\mapsto (O|E_{j+1,j})(\boldsymbol{0}), \quad \genK_{\bi}\mapsto (O|E_{i,i})(\boldsymbol{0}).
    \end{align*}
\end{lem}

Let $A^{\star}=(a_{i,j}^{\ol{0}}|a_{i,j}^{\ol{1}})\in M_n(\N|\N_2)$. 
We now consider the image of the products of positive root vectors in superalgebra $\Avn$.
Similar to \cite[(3.4)]{DW}, the order of the product of positive root vectors is specified as follows. 
   For the $i$-th row (reading to the right) $a_{i,i+1}, a_{i,i+2},\cdots , a_{i,n}$ of the upper triangular part of the matrix $A^{\star}$, which denotes as $A^{\star}_{+}$, we put
    \begin{align*}
    &\rho_i(A^{\star}_{+}):=[({{\genXE}}_{i,i+1})^{{a}^{\ol{0}}_{i,i+1}} (\ol{{\genXE}}_{i,i+1}) ^{{a}^{\ol{1}}_{i,i+1}})]
    \cdots [({{\genXE}} _{i,n}) ^{{a}^{\ol{0}}_{i,n}} ({\ol{\genXE}} _{i,n})^{{a}^{\ol{1}}_{i,n}}],\\
    &\rho_i(A^{\star}_{(+)})=:[({{\genXE}}_{i,i+1})^{({a}^{\ol{0}}_{i,i+1})} (\ol{{\genXE}}_{i,i+1}) ^{{a}^{\ol{1}}_{i,i+1}})]
    \cdots [({{\genXE}} _{i,n}) ^{({a}^{\ol{0}}_{i,n})} ({\ol{\genXE}} _{i,n})^{{a}^{\ol{1}}_{i,n}}],
    \end{align*}
     {since the multiplication formulas for odd elements given in \cite{DGLW2} need to satisfy the SDP condition (See \cite[Theorem 4.6, Corollary 4.7]{DGLW} for more details). And also all the multiplication formulas are carried out from right to left.} Then we set 
    \begin{align*}
     &{\genXE}_{A^{\star}_{+}}:=\rho_{n-1}(A^{\star}_{+})\cdots \rho_{1}(A^{\star}_{+})= \prod_{n-1\ge i\ge 1} \prod_{i+1\leq j\leq n} ({{\genXE}}_{i, j})^{{a}^{\ol{0}}_{i, j}} (\ol{{\genXE}}_{i, j})^{{a}^{\ol{1}}_{i, j}},\\
     &{\genXE}_{A^{\star}_{(+)}}:=\rho_{n-1}(A^{\star}_{(+)})\cdots \rho_{1}(A^{\star}_{(+)})= \prod_{n-1\ge i\ge 1} \prod_{i+1\leq j\leq n} ({{\genXE}}_{i, j})^{({a}^{\ol{0}}_{i, j})} (\ol{{\genXE}}_{i, j})^{{a}^{\ol{1}}_{i, j}}.
    \end{align*}

    {For ${\tau_i}\in \N_2$, denote
    \begin{align*}
    {\genK}_{A^{\star}_0}:=\prod\limits_{ i=1}^{n} \genK_{\ol {i}}^{a_{i,i}^{\ol{1}}},\qquad 
    {\genK}_{A^{\star}_0, \tau} :=\prod\limits_{ i=1}^{n}(\genK^{\tau_i}_{i} \begin{bmatrix}
\genK_i\\{{a}^{\ol{0}}_{i, i}}
\end{bmatrix}_v \genK_{\ol {i}}^{a_{i,i}^{\ol{1}}}).
    \end{align*}}

For any $1\leq h \leq n-1$, $1\leq k\leq n$, denote 
\begin{align*}
    \tilde{a}_{h,k}=\sum\limits_{ j=1}^{k-1} \sum\limits_{ i=1}^{n}a_{i,j}+\sum\limits_{ i=1}^{h}a_{i,k}.
 \end{align*}

By the multiplication formulas given in \cite[Proposition 5.2, Proposition 5.4]{DGLW2}, we obtain the following results.
\begin{lem}\label{Eij}
Let $M=(M^{\ol{0}}|M^{\ol{1}})=(m_{k,l}^{\ol{0}}|m_{k,l}^{\ol{1}}) \in M_n(\N|\N_2)$, $1\leq i< j \leq n$, $l\in\N, t\in \N_2$. 
Assume that $m_{k,l}^{\ol{0}}=m_{k,l}^{\ol{1}}=0$ for all $k\geq l$ and $k>i$ and $k=i~\&~l<j$. Then we have
 \begin{align*}
(1)\quad 
&\genXE_{i,j}\cdot M(\boldsymbol{0})=(-1)^{j-i+1}v^{\sum_{t>j}m_{i,t}}[m_{i,j}^{\ol{0}}+1]_v(M^{\ol{0}}+E_{i,j}|M^{\ol{1}})(\boldsymbol{0}),\\
&\ol{{\genXE}}_{i, j}\cdot M(\boldsymbol{0})=(-1)^{\tilde{m}^{\ol{1}}_{j-2,j}+j-i+1} v^{\sum_{t>j}m_{i,t}}(M^{\ol{0}}|M^{\ol{1}}+E_{i,j})(\boldsymbol{0});
\\
(2)\quad 
&({{\genXE}}_{i, j})^{l}\cdot M(\boldsymbol{0})=(-1)^{l(j-i+1)}v^{l(\sum_{t>j}m_{i,t})} \prod\limits_{ u=1}^{l}[m_{i,j}^{\ol{0}}+u]_v
(M^{\ol{0}}+lE_{i,j}|M^{\ol{1}})(\boldsymbol{0}),\\
&(\ol{{\genXE}}_{i, j})^t\cdot M(\boldsymbol{0})=(-1)^{t(\tilde{m}^{\ol{1}}_{j-2,j}+j-i+1)} v^{t(\sum_{t>j}m_{i,t})}(M^{\ol{0}}|M^{\ol{1}}+tE_{i,j})(\boldsymbol{0}).
\end{align*}
\end{lem}
\begin{proof}
   Assume that $1\leq i<j\leq n$. We prove the first formula in (1) by induction on $j-i$. Indeed, if $j=i+1$, then by Lemma \ref{qiso} and the formula in \cite[Proposition 5.2(1)]{DGLW2}, we have
   \begin{align*}
       \genXE_{i,i+1}\cdot M(\boldsymbol{0})=(E_{i,i+1}|O)(\boldsymbol{0})\cdot M(\boldsymbol{0})=v^{\sum_{t>i+1}m_{i,t}}[m^{\ol{0}}_{i,i+1}+1](M^{\ol{0}}+E_{i,i+1}|M^{\ol{1}})(\boldsymbol{0}).
   \end{align*}
Now suppose that $j> i+1$, then by induction and Corollary \ref{rv}, we have
\begin{align*}
    \genXE_{i,j}\cdot M(\boldsymbol{0})
    &=(-\genE_{i,i+1} \genE_{i+1,j}+v^{-1}\genE_{i+1,j} \genE_{i,i+1})\cdot M(\boldsymbol{0})\\
    &=-\genE_{i,i+1} \genE_{i+1,j}\cdot M(\boldsymbol{0})+v^{-1}\genE_{i+1,j} \genE_{i,i+1}\cdot M(\boldsymbol{0})\\
    &=(-1)^{j-i+1}v^{\sum_{t>j}m_{i+1,t}}\genE_{i,i+1}\cdot(M^{\ol{0}}+E_{i+1,j}|M^{\ol{1}})(\boldsymbol{0})\\
    &\quad+v^{\sum_{t>i+1}m_{i,t}-1}[m^{\ol{0}}_{i,i+1}+1]_v\genE_{i+1,j}\cdot 
    (M^{\ol{0}}+E_{i,i+1}|M^{\ol{1}})(\boldsymbol{0})\\
    &=(-1)^{j-i+1}v^{\sum_{t>i+1}m_{i,t}-1}[m^{\ol{0}}_{i,i+1}+1]_v(M^{\ol{0}}+E_{i,i+1}+E_{i+1,j}|M^{\ol{1}})(\boldsymbol{0})\\
    &\quad+(-1)^{j-i+1}v^{\sum_{t>j}m_{i,t}}[m^{\ol{0}}_{i,j}+1]_v(M^{\ol{0}}+E_{i,j}|M^{\ol{1}})(\boldsymbol{0})\\
    &\quad +(-1)^{j-i}v^{\sum_{t>i+1}m_{i,t}-1}[m^{\ol{0}}_{i,i+1}+1]_v(M^{\ol{0}}+E_{i,i+1}+E_{i+1,j}|M^{\ol{1}})(\boldsymbol{0})\\
    &=(-1)^{j-i+1}v^{\sum_{t>j}m_{i,t}}[m^{\ol{0}}_{i,j}+1]_v(M^{\ol{0}}+E_{i,j}|M^{\ol{1}})(\boldsymbol{0})
\end{align*}
as desired. Applying  the formula in \cite[Proposition 5.4]{DGLW2}, the proof for the second equation in part (1) is similar. As for part (2), (1) proves in the case $l=1$. By induction on $l$, we have 
\begin{align*}
    \genXE^{l+1}_{i,j}\cdot M(\boldsymbol{0})
    &=(-1)^{l(j-i+1)}v^{l(\sum_{t>j}m_{i,t})} \prod\limits_{ u=1}^{l}[m_{i,j}^{\ol{0}}+u]_v
\genXE_{i,j}\cdot(M^{\ol{0}}+lE_{i,j}|M^{\ol{1}})(\boldsymbol{0})\\
&=(-1)^{(l+1)(j-i+1)}v^{(l+1)(\sum_{t>j}m_{i,t})} \prod\limits_{ u=1}^{l+1}[m_{i,j}^{\ol{0}}+u]_v
\cdot(M^{\ol{0}}+(l+1)E_{i,j}|M^{\ol{1}})(\boldsymbol{0}),
\end{align*}
where the last equality is due to part (1). By a parallel argument, the remaining formula in part (2) can be verified.
\end{proof}

Then by Lemma \ref{Eij}, we now derive explicit expressions for $\rho_i(A^{\star}_+)$ and ${\genXE}_{A^{\star}_{+}}$ in superalgebra $\Avn$.

\begin{lem}
Assume that $1\leq i \leq n-1$. Then we have
\begin{align*}
(1)\quad 
&\rho_i(A^{\star}_+)=g_i(v) (\sum\limits_{ j=i+1}^{n}{a}^{\ol{0}}_{i,j}E_{i,j}|\sum\limits_{ j=i+1}^{n}{a}^{\ol{1}}_{i,j}E_{i,j})(\boldsymbol{0}),
\\
(2)\quad 
&{\genXE}_{A^{\star}_{+}}=g(v)(\sum\limits_{ i=1}^{n-1}\sum\limits_{ j=i+1}^{n}{a}^{\ol{0}}_{i,j}E_{i,j}|\sum\limits_{ i=1}^{n-1}\sum\limits_{ j=i+1}^{n}{a}^{\ol{1}}_{i,j}E_{i,j})(\boldsymbol{0}),
\end{align*}
where 
\begin{align*}
&g_i(v)=(-1)^{\sum\limits_{ j=i+1}^{n}(j-i+1){a}_{i,j}}v^{\sum\limits_{j=i+1}^{n-1}a_{i,j}(\sum\limits_{u=j+1}^{n}a_{i,u})}\prod_{j=i+1}^{n}[a_{i,j}^{\ol{0}}]^!_v,\\
&g(v)=(-1)^{{\sum\limits_{i=2}^{n-1}\sum\limits_{j=i+1}^{n}a^{\ol{1}}_{i,j}(\sum\limits_{u=1}^{i-1}\sum\limits_{v=u+1}^{j}a^{\ol{1}}_{u,v})}+\sum\limits_{i=1}^{n-1}\sum\limits_{j=i+1}^{n}(j-i+1)a_{i,j}} v^{\sum\limits_{i=1}^{n-1}\sum\limits_{j=i+1}^{n-1}a_{i,j}(\sum\limits_{u=j+1}^{n}a_{i,u})}\prod_{i=1}^{n-1}\prod_{j=i+1}^{n}[a_{i,j}^{\ol{0}}]^!_v.
\end{align*}
\end{lem}
\begin{proof}
    By Lemma \ref{Eij}, we see that
    \begin{align*}
&({{\genXE}}_{i, n})^{a^{\ol{0}}_{i,n}}=(-1)^{{a^{\ol{0}}_{i,n}}(n-i+1)} [a_{i,n}^{\ol{0}}]^!_v\cdot
({a^{\ol{0}}_{i,n}}E_{i,n}|O)(\boldsymbol{0}),\\
&(\ol{{\genXE}}_{i, n})^{a^{\ol{1}}_{i,n}}=(-1)^{{a^{\ol{1}}_{i,n}}(n-i+1)} (O|{a^{\ol{1}}_{i,n}}E_{i,n})(\boldsymbol{0}),  
    \end{align*}
 then we have
 \begin{align*}
   ({{\genXE}}_{i, n})^{a^{\ol{0}}_{i,n}} (\ol{{\genXE}}_{i, n})^{a^{\ol{1}}_{i,n}}= (-1)^{{a_{i,n}}(n-i+1)} [a_{i,n}^{\ol{0}}]^!_v({a^{\ol{0}}_{i,n}}E_{i,n}|{a^{\ol{1}}_{i,n}}E_{i,n})(\boldsymbol{0}).
 \end{align*}
 For the fixed $i$-th row, we now prove the first formula by induction on the number of columns $j~(i+1\leq j<n)$. Lemma \ref{Eij} implies
 {\small\begin{align*} &[({{\genXE}}_{i,j})^{{a}^{\ol{0}}_{i,j}} (\ol{{\genXE}}_{i,j}) ^{{a}^{\ol{1}}_{i,j}})]
    \cdots [({{\genXE}} _{i,n}) ^{{a}^{\ol{0}}_{i,n}} ({\ol{\genXE}} _{i,n})^{{a}^{\ol{1}}_{i,n}}]\\
    &\quad=(-1)^{\sum\limits_{ u=j+1}^{n}(u-i+1){a}_{i,u}}v^{\sum\limits_{u=j+1}^{n-1}a_{i,u}(\sum\limits_{t=u+1}^{n}a_{i,t})}\prod_{u=j+1}^{n}[a_{i,u}^{\ol{0}}]^!_v\\
    &\qquad\quad\cdot    ({{\genXE}}_{i,j})^{{a}^{\ol{0}}_{i,j}} (\ol{{\genXE}}_{i,j}) ^{{a}^{\ol{1}}_{i,j}}\cdot
    (\sum\limits_{ u=j+1}^{n}{a}^{\ol{0}}_{i,u}E_{i,u}|\sum\limits_{ u=j+1}^{n}{a}^{\ol{1}}_{i,u}E_{i,u})(\boldsymbol{0})\\
    &\quad=(-1)^{\sum\limits_{ u=j+1}^{n}(u-i+1){a}_{i,u}+(j-i+1){a}^{\ol{1}}_{i,j}}v^{\sum\limits_{u=j+1}^{n-1}a_{i,u}(\sum\limits_{t=u+1}^{n}a_{i,t})+{a}^{\ol{1}}_{i,j}(\sum\limits_{t=j+1}^{n}a_{i,t})}\prod_{u=j+1}^{n}[a_{i,u}^{\ol{0}}]^!_v\\
    &\qquad\quad\cdot ({\genXE}_{i,j})^{{a}^{\ol{0}}_{i,j}}\cdot (\sum\limits_{ u=j+1}^{n}{a}^{\ol{0}}_{i,u} E_{i,u}|\sum\limits_{ u=j}^{n}{a}^{\ol{1}}_{i,u}E_{i,u})(\boldsymbol{0})\\
    &=(-1)^{\sum\limits_{ u=j}^{n}(u-i+1){a}_{i,u}}v^{\sum\limits_{u=j}^{n-1} a_{i,u}(\sum\limits^{n}_{t=u+1}a_{i,t})}\prod_{u=j}^{n}[a_{i,u}^{\ol{0}}]^!_v\cdot (\sum\limits_{ u=j}^{n}{a}^{\ol{0}}_{i,u}E_{i,u}|\sum\limits_{ u=j}^{n}{a}^{\ol{1}}_{i,u}E_{i,u})(\boldsymbol{0}).
 \end{align*}}
 As for the second formula, we prove it by induction on the number of rows $i$ ($1\leq i \leq n-1$). Then applying Lemma \ref{Eij} and by induction, we have
 \begin{align*}
   &\rho_{i}(A^{\star}_{+})\rho_{i-1}(A^{\star}_{+})\cdots \rho_{1}(A^{\star}_{+})\\
   &\qquad=g'_{i-1}(v) \rho_{i}(A^{\star}_{+})\cdot (\sum\limits_{ t=1}^{i-1}\sum\limits_{ j=t+1}^{n}{a}^{\ol{0}}_{t,j}E_{t,j}|\sum\limits_{ t=1}^{i-1}\sum\limits_{ j=t+1}^{n}{a}^{\ol{1}}_{t,j}E_{t,j})(\boldsymbol{0})\\
   &\qquad=g'_{i-1}(v)[({{\genXE}}_{i,i+1})^{{a}^{\ol{0}}_{i,i+1}} (\ol{{\genXE}}_{i,i+1}) ^{{a}^{\ol{1}}_{i,i+1}}]
    \cdots [({{\genXE}} _{i,n}) ^{{a}^{\ol{0}}_{i,n}} ({\ol{\genXE}} _{i,n})^{{a}^{\ol{1}}_{i,n}}]\cdot (\sum\limits_{ t=1}^{i-1}\sum\limits_{ j=t+1}^{n}{a}^{\ol{0}}_{t,j}E_{t,j}|\sum\limits_{ t=1}^{i-1}\sum\limits_{ j=t+1}^{n}{a}^{\ol{1}}_{t,j}E_{t,j})(\boldsymbol{0})\\
    &\qquad=(-1)^{\sum\limits_{j=i+1}^{n}a^{\ol{1}}_{i,j}(\sum\limits_{u=1}^{i-1}\sum\limits_{v=u+1}^{j}a^{\ol{1}}_{u,v})} g'_{i-1}(v) g_i(v)
    \cdot (\sum\limits_{ t=1}^{i}\sum\limits_{ j=t+1}^{n}{a}^{\ol{0}}_{t,j}E_{t,j}|\sum\limits_{ t=1}^{i}\sum\limits_{ j=t+1}^{n}{a}^{\ol{1}}_{t,j}E_{t,j})(\boldsymbol{0})\\
    &\qquad=g'_i(v) (\sum\limits_{ t=1}^{i}\sum\limits_{ j=t+1}^{n}{a}^{\ol{0}}_{t,j}E_{t,j}|\sum\limits_{ t=1}^{i}\sum\limits_{ j=t+1}^{n}{a}^{\ol{1}}_{t,j}E_{t,j})(\boldsymbol{0}),
 \end{align*}
 where 
 \begin{align*}
     &g'_{i-1}(v)=(-1)^{{\sum\limits_{t=2}^{i-1}\sum\limits_{j=t+1}^{n}a^{\ol{1}}_{t,j}(\sum\limits_{u=1}^{t-1}\sum\limits_{v=u+1}^{j}a^{\ol{1}}_{u,v})}+\sum\limits_{t=1}^{i-1}\sum\limits_{j=t+1}^{n}(j-t+1)a_{t,j}} v^{\sum\limits_{t=1}^{i-1}\sum\limits_{j=t+1}^{n-1}a_{t,j}(\sum\limits_{u=j+1}^{n}a_{t,u})}\prod_{t=1}^{i-1}\prod_{j=t+1}^{n}[a_{t,j}^{\ol{0}}]^!_v,\\
     &g_i(v)=(-1)^{\sum\limits_{ j=i+1}^{n}(j-i+1){a}_{i,j}}v^{\sum\limits_{j=i+1}^{n-1}a_{i,j}(\sum\limits_{u=j+1}^{n}a_{i,u})}\prod_{j=i+1}^{n}[a_{i,j}^{\ol{0}}]^!_v,\\
     &g'_i(v)=(-1)^{{\sum\limits_{t=2}^{i}\sum\limits_{j=t+1}^{n}a^{\ol{1}}_{t,j}(\sum\limits_{u=1}^{t-1}\sum\limits_{v=u+1}^{j}a^{\ol{1}}_{u,v})}+\sum\limits_{t=1}^{i}\sum\limits_{j=t+1}^{n}(j-t+1)a_{t,j}} v^{\sum\limits_{t=1}^{i}\sum\limits_{j=t+1}^{n-1}a_{t,j}(\sum\limits_{u=j+1}^{n}a_{t,u})}\prod_{t=1}^{i}\prod_{j=t+1}^{n}[a_{t,j}^{\ol{0}}]^!_v.
 \end{align*}
 This completes the proof of the lemma.
\end{proof}

   As for the image of the products of negative root vectors in $\Avn$, we need the following notations.   
   Recalling the order $\preceq$ from \cite{BLM}: for any $A=(a_{i,j}), B=(b_{i,j})\in M_n(\mathbb{N})$
\begin{equation}\label{bj}
\begin{aligned}
   B\preceq A
   \iff
\begin{cases}
   \sum_{i\le s, j\ge t}b_{i,j}\le \sum_{i\le s, j\ge t}a_{i,j} 
   &\mbox{for all}~ s<t; \\
   \sum_{i\ge s, j\le t}b_{i,j}\le \sum_{i\ge s, j\le t} a_{i,j}
   &\mbox{for all}~ s>t. \\
\end{cases}
\end{aligned}
\end{equation}
   And we say $B\prec A$ if $B\preceq A$ and {$B^{\pm}\ne A^{\pm}$}, where $X^{\pm}$ is the matrix obtained from $X$ with all diagonal entries replaced by $0$.

    In the following results, for $A^{\star}=(A^{\ol{0}}|A^{\ol{1}}) \in M_n(\N|\N_2)^{\pm}$, the ``lower terms" in an expression of the form ``$A(\boldsymbol{j})+\mbox{lower terms}$" means a $\Qv$-linear combination of finitely many elements of the form $B(\boldsymbol{p})$ with $\boldsymbol{p} \in \mathbb{Z}^{n}, B \in M_n(\N|\N_2)^{\pm}$ and $B\prec A$.

Similarly, for the $i$-th column (reading downwards) $a_{i+1,i},  a_{i+2,i},\cdots , a_{n,i}$ of the lower triangular part of the matrix $A^{\star}$, which denotes as $A^{\star}_{-}$, we put 
  \begin{align*}
    &\rho_i(A^{\star}_{-}):=[({{\genXE}}_{i+1, i})^{{a}^{\ol{0}}_{i+1, i}} (\ol{{\genXE}}_{i+1, i})^{{a}^{\ol{1}}_{i+1, i}}]
    \cdots
[({{\genXE}}_{n,i})^{{a}^{\ol{0}}_{n,i}} ({\ol{\genXE}}_{n,i})^{{a}^{\ol{1}}_{n,i}}],\\
   &\rho_i(A^{\star}_{(-)})=[({{\genXE}}_{i+1, i})^{({a}^{\ol{0}}_{i+1, i})} (\ol{{\genXE}}_{i+1, i})^{{a}^{\ol{1}}_{i+1, i}}]
    \cdots
[({{\genXE}}_{n,i})^{({a}^{\ol{0}}_{n,i})} ({\ol{\genXE}}_{n,i})^{{a}^{\ol{1}}_{n,i}}],
  \end{align*}
   and define  
   \begin{align*}
    &{\genXE}_{A^{\star}_-}:=\rho_{1}(A^{\star}_-)\cdots \rho_{n-1}(A^{\star}_-)= \prod_{1\le i\le n-1} \prod_{i+1\le j\le n}({{\genXF}}_{j, i})^{{a}^{\ol{0}}_{j, i}} (\ol{{\genXF}}_{j, i})^{{a}^{\ol{1}}_{j, i}},\\
    &{\genXE}_{A^{\star}_{(-)}}:=\rho_{1}(A^{\star}_{(-)})\cdots \rho_{n-1}(A^{\star}_{(-)})= \prod_{1\le i\le n-1} \prod_{i+1\le j\le 1}({{\genXF}}_{j, i})^{({a}^{\ol{0}}_{j, i})} (\ol{{\genXF}}_{j, i})^{{a}^{\ol{1}}_{j, i}}.
   \end{align*}

    Referring to the multiplication formulas in \cite[Proposition 5.2(2), Proposition 5.5]{DGLW2}, we have the following conclusion.

\begin{lem}\label{Eji}
Let $M=(m_{k,l}^{\ol{0}}|m_{k,l}^{\ol{1}}) \in M_n(\N|\N_2)$, $1\leq i< j \leq n$, $l\in\N, t\in\N_2$.  
Assume that $m_{k,l}^{\ol{0}}=m_{k,l}^{\ol{1}}=0$ for all $k\leq l$ and $l<i$ and $l=i~\&~ k<j$. Then we have
 \begin{align*}
(1)\quad 
& {{\genXF}}_{j,i}\cdot M(\boldsymbol{0})=(-1)^{j-i+1}[m_{j,i}^{\ol{0}}+1]_v(M^{\ol{0}}+E_{j,i}|M^{\ol{1}})(\boldsymbol{0})+\text{ lower terms},\\
&{\ol{{\genXF}}_{j,i}\cdot M(\boldsymbol{0})=(-1)^{j-i+1}(M^{\ol{0}}|M^{\ol{1}}+E_{j,i})(\boldsymbol{0})}+\text{ lower terms};
\\
(2)\quad 
& ({{\genXF}}_{j,i})^{l}\cdot M(\boldsymbol{0})=(-1)^{l(j-i+1)}\prod\limits_{ u=1}^{l}[m_{j,i}^{\ol{0}}+u]_v
(M^{\ol{0}}+l E_{j,i}|M^{\ol{1}})(\boldsymbol{0})+\text{ lower terms},\\
&{(\ol{{\genXF}}_{j,i})^t\cdot M(\boldsymbol{0})=(-1)^{t(j-i+1)}(M^{\ol{0}}|M^{\ol{1}}+tE_{j,i})(\boldsymbol{0})}+\text{ lower terms}.
\end{align*}
\end{lem}
\begin{proof}
    Suppose that $1\leq i < j\leq n$. Only the second formula in (1) is proved here by induction on $j-i$; the first formula follows the same lines. Indeed, if $j=i+1$, then by Lemma \ref{qiso} and the formula in \cite[Proposition 5.5]{DGLW2}, we have
    \begin{align*}
        {\ol{\genXE}}_{i+1,i}\cdot M(\boldsymbol{0})=(O|E_{i+1,i})(\boldsymbol{0})\cdot M(\boldsymbol{0})
        =(-1)^{\tilde{m}^{\ol{1}}_{i,i}}(M^{\ol{0}}|M^{\ol{1}}+E_{j,i})(\boldsymbol{0})=(M^{\ol{0}}|M^{\ol{1}}+E_{j,i})(\boldsymbol{0}).
    \end{align*}
    Now assume that $j>i+1$, then associated with \cite[Lemma 6.5]{GLL}, by induction and Corollary \ref{rv_k}, we have 
    \begin{align*}
     {\ol{\genXE}}_{j,i}\cdot M(\boldsymbol{0})
     &=(-{\ol{\genXE}}_{j,j-1} \genXE_{j-1,i} +v \genXE_{j-1,i} {\ol{\genXE}}_{j,j-1}) \cdot M(\boldsymbol{0})\\
     &=-{\ol{\genXE}}_{j,j-1} \genXE_{j-1,i}\cdot M(\boldsymbol{0})+v \genXE_{j-1,i} {\ol{\genXE}}_{j,j-1} \cdot M(\boldsymbol{0})\\
     &=-{\ol{\genXE}}_{j,j-1}[(-1)^{j-i}(M^{\ol{0}}+E_{j-1,i}|M^{\ol{1}})(\boldsymbol{0})+\text{ lower terms}]\\
     &\qquad+v\genXE_{j-1,i}[(-1)^{\tilde{m}^{\ol{1}}_{j-1,i+1}}v^{-m_{j,i}-m^{\ol{1}}_{j-1,i+1}}(M^{\ol{0}}-E_{j-1,i+1}|M^{\ol{1}}+E_{j,i+1})(\boldsymbol{0})\\
      &\qquad\quad (-1)^{\tilde{m}^{\ol{1}}_{j-1,i+1}}v^{-m_{j,i}+m^{\ol{0}}_{j-1,i+1}}[m^{\ol{0}}_{j,i+1}]_v(M^{\ol{0}}+E_{j,i+1}|M^{\ol{1}}-E_{j-1,i+1})(\boldsymbol{0})\\
      &\qquad\quad
     +\text{ lower terms}]\\
     &=(-1)^{\tilde{m}^{\ol{1}}_{j-1,i}+j-i+1}(M^{\ol{0}}|M^{\ol{1}}+E_{j,i})(\boldsymbol{0})\\
     &\qquad
     +(-1)^{\tilde{m}^{\ol{1}}_{j-1,i+1}+j-i}v^{-m_{j,i}-m^{\ol{1}}_{j-1,i+1}+1}(M^{\ol{0}}-E_{j-1,i+1}+E_{j-1,i}|M^{\ol{1}}+E_{j,i+1})(\boldsymbol{0})\\
     &\qquad+(-1)^{\tilde{m}^{\ol{1}}_{j-1,i+1}+j-i}v^{-m_{j,i}+m^{\ol{0}}_{j-1,i+1}+1}[m^{\ol{0}}_{j,i+1}]_v(M^{\ol{0}}+E_{j,i+1}+E_{j-1,i}|M^{\ol{1}}-E_{j-1,i+1})(\boldsymbol{0})\\
     &\qquad
     +\text{ lower terms}\\
     &=(-1)^{j-i+1}(M^{\ol{0}}|M^{\ol{1}}+E_{j,i})(\boldsymbol{0})+\text{ lower terms} 
    \end{align*}
    as desired, where the third and the forth equations are due to the first formula in part (1). As for part (2), (1) proves in the case $l=1$. By induction on $l$, we have 
    \begin{align*}
      ({{\genXF}}_{j,i})^{l+1}\cdot M(\boldsymbol{0})&={\genXF}_{j,i}\cdot
      [(-1)^{l(j-i+1)}\prod\limits_{ u=1}^{l}[m_{j,i}^{\ol{0}}+u]_v (M^{\ol{0}}+l E_{j,i}|M^{\ol{1}})(\boldsymbol{0})+\text{ lower terms}]  \\
      &=(-1)^{(l+1)(j-i+1)}\prod\limits_{ u=1}^{l+1}[m_{j,i}^{\ol{0}}+u]_v
(M^{\ol{0}}+(l+1) E_{j,i}|M^{\ol{1}})(\boldsymbol{0})+\text{ lower terms},
    \end{align*}
  where the last equality is due to part (1).  By a parallel argument, the remaining formula in part (2) can be verified.
\end{proof}

\begin{cor}
Assume that $2\leq j \leq n$. Then we have
\begin{align*}
(1)\quad 
&\rho_i(A^{\star}_-)= f_i(v)(M^{\ol{0}}+\sum\limits_{j=i+1}^{n}E_{j,i}|M^{\ol{1}}+\sum\limits_{j=i+1}^{n}E_{j,i})(\boldsymbol{0})+\text{ lower terms},
\\
(2)\quad 
&{\genXE}_{A^{\star}_{-}}= f(v)(M^{\ol{0}}+\sum\limits_{i=1}^{n-1}\sum\limits_{j=i+1}^{n}E_{j,i}|M^{\ol{1}}+\sum\limits_{i=1}^{n-1}\sum\limits_{j=i+1}^{n}E_{j,i})(\boldsymbol{0})+\text{ lower terms},
\end{align*}
where
\begin{align*}
    &f_i(v)=(-1)^{\sum\limits_{j=i+1}^{n} (j-i+1)a_{j,i}}\prod\limits_{j=i+1}^{n}[a_{j,i}^{\ol{0}}]_v !,\\
    &f(v)=(-1)^{\sum\limits_{i=1}^{n-1}\sum\limits_{j=i+1}^{n} (j-i+1)a_{j,i}}\prod\limits_{i=1}^{n-1}\prod\limits_{j=i+1}^{n}[a_{j,i}^{\ol{0}}]_v !.
\end{align*}
\end{cor}

Hence, we have previously discussed the upper and lower triangular parts; now, for the diagonal part, the proof of \cite[Proposition 8.2(1)]{DW} show that  $\{{\genK}_{A^{\star}_0, \tau}~|~A^{\star}\in M_n(\N|\N_2)^{\pm},\tau\in \N_2^n\}$ spans ${{\boldsymbol U}^{0}_{v,\mc Z}}$. Then associated with the triangular decomposition 
       $$
   \qUZ \cong {{\boldsymbol U}^{-}_{v,\mc Z}} \otimes {{\boldsymbol U}^{0}_{v,\mc Z}} \otimes {{\boldsymbol U}^{+}_{v,\mc Z}}.
    $$
    by \cite[Proposition 8.2(2)]{DW}, we have proved the following theorem.
\begin{thm}\label{basis}
The following hold in $\qUZ$:
   \begin{enumerate}
       \item The set $\{ \genXE_{A^{\star}_{(+)}}~|~A^{\star}\in M_n(\N|\N_2)^{\pm}\}$ forms a $\mcZ$-basis of ${{\boldsymbol U}^{+}_{v,\mc Z}}$;
       \item the set $\{ \genXE_{A^{\star}_{(-)}}~|~A^{\star}\in M_n(\N|\N_2)^{\pm}\}$ forms a $\mcZ$-basis of ${{\boldsymbol U}^{-}_{v,\mc Z}}$;
       \item the set $\{\genXE_{A^{\star}_{(-)}}\cdot{\genK}_{A^{\star}_0, \tau}\cdot \genXE_{A^{\star}_{(+)}}~|~ A^{\star}\in M_n(\N|\N_2)^{\pm},\tau\in \N_2^n \}$ form a $\mcZ$-basis of $\qUZ$.
   \end{enumerate} 
\end{thm}

 Moreover, by Lemma \ref{DGLW2}, in superalgebra $\Avn$, we show that
 \begin{equation}\label{Aj}
 \begin{aligned}
\genXE_{A^{\star}_{-}}\cdot{\genK}_{A^{\star}_0}\cdot\prod_{i=1}^n\genK_i^{\sigma_i}\cdot \genXE_{A^{\star}_{+}}= tA^{\star}(\boldsymbol{j})+\sum\limits_{B^{\star}\prec A^{\star}\atop\boldsymbol{j'}\in \Z^n} t_{B^{\star},\boldsymbol{j'}}B^{\star}(\boldsymbol{j'}),
 \end{aligned}
  \end{equation}
where $A^{\star}\in M_n(\N|\N_2)^{\pm}$, $\boldsymbol{j}\in \Z^n$, $\sigma_i \in \N$, $t,t_{B^{\star},\boldsymbol{j'}}\in \Qv$, $t\neq 0$.

Hence, following from \cite[Definition-Proposition 9.2.1]{CP}, we know that $\qUZ$ is a non-restricted integral form of $\Uvqn$. 

 \section{The homomorphism between $\qUZ$ and ${\mathcal{Q}}^{s}_{v,\mcZ}(n,r)$} \label{surjective}
 
In this section, our aim is to discuss that the superalgebra homomorphism between $\qUZ$ and ${\mathcal{Q}}^{s}_{v,\mcZ}(n,r)$ is surjective.

Following \cite[Theorem 6.3]{DGLW2}, there is a superalgebra homomorphis between the quantum queer superalgebra $\Uvqn$ and the standardised queer $v$-Schur superalgebra $\Qvnr$ given by
\begin{align*}
   & \boldsymbol{\xi}_{n,r}:\Uvqn\to \Qvnr \\
   &\genE_j\mapsto (E_{j,j+1}|O)(\boldsymbol{0},r),\quad \genF_j \mapsto (E_{j+1,j}|O)(\boldsymbol{0},r),\quad \genK_i^{\pm 1}\mapsto (O|O)(\pm\boldsymbol{\ep_i},r),\\
   &\genE_{\ol{j}}\mapsto (O|E_{j,j+1})(\boldsymbol{0},r),\quad \genF_{\ol{j}} \mapsto (O|E_{j+1,j})(\boldsymbol{0},r),\quad \genK_{\ol{i}}\mapsto (O|E_{i,i})(\boldsymbol{0},r).
\end{align*}

Denote the function $\N$ valued $\lVert\cdot \lVert$ in $M_n(\N|\N_2)$ by setting, for any $A^{\star}=(a^{\ol{0}}_{i,j}|a^{\ol{1}}_{i,j})\in M_n(\N|\N_2)$,
\begin{align*}
\lVert A^{\star} \lVert=\sum\limits_{1\leq i< j\leq n} \frac{(j-i)(j-i+1)}{2} (a^{\ol{0}}_{i,j}+a^{\ol{1}}_{i,j}+a^{\ol{0}}_{j,i}+a^{\ol{1}}_{j,i}).  
\end{align*}

Let ${\xi}_{n,r}$ be the restriction of the map $\boldsymbol{\xi}_{n,r}$ to the $\mcZ$-subsuperalgebra $\qUZ$ of $\Uvqn$. We now consider the image of $\begin{bmatrix}
\genK_i\\t
\end{bmatrix},~ \genE_j^{(m)},~ \genF^{(m)}_j$ with
 $1\leq i\leq n, 1\leq j\leq n-1,t,m\in\N$.
 
\begin{lem}\label{image}
    For any $1\leq i \leq n, 1\leq j\leq n-1$, $t,m\in \N$, we have 
    \begin{align*}
&(1)~{\xi}_{n,r}(\begin{bmatrix}
\genK_i\\t
\end{bmatrix}_v )=\sum\limits_{\lambda\in \Lambda(n,r)} \begin{bmatrix}
\lambda_i\\t
\end{bmatrix}_v [\lambda|O],\\
&(2)~{\xi}_{n,r}(\genE_j^{(m)})=\sum\limits_{\lambda\in \Lambda(n,r-m)}[\lambda+mE_{j,j+1}|O],\\
&(3)~{\xi}_{n,r}(\genF_j^{(m)})=\sum\limits_{\lambda\in \Lambda(n,r-m)}[\lambda+mE_{j+1,j}|O].
    \end{align*}
\end{lem}
\begin{proof}
   Suppose $t\in\N$. By \eqref{Ajr} and the definition of $\boldsymbol{\xi}_{n,r}$,  we have $\boldsymbol{\xi}_{n,r}(\genK_i^{\pm 1})=\sum\limits_{\lambda\in \Lambda(n,r)}v^{\pm\lambda_i}[\lambda|O]$. With the notation introduced in \eqref{Kt} and the argument afterwards,
   \begin{align*}
{\xi}_{n,r}(\begin{bmatrix}
\genK_i\\t
\end{bmatrix}_v )
&={\xi}_{n,r}(\prod\limits_{j=1}^t \frac{\genK_iv^{1-j}-\genK_i^{-1}v^{-1+j}}{v^j-v^{-j}})
\\&=\prod\limits_{j=1}^{t}\frac{\sum\limits_{\lambda\in \Lambda(n,r)}v^{\lambda_i}[\lambda|O]\cdot v^{1-j}-\sum\limits_{\lambda\in \Lambda(n,r)}v^{-\lambda_i}[\lambda|O]\cdot v^{-1+j}}{v^j-v^{-j}}
\\&=\prod\limits_{j=1}^{t} \sum\limits_{\lambda\in \Lambda(n,r)} \frac{v^{\lambda_i+1-j}-v^{-\lambda_i-1+j}}{v^j-v^{-j}}[\lambda|O]
\\&=\sum\limits_{\lambda\in \Lambda(n,r)}\prod\limits_{j=1}^{t}\frac{v^{\lambda_i+1-j}-v^{-\lambda_i-1+j}}{v^j-v^{-j}}[\lambda|O]
\\&=\sum\limits_{\lambda\in \Lambda(n,r)}\begin{bmatrix}
\lambda_i\\t
\end{bmatrix}_v [\lambda|O].
   \end{align*}
  Also, we prove the second formula in the Lemma by induction on $m$. Indeed, if $m=2$, then
\begin{align*}
\xi_{n,r}(\genE_j^{(2)})= \frac{(E_{j,j+1}|O)(\boldsymbol{0},r)\cdot(E_{j,j+1}|O)(\boldsymbol{0},r)}{[2]_v !}=(2E_{j,j+1}|O)(\boldsymbol{0},r)
=\sum\limits_{\lambda\in \Lambda(n,r-2)}[\lambda+2E_{j,j+1}|O]
\end{align*}
by \cite[Proposition 5.2(1)]{DGLW2} and \eqref{Ajr}. Now assume $m>2$, then by induction and \eqref{Ajr}, we have
\begin{align*}
    \xi_{n,r}(\genE_j^{(m)})
    &=\xi_{n,r}(\frac{\genE_j\cdot\genE_j^{(m-1)}}{[m]_v})\\
    &=\frac{(E_{j,j+1}|O)(\boldsymbol{0},r)}{[m]_v}\cdot \sum\limits_{\lambda\in \Lambda(n,r-m+1)}[\lambda+(m-1) E_{j,j+1}|O]\\
    &=\frac{(E_{j,j+1}|O)(\boldsymbol{0},r)\cdot ((m-1)E_{j,j+1}|O)(\boldsymbol{0},r)}{[m]_v}\\
    &=(mE_{j,j+1}|O)(\boldsymbol{0},r)\\
    &=\sum\limits_{\lambda\in \Lambda(n,r-m)}[\lambda+mE_{j,j+1}|O]
\end{align*}
as desired. By a parallel argument, the remaining formula can be proved.
\end{proof}

\begin{cor}\label{Klambda}
    For any $\lambda=(\lambda_1,\cdots,\lambda_n)\in \Lambda(n,r)$, the following holds
    \begin{align*}
\xi_{n,r}(\prod\limits_{i=1}^{n}\begin{bmatrix}
\genK_i\\ \lambda_i
\end{bmatrix}_v ) =[\lambda|O].
    \end{align*}
\end{cor}

We now determine the image $\xi_{n,r}(\qUZ)$.
\begin{thm}
    We have $\xi_{n,r}(\qUZ)={\mathcal{Q}}^{s}_{v,\mcZ}(n,r)$.
\end{thm}
\begin{proof}
    Since $\qUZ$ is a $\mcZ$-subsuperalgebra of $\Uvqn$ generated by $\genK_i^{\pm1},~
\begin{bmatrix}
\genK_i\\t
\end{bmatrix},~ \genE_j^{(m)},~ \genF^{(m)}_j,$ $ \genK_{\bi},~ \genE_{\bj},~ \genF_{\bj}$, for $1\leq i\leq n, 1\leq j\leq n-1$, $t,m\in\N$. By the definition of $\boldsymbol{\xi}_{n,r}$ and Lemma \ref{image}, the images of generators all lie in ${\mathcal{Q}}^{s}_{v,\mcZ}(n,r)$. Hence, $\xi_{n,r}$ maps $\qUZ$ into ${\mathcal{Q}}^{s}_{v,\mcZ}(n,r)$.

Referring to\cite[Corollary 3.6]{DGLW2}, the set $\{[A^{\star}]\enspace|~ A^{\star}\in M_n(\N|\N_2)_r\}$ forms a $\mcZ$-basis for ${\mathcal{Q}}^{s}_{v,\mcZ}(n,r)$. Thus, to prove the surjectivity of $\xi_{n,r}$, it suffices to prove that $[A^{\star}]\in \xi_{n,r}(\qUZ)$ for any $A^{\star}=(A^{\ol{0}}|A^{\ol{1}})\in M_n(\N|\N_2)_r$. Let $A^{\ol{0}'}$ be the matrix obtained from $A^{\ol{0}}$ with all diagonal entries replaced by $0$. 
We proceed by induction on $\lVert A^{\star'} \rVert
$ with $A^{\star'}=(A^{\ol{0}'}|A^{\ol{1}})$. If $\lVert A^{\star'} \rVert=0$, then $A^{\star}=diag(\lambda)$, for some $\lambda\in\Lambda(n, r)$, and
\begin{align*}
[A^{\star} ]= [\lambda|O]=\xi_{n,r}(\prod\limits_{i=1}^{n}\begin{bmatrix}
\genK_i\\ \lambda_i
\end{bmatrix}_v )\in \xi_{n,r}(\qUZ)
 \end{align*}
by Corollary \ref{Klambda}. Suppose $A^{\star'}=(A^{\ol{0}'}|A^{\ol{1}})\in M_n(\N|\N_2)^{\pm}$ with $\lVert A^{\star'}\lVert>0$ and the result is true for those $ P^{\star'}  $ with $\lVert P^{\star'}\lVert <\lVert A^{\star'}\lVert $. 
We now consider the image of the basis of $\qUZ$ given in Theorem \ref{basis}. For any $\mu\in \Lambda(n,r)$ with $ro(A^{\star'})\leq \mu$, denote $\begin{bmatrix}
\genK\\ {\mu}
\end{bmatrix}_v:=\prod\limits_{i=1}^{n}\begin{bmatrix}
\genK_i\\ \mu_i
\end{bmatrix}_v $, by \eqref{Aj}, Corollary \ref{Klambda} and the formula in \cite[Lemma 4.1]{DGLW2}, we have 
\begin{align*}
\xi_{n,r}\{\begin{bmatrix}
\genK\\ {\mu}
\end{bmatrix}_v \cdot &\genXE_{A^{\star'}_{(-)}}\cdot\prod_{i=1}^n(\genK_{\ol{i}}^{a_{i,i}^{\ol{1}}}\genK_i^{\tau_i})\cdot \genXE_{A^{\star'}_{(+)}}\}\\
&\qquad= t_v [\mu|O] \cdot A^{\star'}(\boldsymbol{j},r)+\sum\limits_{P^{\star'}\prec A^{\star'}\atop\boldsymbol{j'}\in \Z^n} t_{P^{\star'},\boldsymbol{j'}}[\mu|O] \cdot P^{\star'}(\boldsymbol{j'},r)
\\&\qquad=t_v \sum\limits_{\lambda\in\Lambda(n,r-|A^{\star'}|)}v^{\lambda\cdot \boldsymbol{j}}[\mu|O]\cdot[\lambda+A^{\ol{0}'}|A^{\ol{1}}]+\sum\limits_{P^{\star'}\prec A^{\star'}\atop\boldsymbol{j'}\in \Z^n} t_{P^{\star'},\boldsymbol{j'}}[\mu|O] \cdot P^{\star'}(\boldsymbol{j'},r)
\\&\qquad=t_v \sum\limits_{\lambda\in\Lambda(n,r-|A^{\star'}|) }v^{\lambda\cdot \boldsymbol{j}}\cdot\delta_{\mu,\lambda+ro(A^{\star'})}[\lambda+A^{\ol{0}'}|A^{\ol{1}}]+\sum\limits_{P^{\star'}\prec A^{\star'}\atop\boldsymbol{j'}\in \Z^n} t_{P^{\star'},\boldsymbol{j'}}[\mu|O] \cdot P^{\star'}(\boldsymbol{j'},r)
\\&\qquad=t_v v^{(\mu-ro(A^{\star'}))\cdot \boldsymbol{j}}[\mu-ro(A^{\star'})+A^{\ol{0}'}|A^{\ol{1}}]+\sum\limits_{P^{\star'}\prec A^{\star'}\atop\boldsymbol{j'}\in \Z^n} t_{P^{\star'},\boldsymbol{j'}}[\mu|O] \cdot P^{\star'}(\boldsymbol{j'},r)
\\&\qquad=t'_v [A^{\star}]+\sum\limits_{P^{\star'}\prec A^{\star'}\atop\boldsymbol{j'}\in \Z^n} t'_{P^{\star},\boldsymbol{j'}} [P^{\star}],
\end{align*}
where $\boldsymbol{j}\in \Z^n$, $\tau_i\in \N_2$, $t_v, t'_v,t_{P^{\star'},\boldsymbol{j'}},t'_{P^{\star},\boldsymbol{j'}}\in \mcZ$. Also, $\sum\limits_{P^{\star'}\prec A^{\star'}} t'_{P^{\star},\boldsymbol{j'}} [P^{\star}]$ is a $\mcZ$-linear combination of elements $[P^{\star}]$ with $ P^{\star}\in M_n(\N|\N_2)_r$ such that $ P^{\star' }\prec A^{\star'} $. Thus, by \cite[Lemma 13.21]{DDPW}, $\lVert P^{\star '}\lVert< \lVert A^{\star'}\lVert$, and by induction $[P^{\star }]\in \xi_{n,r}(\qUZ)$.
Consequently, $[A^{\star}]\in \xi_{n,r}(\qUZ)$ and $\xi_{n,r}(\qUZ)$ is a subsuperalgebra of ${\mathcal{Q}}^{s}_{v,\mcZ}(n,r)$.
\end{proof}

 \section{Presentation for the integral form of $\Uvqn$ }\label{Integral}
In this section, we present the Lusztig-type form 0f $\Uvqn$ by generators and relations.

\begin{lem}
    For $m,n\in \N$, the following hold in $\qUZ$:
\begin{align*} 
&(1)\quad
\genE_{\ol{i}}\genF_i^{(n)}=\genF_i^{(n)}\genE_{\ol{i}}-v^{1-n}\genF_i^{(n-1)}\genK_i\genK_{\ol{i+1}}+v^{n-1}\genF_i^{(n-1)}\genK_{\ol{i}}\genK_{i+1}-\genF_i^{(n-2)}\genK_i\genK_{i+1}\genF_{\ol{i}};\\
&(2)\quad \genE_i^{(m)}\genF_{\ol{i}}=\genF_{\ol{i}}\genE_i^{(m)}+v^{m-1}\genE_i^{(m-1)}\genK_{\ol{i}}\genK_{i+1}^{-1}-v^{-m+1}\genE_i^{(m-1)}\genK_{\ol{i+1}}\genK_i^{-1}+\genE_i^{(m-2)}\genE_{\ol{i}}\genK_{i+1}^{-1}\genK_i^{-1};\\
&(3)\quad {\genXE}_{\ol{i}}\genXE_{{i+1}}^{(m)}=v^m\genXE_{{i+1}}^{(m)}{\genXE}_{\ol{i}}+\genXE_{{i+1}}^{(m-1)}(\genXE_{{i}}{\genXE}_{\ol{i+1}}-v{\genXE}_{\ol{i+1}}\genXE_{{i}}).
\end{align*}
\end{lem}
\begin{proof}
    We now prove part (3) by induction on $m$. Indeed, if $m=1$, then it is a relation in (QQ5). Now assume $m>1$, then by induction, we have 
    \begin{align*}
     {\genXE}_{\ol{i}}\genXE_{{i+1}}^{(m+1)}&= \frac{1}{[m+1]_v}[v^m\genXE_{{i+1}}^{(m)}{\genXE}_{\ol{i}}+\genXE_{{i+1}}^{(m-1)}(\genXE_{{i}}{\genXE}_{\ol{i+1}}-v{\genXE}_{\ol{i+1}}\genXE_{{i}})]\genXE_{{i+1}}\\
     &=\frac{v^m}{[m+1]_v}\genXE_{{i+1}}^{(m)}{\genXE}_{\ol{i}}\genXE_{{i+1}}
     +\frac{1}{[m+1]_v}\genXE_{{i+1}}^{(m-1)}(\genXE_{{i}}{\genXE}_{\ol{i+1}}-v{\genXE}_{\ol{i+1}}\genXE_{{i}})\genXE_{{i+1}}\\
     &=\frac{v^m}{[m+1]_v}\genXE_{{i+1}}^{(m)}[v\genXE_{{i+1}}{\genXE}_{\ol{i}}+(\genXE_{{i}}{\genXE}_{\ol{i+1}}-v{\genXE}_{\ol{i+1}}\genXE_{{i}})]\\
     &\qquad
     +\frac{1}{[m+1]_v}\genXE_{{i+1}}^{(m-1)}(\genXE_{{i}}{\genXE}_{\ol{i+1}}-v{\genXE}_{\ol{i+1}}\genXE_{{i}})\genXE_{{i+1}}\\
     &=\frac{v^{m+1}}{[m+1]_v}\genXE_{{i+1}}^{(m)}\genXE_{{i+1}}{\genXE}_{\ol{i}}
     +\frac{v^m}{[m+1]_v}\genXE_{{i+1}}^{(m)}(\genXE_{{i}}{\genXE}_{\ol{i+1}}-v{\genXE}_{\ol{i+1}}\genXE_{{i}})\\
     &\qquad
     +\frac{v^{-1}[m]_v}{[m+1]_v}\genXE_{{i+1}}^{(m)}(\genXE_{{i}}{\genXE}_{\ol{i+1}}-v{\genXE}_{\ol{i+1}}\genXE_{{i}})\\
     &=v^{m+1}\genXE_{{i+1}}^{(m+1)}{\genXE}_{\ol{i}}+\genXE_{{i+1}}^{(m)}(\genXE_{{i}}{\genXE}_{\ol{i+1}}-v{\genXE}_{\ol{i+1}}\genXE_{{i}})
    \end{align*}
    as desired, where the forth equality is due to the relation in (QQ5) and (QQ6). By a parallel argument, we can prove the remaining formulas.
\end{proof}

\begin{rem}
    Recall some formulas from \cite[Section 9.3]{CP} and \cite[Lemma 8.1]{DW} , we have the following.
    \begin{align*}
     &\genE_i^{(m)}\genF_i^{(n)}=\sum_{0\leq t\leq min\{m,n\}}\genF_i^{(n-t)}\begin{bmatrix}
\genK_i;2t-m-n\\t
\end{bmatrix}_v\genE_i^{(m-t)},\\
&\genK_{\ol{i}}^2=v^{-1}\genK_i\begin{bmatrix}
\genK_i;0\\1
\end{bmatrix}_v-(1-v^{-2})\begin{bmatrix}
\genK_i;0\\2
\end{bmatrix}_v.
    \end{align*}
\end{rem}

Along the similar line as \cite[Section 2]{Lus2} and \cite[Section 9.3]{CP}, we introduce and study a $\mcZ$-superalgebra $V_{v,\mcZ}$ defined by generators and relations.

Let $\alpha\in \Phi^+=\{\alpha_{i,j}:=\alpha_i+\cdots+\alpha_{j-1} | 1\leq i<j\leq n\}$. Denote $h(\alpha_{i,j})=j-i$ as the height of $\alpha_{i,j}$ and $g(\alpha_{i,j})=i$.

Let $V_{v,\mcZ}$ be the $\mcZ$-superalgebra defined by the generators $\genXE_\alpha^{(m)}$,${\ol{\genXE}}_\alpha$, $\genF_{\alpha}^{(m)}$,${\ol{\genF}}_{\alpha}$, $\genK_{\ol{i}}$, $\genK_i^{\pm 1}$, $\begin{bmatrix}
\genK_i;c\\t
\end{bmatrix}_v$ with $\alpha\in \Phi^+$,  $1\leq i \leq n $,$c\in \Z$, $m,t\in\N$, and the following relations:
\begin{align*}
   ({\rm QZ1})\quad
& \text{the generators } \genK_{\ol{i}}, \genK_i^{\pm 1}, \begin{bmatrix}
\genK_i;0\\t
\end{bmatrix}_v \text{ commute with each other},\\&
\genK_i\genK_i^{-1}=\genK_i^{-1}\genK_i=1, \quad\begin{bmatrix}
\genK_i;0\\0
\end{bmatrix}_v=1, \quad\genK_{\ol{i}}^2=v^{-1}\genK_i\begin{bmatrix}
\genK_i;0\\1
\end{bmatrix}_v-(1-v^{-2})\begin{bmatrix}
\genK_i;0\\2
\end{bmatrix}_v,\\
&\sum_{0\leq t\leq s} (-1)^t v^{r(s-t)}\begin{bmatrix}
r+t-1\\t
\end{bmatrix}_v\genK_i^t\begin{bmatrix}
\genK_i;0\\r
\end{bmatrix}_v\begin{bmatrix}
\genK_i;0\\s-t
\end{bmatrix}_v=\begin{bmatrix}
r+s\\r
\end{bmatrix}_v\begin{bmatrix}
\genK_i;0\\r+s
\end{bmatrix}_v, r>0, s \geq 0,
\\&
\begin{bmatrix}
\genK_i;-c\\r
\end{bmatrix}_v=\sum_{0\leq s\leq r}(-1)^{s}v^{c(r-s)}\begin{bmatrix}
c+s-1\\s
\end{bmatrix}_v\genK_i^s\begin{bmatrix}
\genK_i;0\\r-s
\end{bmatrix}_v, r\geq0, c > 0,\\&
\begin{bmatrix}
\genK_i;c\\r
\end{bmatrix}_v=\sum_{0\leq s\leq r}v^{c(r-s)}\begin{bmatrix}
c\\s
\end{bmatrix}_v\genK_i^{-s}\begin{bmatrix}
\genK_i;0\\r-s
\end{bmatrix}_v, s\geq0, c \geq 0;\\
({\rm QZ2})\quad
& \genK_i\genXE_{\alpha_j}^{(m)}=v^{m(\epsilon_i,\alpha_j)}\genXE_{\alpha_j}^{(m)}\genK_i,\quad \genK_i{\ol{\genXE}}_{\alpha_j}=v^{(\epsilon_i,\alpha_j)}{\ol{\genXE}}_{\alpha_j}\genK_i,\\&
\genK_i\genF_{\alpha_j}^{(m)}=v^{-m(\epsilon_i,\alpha_j)}\genF_{\alpha_j}^{(m)}\genK_i,\quad \genK_i{\ol{\genF}}_{\alpha_j}=v^{-(\epsilon_i,\alpha_j)}{\ol{\genF}}_{\alpha_j}\genK_i,\\&
\begin{bmatrix}
\genK_i;c\\t
\end{bmatrix}_v\genXE_{\alpha_j}^{(m)}=\genXE_{\alpha_j}^{(m)}\begin{bmatrix}
\genK_i;c+m(\epsilon_i,\alpha_j)\\t
\end{bmatrix}_v,\quad 
\begin{bmatrix}
\genK_i;c\\t
\end{bmatrix}_v{\ol{\genXE}}_{\alpha_j}={\ol{\genXE}}_{\alpha_j}\begin{bmatrix}
\genK_i;c+(\epsilon_i,\alpha_j)\\t
\end{bmatrix}_v,\\&
\begin{bmatrix}
\genK_i;c\\t
\end{bmatrix}_v\genF_{\alpha_j}^{(m)}=\genF_{\alpha_j}^{(m)}\begin{bmatrix}
\genK_i;c-m(\epsilon_i,\alpha_j)\\t
\end{bmatrix}_v,\quad 
\begin{bmatrix}
\genK_i;c\\t
\end{bmatrix}_v{\ol{\genF}}_{\alpha_j}={\ol{\genF}}_{\alpha_j}\begin{bmatrix}
\genK_i;c-(\epsilon_i,\alpha_j)\\t
\end{bmatrix}_v,\\&
\genK_{\ol{i}}\genXE_{\alpha_j}^{(m)}=v^{m(\epsilon_i,\alpha_j)}\genXE_{\alpha_j}^{(m)}\genK_{\ol{i}}+(\delta_{i,j}-\delta_{i-1,j}){\ol{\genXE}}_{\alpha_j}\genXE_{\alpha_{j}}^{(m-1)}\genK_i^{-1},\\&
\genK_{\ol{i}}\genF_{\alpha_j}^{(m)}=v^{m(\epsilon_i,\alpha_j)}\genF_{\alpha_j}^{(m)}\genK_{\ol{i}}-(\delta_{i,j}-\delta_{i-1,j})\genF_{\alpha_{j}}^{(m-1)}{\ol{\genF}}_{\alpha_j}\genK_i,\\&
\genK_{\ol{i}}{\ol{\genXE}}_{\alpha_j}=-v^{(\epsilon_i,\alpha_j)}{\ol{\genXE}}_{\alpha_j}\genK_{\ol{i}}+(\delta_{i,j}+\delta_{i-1,j})\genXE_{\alpha_j}\genK_i^{-1},\\&
\genK_{\ol{i}}{\ol{\genF}}_{\alpha_j}=-v^{(\epsilon_i,\alpha_j)}{\ol{\genF}}_{\alpha_j}\genK_{\ol{i}}+(\delta_{i,j}+\delta_{i-1,j})\genF_{\alpha_j}\genK_i;
\\
({\rm QZ3})\quad
& \genXE_{\alpha_{i}}^{(m)}\genF_{\alpha_{j}}^{(n)}=\genF_{\alpha_{j}}^{(n)}\genXE_{\alpha_{i}}^{(m)},\quad 
\genXE_{\alpha_{i}}^{(m)}{\ol{\genF}}_{\alpha_j}={\ol{\genF}}_{\alpha_j}\genXE_{\alpha_{i}}^{(m)},\\&
{\ol{\genXE}}_{\alpha_i}\genF_{\alpha_{j}}^{(n)}=\genF_{\alpha_{j}}^{(n)}{\ol{\genXE}}_{\alpha_i},\quad 
{\ol{\genXE}}_{\alpha_i}{\ol{\genF}}_{\alpha_j}=-{\ol{\genF}}_{\alpha_j}{\ol{\genXE}}_{\alpha_i}, \text{ if } i\ne j;\\&
\genXE_{\alpha_{i}}^{(m)}\genF_{\alpha_{i}}^{(n)}=\sum_{0\leq t\leq min\{m,n\}}\genF_{\alpha_{i}}^{(n-t)}\begin{bmatrix}
\genK_i;2t-m-n\\t
\end{bmatrix}_v\genXE_{\alpha_{i}}^{(m-t)},\\&
\genXE_{\alpha_{i}}^{(m)}{\ol{\genF}}_{\alpha_i}={\ol{\genF}}_{\alpha_i}\genXE_{\alpha_{i}}^{(m)}-\genK_{\ol{i+1}}\genXE_{\alpha_{i}}^{(m-1)}\genK_i^{-1}+\genK_{\ol{i}}\genXE_{\alpha_{i}}^{(m-1)}\genK_{i+1}^{-1}-{\ol{\genXE}}_{\alpha_{i}}\genXE_{\alpha_{i}}^{(m-2)}\genK_{i+1}^{-1}\genK_i^{-1},\\&
{\ol{\genXE}}_{\alpha_i}\genF_{\alpha_{i}}^{(n)}=\genF_{\alpha_{i}}^{(n)}{\ol{\genXE}}_{\alpha_i}-\genK_i\genF_{\alpha_{i}}^{(n-1)}\genK_{\ol{i+1}}+\genK_{i+1}\genF_{\alpha_{i}}^{(n-1)}\genK_{\ol{i}}-\genF_{\alpha_{i}}^{(n-2)}{\ol{\genF}}_{\alpha_i}\genK_i\genK_{i+1},\\&
{\ol{\genXE}}_{\alpha_i}{\ol{\genF}}_{\alpha_i}=-{\ol{\genF}}_{\alpha_i}{\ol{\genXE}}_{\alpha_i}+\frac{\genK_i\genK_{i+1}-\genK_i^{-1}\genK_{i+1}^{-1}}{v-v^{-1}}+(v-v^{-1})\genK_{\ol{i}}\genK_{\ol{i+1}};
\\
({\rm QZ4})\quad
& \genXE_\alpha^{(m)}\genXE_\alpha^{(n)}=\begin{bmatrix}
m+n\\n
\end{bmatrix}_v\genXE_\alpha^{(m+n)}, \quad \genXE_\alpha^{(0)}=1,\quad \genXE_\alpha^{(m)}{\ol{\genXE}}_{\alpha}={\ol{\genXE}}_{\alpha}\genXE_\alpha^{(m)},\quad {\ol{\genXE}}_{\alpha}^{(2)}=-\frac{v-v^{-1}}{v+v^{-1}}{{\genXE}}_{\alpha}^{(2)};\\&
\genXE_{\alpha_{i}}^{(m)}\genXE_{\alpha}^{(n)}=\genXE_{\alpha}^{(n)}\genXE_{\alpha_{i}}^{(m)},\qquad \genXE_{\alpha_{i}}^{(m)}{\ol{\genXE}}_{\alpha}={\ol{\genXE}}_{\alpha}\genXE_{\alpha_{i}}^{(m)}, \\&
{\ol{\genXE}}_{\alpha_i}\genXE_{\alpha}^{(m)}=\genXE_{\alpha}^{(m)}{\ol{\genXE}}_{\alpha_i},\qquad
{\ol{\genXE}}_{\alpha_i}{\ol{\genXE}}_{\alpha}=-{\ol{\genXE}}_{\alpha}{\ol{\genXE}}_{\alpha_i} 
\text{ if } (\alpha,\alpha_i)=0, \enspace i>g(\alpha);\\&
\genXE_{\alpha'}^{(m)}\genXE_{\alpha}^{(n)}=\sum_{0\leq j\leq min\{m,n\}}v^{j+(m-j)(n-j)}\genXE_{\alpha}^{(n-j)}\genXE_{\alpha+\alpha'}^{(j)}\genXE_{\alpha}^{(m-j)},\\
&
{\ol{\genXE}}_{\alpha'}\genXE_{\alpha}^{(m)}=v^{m}
\genXE_{\alpha}^{(m)}{\ol{\genXE}}_{\alpha'}+v^{m}
{\ol{\genXE}}_{\alpha+\alpha'}\genXE_{\alpha}^{(m-1)},\\
&
v^{mn}\genXE_{\alpha'}^{(m)}\genXE_{\alpha+\alpha'}^{(n)}=\genXE_{\alpha+\alpha'}^{(n)}\genXE_{\alpha'}^{(m)},\quad \genXE_{\alpha}^{(m)}\genXE_{\alpha+\alpha'}^{(n)}=v^{mn}\genXE_{\alpha+\alpha'}^{(n)}\genXE_{\alpha}^{(m)},
\\
&v^m \genXE_{\alpha+\alpha'}^{(m)}{\ol{\genXE}}_{\alpha}={\ol{\genXE}}_{\alpha}\genXE_{\alpha+\alpha'}^{(m)},\quad 
\genXE_{\alpha}^{(m)}{\ol{\genXE}}_{\alpha+\alpha'}=v^m{\ol{\genXE}}_{\alpha+\alpha'}\genXE_{\alpha}^{(m)},\quad {\ol{\genXE}}_{\alpha}{\ol{\genXE}}_{\alpha+\alpha'}=-v{\ol{\genXE}}_{\alpha+\alpha'}{\ol{\genXE}}_{\alpha},\\
&
\genXE_{\alpha'}^{(m)}{\ol{\genXE}}_{\alpha+\alpha'}=v^m{\ol{\genXE}}_{\alpha+\alpha'}\genXE_{\alpha'}^{(m)}-v^m\genXE_{\alpha'}^{(m-1)}{\genXE}_{\alpha+\alpha'}{\ol{\genXE}}_{\alpha'}+v^{m-1}\genXE_{\alpha'}^{(m-1)}{\ol{\genXE}}_{\alpha'}{\genXE}_{\alpha+\alpha'},\\
&
v{\ol{\genXE}}_{\alpha+\alpha'}{\ol{\genXE}}_{\alpha'}=-{\ol{\genXE}}_{\alpha'}{\ol{\genXE}}_{\alpha+\alpha'}-(v-v^{-1})\genXE_{\alpha+\alpha'}\genXE_{\alpha'}, \text{ if } (\alpha, \alpha')=-1\enspace g(\alpha)<g(\alpha');\\
&
{\ol{\genXE}}_{\alpha_i}\genXE_{\alpha_{i+1}}^{(m)}=v^m\genXE_{\alpha_{i+1}}^{(m)}{\ol{\genXE}}_{\alpha_i}+\genXE_{\alpha_{i+1}}^{(m-1)}(\genXE_{\alpha_{i}}{\ol{\genXE}}_{\alpha_{i+1}}-v{\ol{\genXE}}_{\alpha_{i+1}}\genXE_{\alpha_{i}}),\\
&
{\ol{\genXE}}_{\alpha_i}{\ol{\genXE}}_{\alpha_{i+1}}+v{\ol{\genXE}}_{\alpha_{i+1}}{\ol{\genXE}}_{\alpha_i}={{\genXE}}_{\alpha_i}{{\genXE}}_{\alpha_{i+1}}-v{{\genXE}}_{\alpha_{i+1}}{{\genXE}}_{\alpha_i};\\
({\rm QZ5})\quad
& \genF_\alpha^{(m)}\genF_\alpha^{(n)}=\begin{bmatrix}
m+n\\n
\end{bmatrix}_v\genF_\alpha^{(m+n)}, \quad \genF_\alpha^{(0)}=1,\quad \genF_\alpha^{(m)}{\ol{\genF}}_{\alpha}={\ol{\genF}}_{\alpha}\genF_\alpha^{(m)},\quad {\ol{\genF}}_{\alpha}^{(2)}=\frac{v-v^{-1}}{v+v^{-1}}{{\genF}}_{\alpha}^{(2)};\\
&
\genF_{\alpha_{i}}^{(m)}\genF_{\alpha}^{(n)}=\genF_{\alpha}^{(n)}\genF_{\alpha_{i}}^{(m)},\qquad \genF_{\alpha_{i}}^{(m)}{\ol{\genF}}_{\alpha}={\ol{\genF}}_{\alpha}\genF_{\alpha_{i}}^{(m)},\\
&
{\ol{\genF}}_{\alpha_{i}}\genF_{\alpha}^{(m)}=\genF_{\alpha}^{(m)}{\ol{\genF}}_{\alpha_{i}},\qquad
{\ol{\genF}}_{\alpha_i}{\ol{\genF}}_{\alpha}=-{\ol{\genF}}_{\alpha}{\ol{\genF}}_{\alpha_i},
\text{ if } (\alpha,\alpha_i)=0, \enspace i>g(\alpha);\\
&
\genF_{\alpha}^{(m)}\genF_{\alpha'}^{(n)}=\sum_{0\leq j\leq min\{m,n\}}v^{-j-(m-j)(n-j)}\genF_{\alpha}^{(n-j)}\genF_{\alpha+\alpha'}^{(j)}\genF_{\alpha}^{(m-j)},\\
&
{\ol{\genF}}_{\alpha'}\genF_{\alpha}^{(n)}=v^{n}\genF_{\alpha}^{(n)}{\ol{\genF}}_{\alpha'}-\genF_{\alpha}^{(n-1)}{\ol{\genF}}_{\alpha+\alpha'},\\
&
\genF_{\alpha+\alpha'}^{(m)}\genF_{\alpha'}^{(n)}=v^{mn}\genF_{\alpha'}^{(n)}\genF_{\alpha+\alpha'}^{(m)},\quad \genF_{\alpha}^{(m)}\genF_{\alpha+\alpha'}^{(n)}=v^{mn}\genF_{\alpha+\alpha'}^{(n)}\genF_{\alpha}^{(m)},\\
&{\ol{\genF}}_{\alpha}\genF_{\alpha+\alpha'}^{(m)}=v^{m}\genF_{\alpha+\alpha'}^{(m)}{\ol{\genF}}_{\alpha},\quad 
v^m{\ol{\genF}}_{\alpha+\alpha'}\genF_{\alpha}^{(m)}=\genF_{\alpha}^{(m)}{\ol{\genF}}_{\alpha+\alpha'},\quad 
{\ol{\genF}}_{\alpha}{\ol{\genF}}_{\alpha+\alpha'}=-v{\ol{\genF}}_{\alpha+\alpha'}{\ol{\genF}}_{\alpha},\\
&
\genF_{\alpha'}^{(m)}{\ol{\genF}}_{\alpha+\alpha'}=v^m{\ol{\genF}}_{\alpha+\alpha'}\genF_{\alpha'}^{(m)}+{\ol{\genF}}_{\alpha'}\genF_{\alpha+\alpha'}\genF_{\alpha'}^{(m-1)}-v\genF_{\alpha+\alpha'}{\ol{\genF}}_{\alpha'}\genF_{\alpha'}^{(m-1)},\\
&
v{\ol{\genF}}_{\alpha+\alpha'}{\ol{\genF}}_{\alpha'}+{\ol{\genF}}_{\alpha'}{\ol{\genF}}_{\alpha+\alpha'}=(v-v^{-1})\genF_{\alpha+\alpha'}\genF_{\alpha'}, \text{ if } (\alpha, \alpha')=-1,\enspace g(\alpha)<g(\alpha');\\
&
v^{m}\genF_{\alpha_{i+1}}^{(m)}{\ol{\genF}}_{\alpha_i}={\ol{\genF}}_{\alpha_i}\genF_{\alpha_{i+1}}^{(m)}+v^{m-1}(v{\ol{\genF}}_{\alpha_{i+1}}\genF_{\alpha_{i}}-\genF_{\alpha_{i}}{\ol{\genF}}_{\alpha_{i+1}}) \genF_{\alpha_{i+1}}^{(m-1)},\\
&
{\ol{\genF}}_{\alpha_i}{\ol{\genF}}_{\alpha_{i+1}}+v{\ol{\genF}}_{\alpha_{i+1}}{\ol{\genF}}_{\alpha_i}=v{{\genF}}_{\alpha_{i+1}}{{\genF}}_{\alpha_i}-{{\genF}}_{\alpha_i}{{\genF}}_{\alpha_{i+1}}.
\end{align*}

Let $V^{+}_{v,\mcZ}$ (respectively, $V^{-}_{v,\mcZ}$) be the $\mcZ$-subsuperalgebra of $V_{v,\mcZ}$ generated by $\genXE_\alpha^{(m)}$,${\ol{\genXE}}_\alpha$ (respectively, $\genF_{\alpha}^{(m)}$,${\ol{\genF}}_{\alpha}$) with $\alpha\in \Phi^+$,$m\in\N$. Denote by  $V^{0}_{v,\mcZ}$ the  subsuperalgebra of $V_{v,\mcZ}$ generated by $\genK_{\ol{i}}$, $\genK_i^{\pm 1}$, $\begin{bmatrix}
\genK_i;c\\t
\end{bmatrix}_v$ with $1\leq i \leq n $,$c\in \Z$, $t\in\N$.

For any $\beta \in \Phi^+$, there exist integers $i$ and $j$ with $i<j$ such that $\beta=\alpha_{i,j}$. And we define a partial order for the set $\Phi^+$as follows:  for any $\alpha_{i,j},~\alpha_{k,l}\in\Phi^+$ we say $\alpha_{i,j}<\alpha_{k,l}$ if $i<k$ or $i=k$ and $j<l$. We then compute the product $\prod_{\alpha\in\Phi^+}(\genXE_{\alpha}^{(m)}{\ol{\genXE}}_{\alpha}^{m'})$ where $m\in\N$ and $m'\in\{0,1\}$, using any total order on $\Phi^+$ that is compatible with the partial order defined above.
Then, along the similar line as \cite[Proposition 2.12, Proposition 2.17]{Lus2}, we have the following conclusion.
\begin{prop}
\begin{enumerate}
    \item The superalgebra $V^{+}_{v,\mcZ}$ is generated by $\genXE_{\alpha_i}^{(m)}$, ${\ol{\genXE}}_{\alpha_i}$ as a $\mcZ$-superalgebra with $1\leq i <n$, $m\in \N$;
    
    \item $V_{v,\mcZ}$ is generated as a $\mcZ$-superalgebra by $\genXE_{\alpha_i}^{(m)}$, ${\ol{\genXE}}_{\alpha_i}$, $\genF_{\alpha_i}^{(m)}$, ${\ol{\genF}}_{\alpha_i}$, $\genK_{\ol{i}}$, $\genK_i^{\pm 1}$, $\begin{bmatrix}
\genK_i\\t
\end{bmatrix}_v$  with $m\in \N, 1\leq i \leq n, t\in \N$;

    \item  the monomials $\prod_{\alpha\in\Phi^+}(\genXE_{\alpha}^{(m)}{\ol{\genXE}}_{\alpha}^{m'})$  with $m\in \N, m'\in\{0,1\}$ generate $V^{+}_{v,\mcZ}$ as a $\mcZ$-supermodule.
\end{enumerate} 
\end{prop}
\begin{proof}
    By an argument parallel to the proof of Proposition 2.12 in \cite{Lus2}, one can check that the proposition holds.
\end{proof}

Using the natural embedding $\mcZ\in \Qv$, we form the $\Qv$-superalgebra $V_{\Qv}^+$, $V_{\Qv}^-$,$V_{\Qv}^0$ and $V_{\Qv}$ by applying $()\otimes_\Z \Qv$ to $V^{+}_{v,\mcZ}$, $V^{-}_{v,\mcZ}$, $V^{0}_{v,\mcZ}$,$V_{v,\mcZ}$. And we shall write $\genXE_{\alpha}$ and $\genXE_{\alpha}$ instead of $\genXE_{\alpha}^{(1)}$ and $\genXE_{\alpha}^{(1)}$.

Similar to \cite[Proposition 9.3.3]{CP}, we have the following.

\begin{prop}
\begin{enumerate}
    \item The set of products $\prod_{\alpha\in\Phi^+}(\genXE_{\alpha}^{(m)}{\ol{\genXE}}_{\alpha}^{m'})$ form a $\mcZ$-basis of $V^+_{v,\mcZ}$ and a $\Qv$-basis of $V_{\Qv}^+$ with $m\in \N, m'\in\{0,1\}$;
    \item the set of products $\prod_{\alpha\in\Phi^+}(\genF_{\alpha}^{(m)}{\ol{\genF}}_{\alpha}^{m'})$ form a $\mcZ$-basis of $V^-_{v,\mcZ}$ and a $\Qv$-basis of $V_{\Qv}^-$ with $m\in \N, m'\in\{0,1\}$;
\item the set of products $\prod_{ i=1}^{n}(\genK_i^{\sigma_i}\begin{bmatrix}
\genK_i\\t_i
\end{bmatrix}_v\genK_{\ol{i}}^{\xi_i})$ form a $\mcZ$-basis of $V^0_{v,\mcZ}$ and a $\Qv$-basis of $V_{\Qv}^0$ with $t_i\in \N, \sigma_i,\xi_i\in\{0,1\}$;
\item the set of products $\prod_{\alpha\in\Phi^+}(\genF_{\alpha}^{(n)}{\ol{\genF}}_{\alpha}^{n'}) \cdot \prod_{ i=1}^{n}(\genK_i^{\sigma_i}\begin{bmatrix}
\genK_i\\t_i
\end{bmatrix}_v\genK_{\ol{i}}^{\xi_i})\cdot \prod_{\alpha\in\Phi^+}(\genXE_{\alpha}^{(m)}{\ol{\genXE}}_{\alpha}^{m'})$ form a $\mcZ$-basis of $V_{v,\mcZ}$ with $m,n,t_i\in \N,~ m',n',\sigma_i,\xi_i\in\{0,1\}$.
\end{enumerate}
\end{prop}

For any $\beta \in \Phi^+$, it can be written uniquely as $\beta=w_{\beta}(\alpha_{i_{\beta}}):=s_i s_{i+1}\cdots s_{i+m-1}(\alpha_{i+m})$ with some $i<m$. Then the root vectors $\genXE_\beta, {\ol{\genXE}}_\beta, \genF_\beta, {\ol{\genF}}_\beta$ in $\Uvqn$  can be described as 
\begin{align*}
    &\genXE_\beta=T_{w_{\beta}}(\genE_{i_{\beta}})=T_i T_{i+1}\cdots T_{i+m-1}(\genE_{i+m}),\\ 
    &{\ol{\genXE}}_\beta=T_{w_{\beta}}(\genE_{\ol{i_{\beta}}})=T_i T_{i+1}\cdots T_{i+m-1}(\genE_{\ol{i+m}}),\\
    &\genF_\beta=T_{w_{\beta}}(\genE_{i_{\beta}})=T_i T_{i+1}\cdots T_{i+m-1}(\genF_{i+m}),\\ 
    &{\ol{\genF}}_\beta=T_{w_{\beta}}(\genE_{\ol{i_{\beta}}})=T_i T_{i+1}\cdots T_{i+m-1}(\genF_{\ol{i+m}}).
\end{align*}
Associated with Proposition \ref{action_uvqn}, a direct calculation shows that $T_{w_{\beta}}(\genE_{i_{\beta}})=\genXE_{i,i+m+1}$, $T_{w_{\beta}}(\genE_{\ol{i_{\beta}}})={\ol{\genXE}}_{i,i+m+1}$, $T_{w_{\beta}}(\genF_{i_{\beta}})=\genE_{i+m+1,i}$, and $T_{w_{\beta}}(\genF_{\ol{i_{\beta}}})={\ol{\genXE}}_{i+m+1,i}$. Hence, by Theorem \ref{basis}, the proposition follows.

\begin{prop}\label{iso}
\begin{enumerate}
\item There is a unique $\Qv$-superalgebra isomorphism $V_{\Qv}\to \Uvqn$ which takes $\genXE_\beta$ to $T_{w_{\beta}}(\genE_{i_{\beta}})$, ${\ol{\genXE}}_\beta$ to $T_{w_{\beta}}(\genE_{\ol{i_{\beta}}})$, $\genF_\beta$ to $T_{w_{\beta}}(\genF_{i_{\beta}})$, ${\ol{\genF}}_\beta$ to $T_{w_{\beta}}(\genF_{\ol{i_{\beta}}})$, and $\genK_i^{\pm 1}$ to $\genK_i^{\pm 1}$ for $\beta\in \Phi^+, 1\leq i \leq n$.
\item There is a unique $\mcZ$-superalgebra isomorphism $V_{v,\mcZ}\to \qUZ$  that takes $\genXE^{(m)}_\beta$ to $T_{w_{\beta}}(\genE^{(m)}_{i_{\beta}})$, ${\ol{\genXE}}_\beta$ to $T_{w_{\beta}}(\genE_{\ol{i_{\beta}}})$, $\genF^{(m)}_\beta$ to $T_{w_{\beta}}(\genF^{(m)}_{i_{\beta}})$, ${\ol{\genF}}_\beta$ to $T_{w_{\beta}}(\genF_{\ol{i_{\beta}}})$, and $\genK_i^{\pm 1}$ to $\genK_i^{\pm 1}$ for $\beta\in \Phi^+, m\in \N, 1\leq i \leq n$.
\item There is an isomorphism $\qUZ\otimes_{\mcZ}\Qv\to \Uvqn$.
\end{enumerate}
\end{prop}

Hence, $V_{v,\mcZ}$ is the restricted integral form for $\Uvqn$. And it is easy to check that the relations in (QZ1-QZ5) and (QQ1-QQ5) can be derived from each other by Lemma \ref{KE} - Lemma \ref{no-po1}.

 \section{Application: quantum queer superalgebra at root of unity }\label{root 1}

When $v$ is not a root of unity, the restricted and non-restricted integral forms coincide; once $v$ becomes a root of unity, they diverge and are no longer isomorphic. In the present section we examine the root-of-unity case as an application of the preceding constructions. 

Define {$\eUZ=\qUZ\otimes_{\mcZ}\Z[\epsilon,\epsilon^{-1}]$ } using the homomorphism $\mcZ\to \Z[\epsilon,\epsilon^{-1}]$ taking $v$ to $\epsilon$, where $\epsilon\in \C^{\times}$. And we assume that $\epsilon$ is a primitive $l$-th root of unity, where $l$ is an odd number. 

There are some properties in $\eUZ$ following from \cite[Section 5]{Lus2} and \cite[Section 9.3]{CP}.

\begin{lem}\cite[Section 9.3]{CP}
 In $\eUZ$, we have
 \begin{enumerate}
 \item if $0\leq m<l,~0\leq n<l,~m+n\geq l$, then $\begin{bmatrix}
m+n\\n
\end{bmatrix}_v=0$. And also $(\genXE_{\alpha})^l =0$, $\genXE_{\alpha}({\ol{\genXE}}_{\alpha})^{l-1} =0$ for all $\alpha\in\Phi^+$;

 \item $\genK_i^l$ is central and $\genK_i^{2l}=1$ for $1\leq i \leq n$.
 \end{enumerate}
\end{lem}

 Let ${{\boldsymbol U}^{res}_{\ep,\mc Z}}$ be the $\mcZ$-subsuperalgebra of $\qUZ$ generated by $\genXE_\alpha^{(m)}$,${\ol{\genXE}}_\alpha$, $\genF_{\alpha}^{(n)}$,${\ol{\genF}}_{\alpha}$, $\genK_{\ol{i}}$, $\genK_i^{\pm 1}$, $\begin{bmatrix}
\genK_i\\t_i
\end{bmatrix}_v$ with $\alpha\in \Phi^+$, $1\leq i \leq n $, $0\leq m,n,t_i<l$.  
Define the  subsuperalgebra ${{\boldsymbol U}^{res~+}_{\ep,\mc Z}}$ (respectively, ${{\boldsymbol U}^{res~-}_{\ep,\mc Z}}$) generated by $\genXE_\alpha^{(m)}$, ${\ol{\genXE}}_\alpha$ (respectively, $\genF_{\alpha}^{(n)}$, ${\ol{\genF}}_{\alpha}$). Denote by  ${{\boldsymbol U}^{res~0}_{\ep,\mc Z}}$ the  subsuperalgebra of $\qUZ$ generated by $\genK_{\ol{i}}$, $\genK_i^{\pm 1}$, $\begin{bmatrix}
K_i\\t_i
\end{bmatrix}_v$. 

Then, similar to \cite[Proposition 9.3.7]{CP}, we have the following conclusion.
\begin{prop}
    Suppose that $\ep$ is a primitive $l$-th root of unity, where $l$ is odd. Then we have
    \begin{enumerate}
        \item ${{\boldsymbol U}^{res~+}_{\ep,\mc Z}}$ is a free $\Z[\ep,\ep^{-1}]$-supermodule of rank $(l^{\frac{n(n+1)}{2}}~|~2^{\frac{n(n+1)}{2}})$ with basis the set of products 
         $\prod_{\alpha\in\Phi^+}(\genXE_{\alpha}^{(m)}{\ol{\genXE}}_{\alpha}^{m'}),$ where $0\leq m <l,~ m'\in\{0,1\}$;
    \item ${{\boldsymbol U}^{res~-}_{\ep,\mc Z}}$ is a free $\Z[\ep,\ep^{-1}]$-supermodule of rank $(l^{\frac{n(n+1)}{2}}~|~2^{\frac{n(n+1)}{2}})$ with basis the set of products 
   $\prod_{\alpha\in\Phi^+}(\genF_{\alpha}^{(m)}{\ol{\genF}}_{\alpha}^{m'})$, where $0\leq m<0,~ m'\in\{0,1\}$;
\item ${{\boldsymbol U}^{res~0}_{\ep,\mc Z}}$ is a free $\Z[\ep,\ep^{-1}]$-supermodule of rank $((2l)^n~|~2^n)$ with basis the set of products 
$\prod_{ i=1}^{n}\genK_i^{\sigma_i}\begin{bmatrix}
\genK_i\\t_i
\end{bmatrix}_v\genK_{\ol{i}}^{\xi_i},$ where $0\leq t_i<l,~ \sigma_i,\xi_i\in\{0,1\}$;
\item there is an isomorphism of $\Z[\ep,\ep^{-1}]$-supermodules ${{\boldsymbol U}^{res~-}_{\ep,\mc Z}}\otimes{{\boldsymbol U}^{res~0}_{\ep,\mc Z}}\otimes{{\boldsymbol U}^{res~+}_{\ep,\mc Z}}\to {{\boldsymbol U}^{res}_{\ep,\mc Z}}$.
    \end{enumerate}
\end{prop}

Define {$\eU=\qUZ\otimes_{\mcZ}\Q(\epsilon)$}. Then we show that the Hopf algebra $\eU$ has finite dimension $(2^n l^{n(n+2)}~|~2^{n(n+2)})$ over $\Q(\ep)$ and the basis is defined as in the preceding proposition, with the elements in part (3) replaced by the products $\prod_{ i=1}^{n}\genK_i^{\sigma_i}\genK_{\ol{i}}^{\xi_i},$ where $0\leq \sigma_i<2l,~\xi_i\in\{0,1\}$. In addition, a presentation of $\eU$ can be provided by generators and relations.
\begin{prop}\label{rootof1}
    $\eU$ is isomorphic to the associative $\Q(\ep)$-superalgebra with generators $\genXE_\alpha$, ${\ol{\genXE}}_\alpha$, $\genF_\alpha$, ${\ol{\genF}}_\alpha$, $\genK_i^{\pm 1}$, $\genK_{\ol{i}}$, for $\alpha\in\Phi^+$, $1\leq i \leq n$, and the following defining relations:
    \begin{align*}
({\rm QU1})\quad & 
    {\genK}_{i} {\genK}_{i}^{-1} = {\genK}_{i}^{-1} {\genK}_{i} = 1,  \qquad
	{\genK}_{i} {\genK}_{j} = {\genK}_{j} {\genK}_{i} , \qquad
	{\genK}_{i} {\genK}_{\ol{j}} = {\genK}_{\ol{j}} {\genK}_{i}, \\
    &	{\genK}_{\ol{i}} {\genK}_{\ol{j}} + {\genK}_{\ol{j}} {\genK}_{\ol{i}}
	= 2 {\delta}_{i,j} \frac{{\genK}_{i}^2 - {\genK}_{i}^{-2}}{{\ep}^2 - {\ep}^{-2}},\qquad \genK_i^{2l}=1; \\
({\rm QU2})\quad &
    \genK_i\genXE_{\alpha_j}={\ep}^{(\epsilon_i,\alpha_j)}\genXE_{\alpha_j}\genK_i,\quad \genK_{\ol{i}}\genXE_{\alpha_j}={\ep}^{(\epsilon_i,\alpha_j)}\genXE_{\alpha_j}\genK_{\ol{i}}+(\delta_{i,j}-\delta_{i-1,j}){\ol{\genXE}}_{\alpha_j}\genK_i^{-1},\\& 
    \genK_i{\ol{\genXE}}_{\alpha_j}={\ep}^{(\epsilon_i,\alpha_j)}{\ol{\genXE}}_{\alpha_j}\genK_i,\quad \genK_{\ol{i}}{\ol{\genXE}}_{\alpha_j}=-{\ep}^{(\epsilon_i,\alpha_j)}{\ol{\genXE}}_{\alpha_j}\genK_{\ol{i}}+(\delta_{i,j}+\delta_{i-1,j})\genXE_{\alpha_j}\genK_i^{-1},\\&    
    \genK_i\genF_{\alpha_j}={\ep}^{-(\epsilon_i,\alpha_j)}\genF_{\alpha_j}\genK_i,\quad \genK_{\ol{i}}\genF_{\alpha_j}={\ep}^{(\epsilon_i,\alpha_j)}\genF_{\alpha_j}\genK_{\ol{i}}-(\delta_{i,j}-\delta_{i-1,j}){\ol{\genF}}_{\alpha_j}\genK_i,\\&
    \genK_i{\ol{\genF}}_{\alpha_j}={\ep}^{-(\epsilon_i,\alpha_j)}{\ol{\genF}}_{\alpha_j}\genK_i,\quad
    \genK_{\ol{i}}{\ol{\genF}}_{\alpha_j}=-{\ep}^{(\epsilon_i,\alpha_j)}{\ol{\genF}}_{\alpha_j}\genK_{\ol{i}}+(\delta_{i,j}+\delta_{i-1,j})\genF_{\alpha_j}\genK_i;\\
({\rm QU3})\quad & 
\genXE_{\alpha_{i}}\genF_{\alpha_{j}}-\genF_{\alpha_{j}}\genXE_{\alpha_{i}}=\delta_{i,j}\frac{{\genK}_{i} {\genK}_{i+1}^{-1} - {\genK}_{i}^{-1}{\genK}_{i+1}}{{\ep} - {\ep}^{-1}},\\&
\genXE_{\alpha_{i}}{\ol{\genF}}_{\alpha_j}-{\ol{\genF}}_{\alpha_j}\genXE_{\alpha_{i}}=\delta_{i,j}(\genK_{i+1}^{-1}\genK_{\ol{i}}-\genK_{\ol{i+1}}\genK_i^{-1}),\\&
{\ol{\genXE}}_{\alpha_i}\genF_{\alpha_{j}}-\genF_{\alpha_{j}}{\ol{\genXE}}_{\alpha_i}=\delta_{i,j}({\genK_{i+1}\genK_{\ol{i}}-\genK_{\ol{i+1}}\genK_i}),\\&
{\ol{\genXE}}_{\alpha_i}{\ol{\genF}}_{\alpha_i}+{\ol{\genF}}_{\alpha_i}{\ol{\genXE}}_{\alpha_i}=\delta_{i,j}(\frac{\genK_i\genK_{i+1}-\genK_i^{-1}\genK_{i+1}^{-1}}{{\ep}-{\ep}^{-1}}+({\ep}-{\ep}^{-1})\genK_{\ol{i}}\genK_{\ol{i+1}});
\\
({\rm QU4})\quad &
{\ol{\genXE}}_{\alpha}^2=-\frac{{\ep}-{\ep}^{-1}}{{\ep}+{\ep}^{-1}}{{\genXE}}_{\alpha}^2,\quad {{\genXE}}_{\alpha}^l=0, \quad {{\genXE}}_{\alpha}{\ol{\genXE}}_{\alpha}={\ol{\genXE}}_{\alpha}{{\genXE}}_{\alpha};\\&
{{\genXE}}_{\alpha_i}{{\genXE}}_{\alpha}={{\genXE}}_{\alpha}{{\genXE}}_{\alpha_i},\qquad {{\genXE}}_{\alpha_i}{\ol{\genXE}}_{\alpha}={\ol{\genXE}}_{\alpha}{{\genXE}}_{\alpha_i},\\&
 {\ol{\genXE}}_{\alpha_i}{{\genXE}}_{\alpha}={{\genXE}}_{\alpha}{\ol{\genXE}}_{\alpha_i},\qquad
{\ol{\genXE}}_{\alpha_i}{\ol{\genXE}}_{\alpha}=-{\ol{\genXE}}_{\alpha}{\ol{\genXE}}_{\alpha_i},\quad \text{ if } (\alpha,\alpha_i)=0, \enspace i>g(\alpha);\\&
{{\genXE}}_{\alpha'}{{\genXE}}_{\alpha}={\ep}{{\genXE}}_{\alpha}{{\genXE}}_{\alpha'}+{\ep}{{\genXE}}_{\alpha+\alpha'},\quad
{\ol{\genXE}}_{\alpha'}{{\genXE}}_{\alpha}={\ep}{{\genXE}}_{\alpha}{\ol{\genXE}}_{\alpha'}+{\ep}{\ol{\genXE}}_{\alpha+\alpha'},\\&
{\ep}{{\genXE}}_{\alpha'}{{\genXE}}_{\alpha+\alpha'}={{\genXE}}_{\alpha+\alpha'}{{\genXE}}_{\alpha'},\quad
{{\genXE}}_{\alpha}{{\genXE}}_{\alpha+\alpha'}={\ep}{{\genXE}}_{\alpha+\alpha'}{{\genXE}}_{\alpha},\\&
{\ol{\genXE}}_{\alpha}{{\genXE}}_{\alpha+\alpha'}={\ep}{{\genXE}}_{\alpha+\alpha'}{\ol{\genXE}}_{\alpha},\quad
{{\genXE}}_{\alpha}{\ol{\genXE}}_{\alpha+\alpha'}={\ep}{\ol{\genXE}}_{\alpha+\alpha'}{{\genXE}}_{\alpha},\quad {\ol{\genXE}}_{\alpha}{\ol{\genXE}}_{\alpha+\alpha'}=-\ep{\ol{\genXE}}_{\alpha+\alpha'}{\ol{\genXE}}_{\alpha},\\
&
{{\genXE}}_{\alpha'}{\ol{\genXE}}_{\alpha+\alpha'}-\ep{\ol{\genXE}}_{\alpha+\alpha'}{{\genXE}}_{\alpha'}={\ol{\genXE}}_{\alpha'}{{\genXE}}_{\alpha+\alpha'}-\ep{{\genXE}}_{\alpha+\alpha'}{\ol{\genXE}}_{\alpha'},\\
&{\ol{\genXE}}_{\alpha'}{\ol{\genXE}}_{\alpha+\alpha'}+{\ep}{\ol{\genXE}}_{\alpha+\alpha'}{\ol{\genXE}}_{\alpha'}=-({\ep}-{\ep}^{-1}){{\genXE}}_{\alpha+\alpha'}{{\genXE}}_{\alpha'}, \text{ if }~(\alpha,\alpha')=-1,\enspace g(\alpha)<g(\alpha');\\&
{\ol{\genXE}}_{\alpha_i}{{\genXE}}_{\alpha_{i+1}}={\ep}{{\genXE}}_{\alpha_{i+1}}{\ol{\genXE}}_{\alpha_i}+{{\genXE}}_{\alpha_i}{\ol{\genXE}}_{\alpha_{i+1}}-{\ep}{\ol{\genXE}}_{\alpha_{i+1}}{{\genXE}}_{\alpha_i},\\&
{\ol{\genXE}}_{\alpha_i}{\ol{\genXE}}_{\alpha_{i+1}}+{\ep}{\ol{\genXE}}_{\alpha_{i+1}}{\ol{\genXE}}_{\alpha_i}={{\genXE}}_{\alpha_i}{{\genXE}}_{\alpha_{i+1}}-{\ep}{{\genXE}}_{\alpha_i}{{\genXE}}_{\alpha_i};\\
({\rm QU5})\quad &
{\ol{\genF}}_{\alpha}^2=\frac{{\ep}-{\ep}^{-1}}{{\ep}+{\ep}^{-1}}{{\genF}}_{\alpha}^2,\quad {{\genF}}_{\alpha}^l=0,\quad {{\genF}}_{\alpha}{\ol{\genF}}_{\alpha}={\ol{\genF}}_{\alpha}{{\genF}}_{\alpha}; \\&
{{\genF}}_{\alpha_i}{{\genF}}_{\alpha}={{\genF}}_{\alpha}{{\genF}}_{\alpha_i},\quad {{\genF}}_{\alpha_i}{\ol{\genF}}_{\alpha}={\ol{\genF}}_{\alpha}{{\genF}}_{\alpha_i}, \\&
{\ol{\genF}}_{\alpha_i}{\ol{\genF}}_{\alpha}={\ol{\genF}}_{\alpha}{\ol{\genF}}_{\alpha_i}, \quad {\ol{\genF}}_{\alpha_i}{\ol{\genF}}_{\alpha}=-{\ol{\genF}}_{\alpha}{\ol{\genF}}_{\alpha_i},\quad \text{ if } (\alpha,\alpha_i)=0, \enspace i>g(\alpha);\\&
{{\genF}}_{\alpha'}{{\genF}}_{\alpha}={\ep}{{\genF}}_{\alpha}{{\genF}}_{\alpha'}-{{\genF}}_{\alpha+\alpha'},\quad
{\ol{\genF}}_{\alpha'}{{\genF}}_{\alpha}={\ep}{{\genF}}_{\alpha}{\ol{\genF}}_{\alpha'}-{\ol{\genF}}_{\alpha+\alpha'},\\&
{\ep}{{\genF}}_{\alpha'}{{\genF}}_{\alpha+\alpha'}={{\genF}}_{\alpha+\alpha'}{{\genF}}_{\alpha'},\quad
{{\genF}}_{\alpha}{{\genF}}_{\alpha+\alpha'}={\ep}{{\genF}}_{\alpha+\alpha'}{{\genF}}_{\alpha},\\&
{\ol{\genF}}_{\alpha}{{\genF}}_{\alpha+\alpha'}={\ep}{{\genF}}_{\alpha+\alpha'}{\ol{\genF}}_{\alpha'},\quad
{{\genF}}_{\alpha}{\ol{\genF}}_{\alpha+\alpha'}={\ep}{\ol{\genF}}_{\alpha+\alpha'}{{\genF}}_{\alpha},\quad {\ol{\genF}}_{\alpha}{\ol{\genF}}_{\alpha+\alpha'}=-\ep{\ol{\genF}}_{\alpha+\alpha'}{\ol{\genF}}_{\alpha},\\
&
{{\genF}}_{\alpha'}{\ol{\genF}}_{\alpha+\alpha'}-\ep{\ol{\genF}}_{\alpha+\alpha'}{{\genF}}_{\alpha'}={\ol{\genF}}_{\alpha'}{{\genF}}_{\alpha+\alpha'}-\ep{{\genF}}_{\alpha+\alpha'}{\ol{\genF}}_{\alpha'},\\
&
{\ol{\genF}}_{\alpha'}{\ol{\genF}}_{\alpha+\alpha'}+{\ep}{\ol{\genF}}_{\alpha+\alpha'}{\ol{\genF}}_{\alpha'}=({\ep}-{\ep}^{-1}){{\genF}}_{\alpha+\alpha'}{{\genF}}_{\alpha'}, \text{ if }~(\alpha,\alpha')=-1,\enspace g(\alpha)<g(\alpha');\\&
{\ol{\genF}}_{\alpha_i}{{\genF}}_{\alpha_{i+1}}={\ep}{{\genF}}_{\alpha_{i+1}}{\ol{\genF}}_{\alpha_i}+{{\genF}}_{\alpha_i}{\ol{\genF}}_{\alpha_{i+1}}-{\ep}{\ol{\genF}}_{\alpha_{i+1}}{{\genF}}_{\alpha_i},\\&
{\ol{\genF}}_{\alpha_i}{\ol{\genF}}_{\alpha_{i+1}}+{\ep}{\ol{\genF}}_{\alpha_{i+1}}{\ol{\genF}}_{\alpha_i}={\ep}{{\genF}}_{\alpha_i}{{\genF}}_{\alpha_i}-{{\genF}}_{\alpha_i}{{\genF}}_{\alpha_{i+1}}.
    \end{align*}
\end{prop}

By comparing the generating relations in Definition \ref{defqn} with those in Proposition \ref{rootof1}, we immediately conclude that the restricted and non-restricted integral forms are non-isomorphic when $v$ is a root of unity.\\

\textbf{Acknowledgements.}  The authors would like to thank Jie Du for helpful suggestions. This work was partially supported by the National Natural Science Foundation of China (Grants Nos. 12371040 and 12131018).\\


\end{document}